\documentclass[10pt,twocolumn,letterpaper]{article}

\usepackage{iccv}
\usepackage{times}
\usepackage{epsfig}
\usepackage{graphicx}
\usepackage{amsmath}
\usepackage{amsthm}
\usepackage{amssymb}
\usepackage{prettyref}
\usepackage{enumitem}

\usepackage[pagebackref=true,breaklinks=true,letterpaper=true,colorlinks,bookmarks=false]{hyperref}

\iccvfinalcopy 

\def\iccvPaperID{3476} 

\ificcvfinal\fi

\usepackage{amsfonts}
\usepackage{amsmath}
\usepackage{amssymb}

\usepackage{graphicx,url}

\usepackage{xspace}
\usepackage{bm}
\usepackage[T1]{fontenc}
\usepackage{latexsym}
\usepackage{xstring}
\usepackage{relsize}

\usepackage[noend]{algorithmic}
\usepackage{multirow}
\usepackage{xcolor} 

\newtheorem{theorem}{Theorem}
\newtheorem{problem}{Problem}

\newtheorem{lemma}[theorem]{Lemma}

\newtheorem{proposition}[theorem]{Proposition}
\newtheorem{remark}[theorem]{Remark}

\newcommand{\cf}{\emph{cf.}\xspace}

\newcommand{\bdmath}{\begin{dmath}}
\newcommand{\edmath}{\end{dmath}}
\newcommand{\beq}{\begin{equation}}
\newcommand{\eeq}{\end{equation}}
\newcommand{\bdm}{\begin{displaymath}}
\newcommand{\edm}{\end{displaymath}}
\newcommand{\bea}{\begin{eqnarray}}
\newcommand{\eea}{\end{eqnarray}}
\newcommand{\beal}{\beq \begin{array}{ll}}
\newcommand{\eeal}{\end{array} \eeq}
\newcommand{\beas}{\begin{eqnarray*}}
\newcommand{\eeas}{\end{eqnarray*}}
\newcommand{\ba}{\begin{array}}
\newcommand{\ea}{\end{array}}
\newcommand{\bit}{\begin{itemize}}
\newcommand{\eit}{\end{itemize}}
\newcommand{\ben}{\begin{enumerate}}
\newcommand{\een}{\end{enumerate}}
\newcommand{\expl}[1]{&&\qquad\text{\color{gray}(#1)}\nonumber}

\newcommand{\calA}{{\cal A}}
\newcommand{\calB}{{\cal B}}

\newcommand{\calL}{{\cal L}}

\newcommand{\calN}{{\cal N}}

\newcommand{\calS}{{\cal S}}

\newcommand{\etal}{\emph{et~al.}\xspace}
\newcommand{\setal}{~\emph{et~al.}\xspace}

\newcommand{\M}[1]{{\bm #1}} 

\newcommand{\hide}[1]{}

\newcommand{\hiddenText}{{\color{gray} hidden text.}}
\newcommand{\hideWithText}[1]{\hiddenText}

\newcommand{\kron}{\otimes}

\newcommand{\subject}{\text{ subject to }}

\newcommand{\tran}{^{\mathsf{T}}}

\newcommand{\trace}[1]{\mathrm{tr}\left(#1\right)}

\newcommand{\rank}[1]{\mathrm{rank}\left(#1\right)}

\newcommand{\inv}{^{-1}}

\newcommand{\zero}{{\mathbf 0}}
\newcommand{\eye}{{\mathbf I}}

\newcommand{\Real}[1]{ { {\mathbb R}^{#1} } }

\newcommand{\SOthree}{\ensuremath{\mathrm{SO}(3)}\xspace}

\newcommand{\MA}{\M{A}}

\newcommand{\MD}{\M{D}}
\newcommand{\ME}{\M{E}}
\newcommand{\MJ}{\M{J}}
\newcommand{\MK}{\M{K}}

\newcommand{\MM}{\M{M}}
\newcommand{\MN}{\M{N}}

\newcommand{\MQ}{\M{Q}}

\newcommand{\MR}{\M{R}}

\newcommand{\MH}{\M{H}}

\newcommand{\MZ}{\M{Z}}
 
\newcommand{\MOmega}{\M{\Omega}}

\newcommand{\MLambda}{\M{\Lambda}}

\newcommand{\va}{\boldsymbol{a}} 
 
\newcommand{\vb}{\boldsymbol{b}}

\newcommand{\ve}{\boldsymbol{e}}

\newcommand{\vo}{\boldsymbol{o}}

\newcommand{\vq}{\boldsymbol{q}}
\newcommand{\vr}{\boldsymbol{r}}

\newcommand{\vu}{\boldsymbol{u}}
\newcommand{\vv}{\boldsymbol{v}}

\newcommand{\vxx}{\boldsymbol{x}} 
\newcommand{\vy}{\boldsymbol{y}}

\newcommand{\valpha}{\boldsymbol{\alpha}}

\newcommand{\vepsilon}{\boldsymbol{\epsilon}}

\newcommand{\scenario}[1]{{\smaller \sf#1}\xspace}

\newcommand{\sphere}{\scenario{sphere}}

\newcommand{\cvx}{{\sf cvx}\xspace}

\newcommand{\blue}[1]{{\color{blue}#1}}

\newcommand{\red}[1]{{\color{red}#1}}

\newcommand{\linkToPdf}[1]{\href{#1}{\blue{(pdf)}}}
\newcommand{\linkToPpt}[1]{\href{#1}{\blue{(ppt)}}}
\newcommand{\linkToCode}[1]{\href{#1}{\blue{(code)}}}
\newcommand{\linkToWeb}[1]{\href{#1}{\blue{(web)}}}
\newcommand{\linkToVideo}[1]{\href{#1}{\blue{(video)}}}
\newcommand{\award}[1]{\xspace} 

\newcommand{\maxOutliers}{{95\%}\xspace}

\newcommand{\myplus}{\red{{\bf+}}}

\newcommand{\betaNoise}{\sigma}

\renewcommand{\sphere}[1]{\calS^{#1}}

\newcommand{\TLS}{\scenario{TLS}}

\newcommand{\bnb}{BnB\xspace}

\newcommand{\TWlong}{Robust Wahba\xspace}

\newcommand{\MZero}{\M{0}}
\newcommand{\QCQP}{QCQP\xspace}

\newcommand{\barMQ}{\bar{\MQ}}

\newcommand{\sumAllPointsi}{\sum_{i=1}^{N}}

\newcommand{\barc}{\bar{c}}
\newcommand{\barcsq}{\barc^2}

\renewcommand{\eg}{\emph{e.g.},~}

\newcommand{\myParagraph}[1]{{\bf #1.}}

\newcommand{\supp}{Supplementary Material\xspace}

\newcommand{\name}{\scenario{QUASAR}}
\newcommand{\nameLong}{QUAternion-based Semidefinite relAxation for Robust alignment\xspace}

\newcommand{\FGR}{\scenario{FGR}}
\newcommand{\GORE}{\scenario{GORE}} 
 
\newcommand{\ransac}{\scenario{RANSAC}}

\newcommand{\SURF}{\scenario{SURF}} 
\newcommand{\mosek}{\scenario{MOSEK}} 
\newcommand{\wahba}{\scenario{Wahba}} 
\newcommand{\sdpnalplus}{\scenario{SDPNAL+}} 

\newcommand{\SIFT}{\scenario{SIFT}} 
\newcommand{\ORB}{\scenario{ORB}} 
\newcommand{\FPFH}{\scenario{FPFH}} 

\newcommand{\bunny}{\scenario{Bunny}}

\newcommand{\BnBLtwo}{\scenario{BnB-L2}}
\newcommand{\BnBAng}{\scenario{BnB-Ang}}

\newcommand{\bmat}{\left[ \begin{array}}
\newcommand{\emat}{\end{array} \right]}

\newcommand{\hatva}{\hat{\va}}
\newcommand{\hatvb}{\hat{\vb}}

\newcommand{\vxstar}{\vxx^\star}
\newcommand{\barM}{\bar{\MM}}
\newcommand{\barZ}{\bar{\MZ}}

\newcommand{\vSkew}[1]{[#1]_{\times}}
\newcommand{\barLambda}{\bar{\MLambda}}

\newcommand{\barvu}{\bar{\vu}}

\newcommand{\sym}{\text{Sym}}
\newcommand{\MTheta}{\boldsymbol{\Theta}}

\newcommand{\edit}[1]{#1}

\graphicspath{{./}}

\renewcommand{\expl}[1]{\text{\small\color{gray}(#1)}\nonumber\\}

\title{\vspace{-30mm} \small{\normalfont This paper has been accepted for publication at the International Conference on Computer Vision, 2019. \\ Please cite the paper as: H. Yang and L. Carlone, \\ ``A Quaternion-based Certifiably Optimal Solution to the Wahba Problem with Outliers'', \\ In \emph{International Conference on Computer Vision (ICCV)}, 2019.} \\ \vspace{10mm} \Large{A Quaternion-based Certifiably Optimal Solution \\ to the Wahba Problem with Outliers }} 

\author{Heng Yang and Luca Carlone \\
Laboratory for Information \& Decision Systems (LIDS) \\
Massachusetts Institute of Technology \\
{\tt\small \{hankyang, lcarlone\}@mit.edu }
}

\begin{document}

\maketitle



\begin{abstract}
   The \emph{Wahba problem}, also known as \emph{rotation search},  
   seeks to find the best rotation to align two sets of vector observations given putative correspondences, and is a fundamental routine in many computer vision and robotics applications.
    This work proposes the first polynomial-time \emph{certifiably optimal} approach for solving the Wahba problem when a large number of vector observations are outliers. 
   Our first contribution is to formulate the Wahba problem using a \emph{Truncated Least Squares} (\TLS) cost that is insensitive to a large fraction of spurious correspondences. 
   The second contribution is to rewrite the problem using unit quaternions and show that the \TLS cost can be framed 
as a Quadratically-Constrained Quadratic Program (\QCQP).
   %
   Since the resulting optimization is still highly non-convex and hard to solve globally, our third contribution is to develop a convex Semidefinite Programming (SDP) relaxation. We show that while a naive relaxation performs poorly in general, our relaxation is tight even in the presence of large noise and outliers. 
   %
   We validate the proposed algorithm, named \name (\nameLong), in both synthetic and real datasets showing that the algorithm outperforms \ransac, robust local optimization techniques, \edit{global outlier-removal procedures, and Branch-and-Bound methods}. \name is able to compute certifiably optimal solutions (\ie the relaxation is exact) even in the case when \maxOutliers of the correspondences are outliers.
\end{abstract}


\section{Introduction}
\label{sec:intro}
\begin{figure}[t]
	\begin{center}
	\begin{minipage}{\columnwidth}
			\centering%
			\includegraphics[width=0.95\columnwidth]{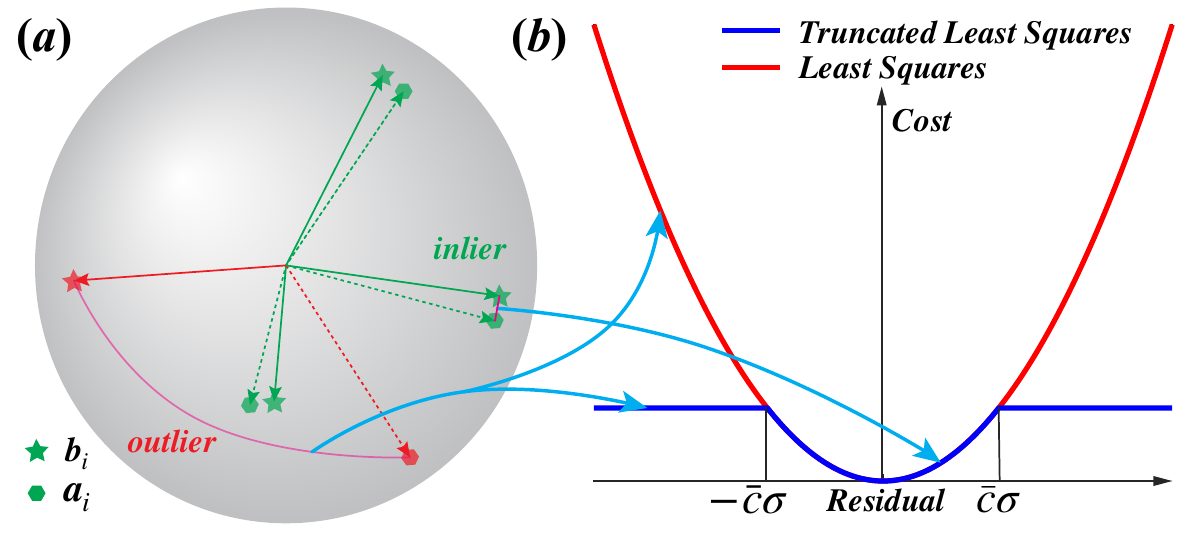}
	\end{minipage}
	\caption{We propose \name (\nameLong), a certifiably optimal solution to the Wahba problem 
	with outliers.
	(a) Wahba problem with four vector observations (three inliers and a single outlier). 
	(b) Contrary to standard least squares formulations, \name uses a truncated least squares cost that assigns a constant cost to 
	measurements with large residuals, hence
	being insensitive to outliers.
	 \label{fig:sphereandTLS}}
	\end{center}
	\vspace{-8mm}
\end{figure}

The \emph{Wahba problem}~\cite{wahba1965siam-wahbaProblem,Chin2018arxiv-starTrackEvent}, also known as \emph{rotation search}~\cite{Bustos2015iccv-gore3D,Hartley09ijcv-globalRotationRegistration,Bazin14eccv-robustRelRot}, is a fundamental problem in computer vision, 
robotics, and aerospace engineering and consists in finding the rotation between two coordinate frames given vector observations taken in the two frames.
%
The problem finds extensive applications in point cloud registration~\cite{Besl92pami,Yang16pami-goicp}, image stitching~\cite{Bazin14eccv-robustRelRot}, motion estimation and 3D reconstruction~\cite{Blais95pami-registration,Choi15cvpr-robustReconstruction,Makadia06cvpr-registration}, and satellite attitude determination~\cite{wahba1965siam-wahbaProblem,Chin2018arxiv-starTrackEvent,cheng2019aiaaScitech-totalLeastSquares}, to name a few.

Given two sets of vectors $\va_i, \vb_i \in \Real{3}$, $i=1,\ldots,N$,
the Wahba problem is formulated as a least squares problem 
%
\bea \label{eq:wahbaProblem}
\min_{\MR \in \SOthree} & \sumAllPointsi w_i^2 \Vert \vb_i - \MR \va_i \Vert^2 
\eea
which computes the best rotation $\MR$ that aligns vectors $\va_i$ and $\vb_i$, 
and where $\{w_i^2\}_{i=1}^N$ are (known) weights associated to each pair of measurements.
Here $\SOthree \doteq \{\MR \in \Real{3\times3}: \MR\tran \MR \!=\!\MR \MR\tran\!=\!\eye_3, \det(\MR)=1 \}$ is the 3D \emph{Special Orthogonal Group} containing proper 3D rotation matrices and $\eye_d$ denotes the identity matrix of size $d$. 
Problem~\eqref{eq:wahbaProblem} is known to be a  maximum likelihood estimator for the unknown rotation when the ground-truth correspondences $(\va_i, \vb_i)$ are known and the observations are corrupted with zero-mean isotropic Gaussian noise~\cite{Shuster89jas}. 
In other words,~\eqref{eq:wahbaProblem} computes an accurate estimate for $\MR$ when the 
observations can be written as $\vb_i = \MR \va_i  + \vepsilon_i$ ($i=1,\ldots,N$), where $\vepsilon_i$ is isotropic Gaussian noise. 
Moreover, problem~\eqref{eq:wahbaProblem} can be solved in closed form~\cite{markley2014book-fundamentalsAttitudeDetermine,Horn87josa,markley1988jas-svdAttitudeDeter,Arun87pami,Schonemann66psycho,horn1988josa-rotmatsol}.

 Unfortunately, in practical application, many of the vector observations may be \emph{outliers}, typically due 
 to incorrect vector-to-vector correspondences, \cf Fig.~\ref{fig:sphereandTLS}(a). 
 In computer vision, correspondences are established through 
 2D (\eg \SIFT~\cite{Lowe04ijcv}, \ORB~\cite{rublee2011iccv-orb}) or 3D (\eg \FPFH~{\cite{Rusu09icra-fast3Dkeypoints}) feature matching techniques, which are prone to produce many outlier correspondences. For instance, it is not uncommon to observe 
 95\% outliers when using \FPFH for point cloud registration~\cite{Bustos18pami-GORE}.
 In the presence of outliers, the standard Wahba problem~\eqref{eq:wahbaProblem} is no longer a maximum likelihood estimator and the resulting estimates are affected by large errors~\cite{Bustos18pami-GORE}. 
 A common approach to gain robustness against outliers is to incorporate solvers for problem~\eqref{eq:wahbaProblem} in a \ransac scheme~\cite{Fischler81}.
However, \ransac's runtime grows exponentially with the outlier ratio~\cite{Bustos18pami-GORE} and its performance, as we will see in Section \ref{sec:experiments}, quickly degrades in the presence of noise and high outlier ratios.  

This paper is motivated by the goal of designing an approach that (i) can solve the Wahba problem globally (without an initial guess), (ii) can tolerate large noise (\eg the magnitude of the noise is 10\% of the magnitude of the measurement vector) and extreme amount of outliers (\eg over 95\% of the observations are outliers), (iii) runs in polynomial time, and (iv) provides certifiably optimal solutions. The related literature, reviewed in Section~\ref{sec:relatedWork}, fails to simultaneously meet these goals, and only includes algorithms that are robust to moderate amounts of outlier, or 
are robust to extreme amounts of outliers (\eg 90\%) but only provide sub-optimal solutions (\eg \GORE~\cite{Bustos2015iccv-gore3D}), or that are globally optimal but run in exponential time in the worst case, such as \emph{branch-and-bound} (\bnb) methods (\eg \cite{Hartley09ijcv-globalRotationRegistration,bazin2012accv-globalRotSearch}).

\myParagraph{Contribution} Our first contribution, presented in Section~\ref{sec:TLSFormulation}, is to reformulate the Wahba problem~\eqref{eq:wahbaProblem} using a \emph{Truncated Least Squares} (\TLS) cost that is robust against a large fraction of outliers. We name the resulting optimization problem the (outlier-)\emph{\TWlong} problem. 

The second contribution (Section~\ref{sec:binaryClone}) is to depart from the rotation matrix representation and rewrite the \edit{\TWlong} problem using unit quaternions.
 In addition, we show how to rewrite the 
 \TLS cost function by using
  additional binary variables that decide whether a measurement is an inlier or outlier. 
  Finally, we prove that the mixed-integer program (including a quaternion and $N$ binary variables) can be rewritten as an optimization involving $N+1$ unit quaternions. This sequence of re-parametrizations, that we call \emph{binary cloning}, leads to a non-convex Quadratically-Constrained Quadratic Program (\QCQP).

The third contribution (Section \ref{sec:SDPRelaxation}) is to provide a polynomial-time certifiably optimal solver for the \QCQP.
Since the \QCQP is highly non-convex and hard to solve globally,  we propose a  
\emph{Semidefinite Programming (SDP) relaxation} that is empirically tight.
We show that while a naive SDP relaxation is not tight and performs poorly in practice, the proposed method remains tight even when observing \maxOutliers outliers.
Our approach is \emph{certifiably optimal}~\cite{Bandeira16crm} in the sense that it provides a way to check optimality 
of the resulting solution, and computes optimal solutions in practical problems.
To the best of our knowledge, this is the first polynomial-time method that solves the robust Wahba problem with certifiable optimality guarantees. 

We validate the proposed algorithm, named \name (\nameLong), in both synthetic and real datasets 
for point cloud registration and image stitching, 
showing that the algorithm outperforms \ransac, robust local optimization techniques, \edit{ outlier-removal procedures, and \bnb methods}. 



\section{Related Work}
\label{sec:relatedWork}
\subsection{Wahba problem without outliers}
The Wahba problem was first proposed by Grace Wahba in 1965 with the goal of estimating satellite attitude given vector observations~\cite{wahba1965siam-wahbaProblem}. 
In aerospace, the vector observations are typically the directions to visible stars observed by sensors onboard the satellite.
The Wahba problem~\eqref{eq:wahbaProblem} is related to the well-known 
\emph{Orthogonal Procrustes} problem~\cite{Gower05oss-procrustes} where one searches for 
orthogonal matrices (rather than rotations) and all the weights $w_i$ are set to be equal to one.
Schonemann~\cite{Schonemann66psycho} provides a closed-form solution for the Orthogonal Procrustes problem using
 singular value decomposition. 
 Subsequent research effort across multiple communities led to the derivation of 
 closed-form solutions for the Wahba problem~\eqref{eq:wahbaProblem} using both quaternion~\cite{Horn87josa,markley2014book-fundamentalsAttitudeDetermine} and rotation matrix~\cite{horn1988josa-rotmatsol,markley1988jas-svdAttitudeDeter,Arun87pami,Forbes15jgcd-wahba,Saunderson15siam,Khoshelham16jprs} representations. 

The computer vision community has investigated the Wahba problem in the context of  
 point cloud registration~\cite{Besl92pami,Yang16pami-goicp}, image stitching~\cite{Bazin14eccv-robustRelRot}, motion estimation and 3D reconstruction~\cite{Blais95pami-registration,Choi15cvpr-robustReconstruction}. 
 In particular, the closed-form solutions from  
 Horn~\cite{Horn87josa} and Arun\setal~\cite{Arun87pami} 
 are now commonly used in point cloud registration techniques~\cite{Besl92pami}. 
 While the formulation~\eqref{eq:wahbaProblem} implicitly assumes zero-mean isotropic Gaussian noise, 
 several authors investigate a generalized version of the Wahba problem~\eqref{eq:wahbaProblem}, where the noise follows an anisotropic Gaussian~\cite{Briales17cvpr-registration,cheng2019aiaaScitech-totalLeastSquares,Crassidis07-survey}. 
 Cheng and Crassidis~\cite{cheng2019aiaaScitech-totalLeastSquares} develop a local iterative optimization algorithm. Briales and Gonzalez-Jimenez~\cite{Briales17cvpr-registration} propose a convex relaxation to compute global solutions for the anisotropic case. 
 Ahmed\setal~\cite{Ahmed12tsp-wahba} develop an SDP relaxation for the case with bounded noise and no outliers.
 All these methods 
 are known to perform poorly in the presence of outliers.

\subsection{Wahba problem with outliers} 

In computer vision applications, the vector-to-vector correspondences are typically established through descriptor matching, which may lead to spurious correspondences and outliers~\cite{Bustos2015iccv-gore3D,Bustos18pami-GORE}. This observation triggered research into outlier-robust variants of the Wahba problem.

\myParagraph{Local Methods} The most widely used method for handling outliers is \ransac, which is efficient and accurate in the low-noise and low-outlier regime~\cite{Fischler81,Meer91ijcv-robustVision}. However, \ransac is non-deterministic (different runs give different solutions), sub-optimal (the solution may not capture all the inlier measurements), and its performance quickly deteriorates in the large-noise and high-outlier regime. Other approaches resort to \emph{M-estimators}, which replace the least squares cost in eq.~\eqref{eq:wahbaProblem} with robust cost functions that are less sensitive to outliers~\cite{Black96ijcv-unification,MacTavish15crv-robustEstimation}. Zhou \etal~\cite{Zhou16eccv-fastGlobalRegistration} propose \emph{Fast Global Registration} (\FGR) that employs the Geman-McClure robust cost function and solves the resulting optimization iteratively using a continuation method. 
Since the problem becomes more and more non-convex at each iteration, \FGR does not guarantee global optimality in general. In fact, as we show in Section \ref{sec:experiments}, \FGR tends to fail when the outlier ratio is above 70\% of the observations. 

\myParagraph{Global Methods} The most popular class of global methods for robust rotation search is based on \emph{Consensus Maximization}~\cite{chin2017slcv-maximumconsensusadvances} and \emph{branch-and-bound} (\bnb)~\cite{campbell2017iccv-2D3Dposeestimation}. Hartley and Kahl~\cite{Hartley09ijcv-globalRotationRegistration} first proposed using \bnb for rotation search, and Bazin \etal~\cite{bazin2012accv-globalRotSearch} adopted consensus maximization to extend their \bnb algorithm with a robust formulation. \bnb is guaranteed to return the globally optimal solution, but it runs in exponential time in the worst case. 
Another class of global methods for consensus maximization enumerates all possible subsets of measurements with size no larger than the problem dimension (3 for rotation search) to analytically compute candidate solutions, and then verify global optimality using computational geometry~\cite{olsson2008cvpr-polyRegOutlier,enqvist2012eccv-robustfitting}. These methods 
still require exhaustive enumeration.  

\myParagraph{Outlier-removal Methods} Recently, Para and Chin~\cite{Bustos2015iccv-gore3D} proposed a \emph{guaranteed outlier removal} (\GORE) algorithm that removes gross outliers while ensuring that all inliers are preserved. 
Using \GORE as a preprocessing step for \bnb is shown to boost the speed of \bnb, while using \GORE alone still provides a reasonable sub-optimal solution. Besides rotation search, \GORE has also been successfully applied to other geometric vision problems such as triangulation~\cite{Chin16cvpr-outlierRejection}, and registration~\cite{Bustos2015iccv-gore3D,Bustos18pami-GORE}. Nevertheless, 
 the existing literature is still missing a polynomial-time algorithm that can
simultaneously tolerate extreme amounts of outliers and return globally optimal solutions.

\section{Problem Formulation: Robust Wahba}
\label{sec:TLSFormulation}

Let $\calA=\{\va_i\}_{i=1}^N$ and $\calB=\{\vb_i\}_{i=1}^N$ be two sets of 3D vectors ($\va_i,\vb_i \in \Real{3}$), such that, given 
$\calA$, the vectors in $\calB$ are described by the following generative model: 
%
\bea \label{eq:generativeModel}
\begin{cases}
\vb_i=\MR \va_i + \vepsilon_i & \text{if }\vb_i \text{ is an inlier, or} \\
\vb_i = \vo_i & \text{if }\vb_i \text{ is an outlier}
\end{cases}
\eea
where $\MR \in \SOthree$ is an (unknown, to-be-estimated) rotation matrix, 
 $\vepsilon_i \in \Real{3}$ models the inlier measurement noise, and $\vo_i \in \Real{3}$ is an arbitrary vector. In other words, if $\vb_i$ is an inlier, then it must be a rotated version of $\va_i$ plus noise, while if $\vb_i$ is an outlier, then $\vb_i$ is  arbitrary. 
 In the special case where all $\vb_i$'s are inliers and the noise obeys a zero-mean isotropic Gaussian distribution, \ie $\vepsilon_i \sim \calN(\zero_3, \sigma_i^2 \eye_3)$, then the Maximum Likelihood estimator of $\MR$ takes the form of eq.~\eqref{eq:wahbaProblem}, with the weights chosen as the inverse of the measurement variances, $w_i^2=1/\sigma_i^2$. In this paper, 
 we are interested in the case where measurements include outliers.

\subsection{The Robust Wahba Problem} 

We introduce a novel formulation for the (outlier-)\emph{\TWlong} problem, 
that uses a \emph{truncated least squares} (\TLS) cost function: 
%
\bea \label{eq:TLSWahbaRot} 
\min_{\MR \in \SOthree}  \sumAllPointsi \min \left( \frac{1}{\betaNoise_i^2} \| \vb_i - \MR \va_i \|^2 \; , \; \barcsq \right)
\eea
where the inner ``$\min(\cdot,\cdot)$'' returns the minimum between two scalars.\edit{~The \TLS cost has been recently shown to be robust against high outlier rates in pose graph optimization~\cite{Lajoie2019RAL-DCGM}}.~Problem~\eqref{eq:TLSWahbaRot} computes a least squares solution for measurements with small residuals, \ie when $\frac{1}{\betaNoise_i^2} \| \vb_i - \MR \va_i \|^2 \leq \barcsq$ (or, equivalently, $\| \vb_i - \MR \va_i \|^2 \leq \betaNoise_i^2 \barcsq$), while discarding measurements with large residuals (when $\frac{1}{\betaNoise_i^2} \| \vb_i - \MR \va_i \|^2 > \barcsq$ the $i$-th term becomes a constant $\barcsq$ and has no effect on the optimization). 
Problem~\eqref{eq:TLSWahbaRot} implements the \TLS cost function illustrated in Fig.~\ref{fig:sphereandTLS}(b). 
The \edit{parameters} $\betaNoise_i$ and $\barcsq$ are fairly easy to set in practice, as discussed in the following remarks.

\begin{remark}[Probabilistic choice of $\betaNoise_i$ and $\barcsq$]\label{rmk:chi2}
Let us assume that the inliers follow the generative model~\eqref{eq:generativeModel} with  
$\vepsilon_i \sim \calN(\MZero_3, \sigma_i^2\eye_3)$; we will not make assumptions on the generative model for the outliers, which is unknown in practice. Since the noise on the inliers is Gaussian, it holds:
\bea
\frac{1}{\betaNoise_i^2} \| \vb_i - \MR \va_i \|^2 = \frac{1}{\betaNoise_i^2} \| \vepsilon_i \|^2 
\sim \chi^2(3)
\eea
where $\chi^2(3)$ is the Chi-squared distribution with three degrees of freedom. 
Therefore, with desired probability $p$, the weighted error $\frac{1}{\betaNoise_i^2} \| \vepsilon_i \|^2$ for the inliers satisfies:
\bea \label{eq:choiceofp}
\mathbb{P} \left( \frac{\Vert \vepsilon_i \Vert^2}{\sigma_i^2} \leq \barcsq \right) = p,
\eea
where $\barcsq$ is the quantile of the $\chi^2$ distribution with three degrees of freedom and lower tail probability equal to $p$. 
Therefore, one can simply set the $\betaNoise_i$ in Problem~\eqref{eq:TLSWahbaRot} to be the standard deviation of the inlier noise, and compute $\barcsq$ from the $\chi^2(3)$ distribution for a desired probability $p$ (\eg $p=0.99$).
The constant $\barcsq$ monotonically increases with $p$; therefore, setting $p$ close to 1 makes the formulation~\eqref{eq:TLSWahbaRot} more prone to accept measurements with large residuals, while a small $p$ makes~\eqref{eq:TLSWahbaRot} more selective. 
\end{remark}

\begin{remark}[Set membership choice of $\betaNoise_i$ and $\barcsq$] 
Let us assume that the inliers follow the generative model~\eqref{eq:generativeModel} with  
$\| \vepsilon_i \| \leq \beta_i$, where $\beta_i$ is a given noise bound; this is the typical setup assumed in 
\emph{set membership estimation}~\cite{Milanese89chapter-ubb}. 
In this case, it is easy to see that the inliers satisfy: 
\bea
\| \vepsilon_i \| \leq \beta_i \iff \| \vb_i - \MR \va_i \|^2 \leq \beta_i^2
\eea
hence one can simply choose $\betaNoise_i^2 \barcsq = \beta_i^2$, 
Intuitively, the constant $\betaNoise_i^2 \barcsq$ is the largest (squared) residual error we are willing to tolerate on the $i$-th measurement. 
\end{remark}

\subsection{Quaternion formulation}

We now adopt a quaternion formulation for~\eqref{eq:TLSWahbaRot}. 
 Quaternions are an alternative representation for 3D rotations~\cite{Shuster93jas-attitude,Breckenridge99tr-quaternions} and their use will simplify the derivation of our 
 convex relaxation in Section~\ref{sec:SDPRelaxation}. We start by reviewing basic facts about quaternions and then state the quaternion-based \TWlong formulation in Problem~\ref{prob:qTW} below.

{\bf Preliminaries on Unit Quaternions.}
We denote a unit quaternion as a unit-norm column vector $\vq=[\vv\tran\ s]\tran$, where $\vv \in \Real{3}$ is the \emph{vector part} of the quaternion and the last element $s$ is the \emph{scalar part}. We  also use $\vq = [q_1\ q_2\ q_3\ q_4]\tran$ to denote the four entries of the quaternion. 

Each quaternion represents a 3D rotation and the composition of two rotations $\vq_a$ and $\vq_b$ can be 
computed using the \emph{quaternion product} $\vq_c = \vq_a \kron \vq_b$:
\bea \label{eq:quatProduct}
\vq_c = \vq_a \kron \vq_b = \MOmega_1(\vq_a) \vq_b = \MOmega_2(\vq_b) \vq_a,
\eea
where $\MOmega_1(\vq)$ and $\MOmega_2(\vq)$ are defined as follows:
\vspace{-3mm}
\begin{equation} \label{eq:Omega_1}
\arraycolsep=2pt\def\arraystretch{0.5}
\scriptscriptstyle \MOmega_1(\vq) \scriptscriptstyle = 
\scriptscriptstyle \left[ \scriptscriptstyle \begin{array}{cccc}
\scriptscriptstyle q_4 & \scriptscriptstyle -q_3 &  \scriptscriptstyle q_2 &  \scriptscriptstyle q_1 \\
\scriptscriptstyle q_3 & \scriptscriptstyle q_4 & \scriptscriptstyle -q_1 & \scriptscriptstyle q_2 \\
\scriptscriptstyle -q_2 & \scriptscriptstyle q_1 & \scriptscriptstyle q_4 & \scriptscriptstyle q_3 \\
\scriptscriptstyle -q_1 & \scriptscriptstyle -q_2 & \scriptscriptstyle -q_3 & \scriptscriptstyle q_4 
\scriptscriptstyle \end{array} \scriptscriptstyle \right],\quad 
\arraycolsep=2pt\def\arraystretch{0.5}
\scriptscriptstyle \MOmega_2(\vq) = \scriptscriptstyle
\bmat{cccc}
\scriptscriptstyle q_4 & \scriptscriptstyle q_3 & \scriptscriptstyle -q_2 & \scriptscriptstyle q_1 \\
\scriptscriptstyle -q_3 & \scriptscriptstyle q_4 & \scriptscriptstyle q_1 & \scriptscriptstyle q_2 \\
\scriptscriptstyle q_2 & \scriptscriptstyle -q_1 & \scriptscriptstyle q_4 & \scriptscriptstyle q_3 \\
\scriptscriptstyle -q_1 & \scriptscriptstyle -q_2 & \scriptscriptstyle -q_3 & \scriptscriptstyle q_4 
\emat.
\end{equation}
\vspace{-3mm}
The inverse of a quaternion $\vq=[\vv\tran\ s]\tran$ is defined as:
\bea
\label{eq:q_inverse}
\vq\inv = \bmat{c}
-\vv \\
s
\emat,
\eea
where one simply reverses the sign of the vector part. 

The rotation of a vector $\va \in \Real{3}$ can be expressed in terms of quaternion product. 
Formally, if $\MR$ is the (unique)  rotation matrix corresponding to a unit quaternion $\vq$, then:
\bea
\label{eq:q_pointRot}
\bmat{c}
\MR \va \\
0 
\emat = \vq \kron \hat{\va} \kron \vq\inv,
\eea
where $\hat{\va}=[\va\tran\ 0]\tran$ is the homogenization of $\va$, obtained by augmenting $\va$ with an extra entry equal to zero.

The set of unit quaternions, denoted as $\sphere{3}=\{\vq\in\Real{4}: \Vert \vq \Vert = 1\}$, is the \emph{3-Sphere} manifold.  $\sphere{3}$ is a \emph{double cover} of $\SOthree$ since $\vq$ and $-\vq$ represent the same rotation (intuitively, eq.~\eqref{eq:q_pointRot} implements the same rotation if we replace $\vq$ with $-\vq$, since the matrices in~\eqref{eq:Omega_1} are linear in $\vq$).

{\bf Quaternion-based Robust Wahba Problem.}
Equipped with the relations reviewed above, it is now easy to rewrite~\eqref{eq:TLSWahbaRot}  
using unit quaternions. 

\begin{problem}[Quaternion-based Robust Wahba]\label{prob:qTW}
 The robust Wahba problem~\eqref{eq:TLSWahbaRot} can be equivalently written as: 
\bea \label{eq:TLSWahba} 
\min_{\vq \in \sphere{3}}  \sumAllPointsi \min \left( \frac{1}{\betaNoise_i^2} \Vert \hatvb_i - \vq \otimes \hatva_i \otimes \vq\inv \Vert^2 \;,\; \barcsq \right),
\eea
where we defined 
$\hatva_i \doteq [\va_i\tran \; 0]\tran$ and $\hatvb_i \doteq [\vb_i\tran \; 0]\tran$, and 
$\otimes$ denotes the quaternion product.  
\end{problem}

The equivalence between~\eqref{eq:TLSWahba} and~\eqref{eq:TLSWahbaRot} can be easily understood from eq.~\eqref{eq:q_pointRot}.
The main advantage of using~\eqref{eq:TLSWahba} is that we replaced the set \SOthree with a simpler set, the set of unit-norm vectors $\sphere{3}$. 

\vspace{-2mm}
\section{Binary Cloning and \QCQP}
\label{sec:binaryClone}
\vspace{-2mm}

The goal of this section is to rewrite~\eqref{eq:TLSWahba} as a Quadratically-Constrained Quadratic Program (\QCQP). We do so in three steps: (i) we show that~\eqref{eq:TLSWahba} can be written using binary variables, (ii) we show that the problem with binary variables can be written using $N+1$ quaternions, and (iii) we manipulate the resulting problem to expose its quadratic nature. The derivation of the \QCQP will pave the way to our convex relaxation (Section~\ref{sec:SDPRelaxation}).

{\bf From Truncated Least Squares to Mixed-Integer Programming.}
Problem~\eqref{eq:TLSWahba} is hard to solve globally, due to the non-convexity of both the cost function and  the constraint ($\sphere{3}$ is a non-convex set). As a first re-parametrization, we expose the non-convexity of the cost by 
rewriting the \TLS cost using binary variables. 
Towards this goal, we rewrite the inner ``$\min$'' in~\eqref{eq:TLSWahba} using the following property, that holds for any pair of  scalars $x$ and $y$: 
\bea
\label{eq:minxy}
\min(x,y)=\min_{\theta \in \{+1,-1\}} \frac{1+\theta}{2}x + \frac{1-\theta}{2}y.
\eea
Eq.~\eqref{eq:minxy} can be verified to be true by inspection: the right-hand-side returns $x$ (with minimizer $\theta=+1$) 
if $x<y$, and $y$ (with minimizer $\theta=-1$) if $x>y$. This enables us to rewrite problem \eqref{eq:TLSWahba} as a mixed-integer program including the quaternion
$\vq$ and binary variables $\theta_i, \; i=1,\ldots,N$:
\bea \label{eq:TLSadditiveForm}
\min_{\substack{ \vq \in S^3 \\ \theta_i = \{\pm1\}} } \sumAllPointsi \frac{1+\theta_i}{2}\frac{\Vert \hatvb_i - \vq \kron \hatva_i \kron \vq\inv \Vert^2}{\betaNoise_i^2}  + \frac{1-\theta_i}{2}\barcsq.
\eea
The reformulation is related to the Black-Rangarajan duality 
between robust estimation and line processes~\cite{Black96ijcv-unification}: the \TLS cost is an extreme case of robust function that results in a binary line process.
Intuitively, the binary variables $\{\theta_i\}_{i=1}^N$ in problem~\eqref{eq:TLSadditiveForm} decide whether a given measurement $i$ is an inlier 
($\theta_i=+1$) or an outlier ($\theta_i=-1$).

{\bf From Mixed-Integer to Quaternions.}
Now we convert the mixed-integer program~\eqref{eq:TLSadditiveForm} to an optimization over $N+1$ quaternions. The intuition is that, if we define extra quaternions $\vq_i \doteq \theta_i \vq$, 
we can rewrite~\eqref{eq:TLSadditiveForm} as a function of $\vq$ and $\vq_i$ ($i=1,\ldots,N$). This is a key step towards getting a quadratic cost (Proposition~\ref{prop:qcqp}). 
The re-parametrization is formalized in the following proposition.
\vspace{-2mm}
\begin{proposition}[Binary cloning] \label{prop:binaryCloning}
The mixed-integer program~\eqref{eq:TLSadditiveForm} is equivalent (in the sense that they admit the same 
optimal solution $\vq$) to the following optimization
%
\bea
\label{eq:TLSadditiveForm2}
\min_{\substack{ \vq \in S^3 \\ \vq_i = \{\pm\vq\}} } & \hspace{-4mm}
\displaystyle\sumAllPointsi 
\displaystyle\frac{\Vert \hatvb_i \!-\! \vq \kron \hatva_i \kron \vq\inv 
+ \vq\tran\vq_i \hatvb_i \!-\! \vq \kron \hatva_i \kron \vq_i\inv \Vert^2  
}{4\betaNoise_i^2} \nonumber \hspace{-5mm}\\
&\displaystyle  + \frac{1-\vq\tran\vq_i}{2}\barcsq.   \hspace{-5mm}
\eea
%
which involves $N+1$ quaternions ($\vq$ and $\vq_i$, $i=1,\ldots,N$).
\end{proposition}
While a formal proof is given in the \supp, it is fairly easy to see that 
if $\vq_i = \{\pm\vq\}$, or equivalently, $\vq_i = \theta_i \vq$ with $\theta_i \in \{\pm1\}$, then 
$\vq_i\tran\vq = \theta_i$, and $\vq \kron \hatva_i \kron \vq_i\inv =  \theta_i ( \vq \kron \hatva_i \kron \vq\inv)$ which exposes the relation between~\eqref{eq:TLSadditiveForm}  and~\eqref{eq:TLSadditiveForm2}.  
We dubbed the re-parametrization~\eqref{eq:TLSadditiveForm2} \emph{binary cloning} since now we created a ``clone'' $\vq_i$ for each measurement, such that $\vq_i = \vq$ for inliers (recall that $\vq_i = \theta_i \vq$) and $\vq_i = -\vq$ for outliers.

{\bf From Quaternions to \QCQP.}
We conclude this section by showing that~\eqref{eq:TLSadditiveForm2} can be actually written as a \QCQP. 
This observation is non-trivial since~\eqref{eq:TLSadditiveForm2} has a \emph{quartic} cost and $\vq_i = \{\pm\vq\}$ 
is not in the form of a quadratic constraint. 
The re-formulation as a \QCQP is given in the following.
\vspace{-2mm}
\begin{proposition}[Binary Cloning as a \QCQP] \label{prop:qcqp}
Define a single column vector $\vxx = [\vq\tran\ \vq_1\tran\ \dots\ \vq_N\tran]\tran$ 
stacking all variables in Problem~\eqref{eq:TLSadditiveForm2}. 
Then, Problem~\eqref{eq:TLSadditiveForm2}  
is equivalent (in the sense that they admit the same 
optimal solution $\vq$) to the following Quadratically-Constrained Quadratic Program:
\beal \label{eq:TLSBinaryClone}
\displaystyle\min_{\vxx \in \Real{4(N+1)}} & \sum\limits_{i=1}^N \vxx\tran \MQ_i \vxx \\
\subject &  \vxx_q\tran \vxx_q = 1 \hspace{30mm} \\
& \vxx_{q_i} \vxx_{q_i}\tran = \vxx_{q} \vxx_{q}\tran, \forall i = 1,\dots,N 
\eeal
where $\MQ_i \in \Real{4(N+1)\times 4(N+1)}$ ($i=1,\ldots,N$) are known symmetric matrices that depend on the 3D vectors $\va_i$ and 
$\vb_i$ (the explicit expression is given in the \supp), and the notation $\vxx_q$ (resp. $\vxx_{q_i}$) denotes the 4D subvector of $\vxx$ corresponding to $\vq$ (resp. $\vq_i$).
\end{proposition}

A complete proof of Proposition~\ref{prop:qcqp} is given in the \supp.
 Intuitively, (i) we developed the squares in the cost function~\eqref{eq:TLSadditiveForm2}, (ii) we used the properties of unit quaternions (Section~\ref{sec:TLSFormulation}) to simplify the expression to a quadratic cost, and (iii) we adopted the 
 more compact notation afforded by the vector $\vxx$ to obtain~\eqref{eq:TLSBinaryClone}. 
\section{Semidefinite Relaxation}
\label{sec:SDPRelaxation}
\vspace{-2mm}
Problem~\eqref{eq:TLSBinaryClone} writes the \TWlong problem as a \QCQP. 
Problem~\eqref{eq:TLSBinaryClone} is still a non-convex problem (quadratic equality constraints are non-convex).
Here we develop a convex semidefinite programming (SDP) relaxation for problem~\eqref{eq:TLSBinaryClone}. 

The crux of the relaxation consists in rewriting problem~\eqref{eq:TLSBinaryClone}  
as a function of the following matrix:
\bea \label{eq:matrixZ}
\MZ = \vxx \vxx\tran = \bmat{cccc}
\vq\vq\tran & \vq \vq_1\tran & \cdots & \vq\vq_N\tran \\
\star & \vq_1\vq_1\tran & \cdots & \vq_1\vq_N\tran \\
\vdots & \vdots & \ddots & \vdots \\
\star & \star & \cdots & \vq_N \vq_N\tran 
\emat.
\eea
For this purpose we note that if we define $\MQ \doteq \sum_{i=1}^N \MQ_i$:
\beal
\sum\limits_{i=1}^N \vxx\tran \MQ_i \vxx = \vxx\tran \MQ \vxx = 
\trace{ \MQ \vxx \vxx\tran} = \trace{\MQ \MZ}
\eeal
and that $\vxx_{q_i} \vxx_{q_i}\tran = [\MZ]_{q_i q_i}$, where  
$[\MZ]_{q_i q_i}$ denotes the $4\times4$ diagonal block of $\MZ$ with row and column indices corresponding to $\vq_i$.
Since any matrix in the form $\MZ = \vxx \vxx\tran$ is a positive-semidefinite rank-1 matrix, we obtain:
\vspace{-2mm}
\begin{proposition}[Matrix Formulation of Binary Cloning] \label{prop:qcqpZ}
Problem~\eqref{eq:TLSBinaryClone}
is equivalent (in the sense that optimal solutions of a problem can be mapped to optimal solutions of the other) 
to the following non-convex program:
\bea \label{eq:qcqpZ}
\min_{\MZ \succeq 0} & \trace{\MQ \MZ} \nonumber \\
\subject & \trace{[\MZ]_{qq}} = 1 \nonumber \\
& [\MZ]_{q_i q_i} = [\MZ]_{qq}, \;\; \forall i=1,\dots,N \nonumber \\
& \rank{\MZ} = 1
\eea
\end{proposition}
\vspace{-2mm}
At this point it is straightforward to  develop a (standard) SDP relaxation 
by dropping the rank-1 constraint, which is the only source of non-convexity in~\eqref{eq:qcqpZ}.
\vspace{-2mm}
\begin{proposition}[Naive SDP Relaxation] \label{prop:naiveRelax}
The following SDP is a convex relaxation of Problem~\eqref{eq:qcqpZ}:
\bea \label{eq:naiveRelaxation}
\min_{\MZ \succeq 0}  & \trace{\MQ \MZ} \\
\subject  & \trace{[\MZ]_{qq}} = 1 \nonumber \\
& [\MZ]_{q_i q_i} = [\MZ]_{qq}, \forall i=1,\dots,N \nonumber
\eea
\end{proposition}

The following theorem proves that the naive SDP relaxation~\eqref{prop:naiveRelax} 
is tight in the absence of noise and outliers.

\begin{theorem}[Tightness in Noiseless and Outlier-free Wahba]\label{thm:strongDualityNoiseless}
When there is no noise and no outliers in the measurements, and there are at least two vector measurements ($\va_i$'s) that are \edit{not parallel} to each other, the SDP relaxation~\eqref{eq:naiveRelaxation} is always \emph{tight}, \ie:
\begin{enumerate}
\item the optimal cost of~\eqref{eq:naiveRelaxation} matches the optimal cost of the \QCQP~\eqref{eq:TLSBinaryClone},
\item the optimal solution $\MZ^\star$ of~\eqref{eq:naiveRelaxation} has rank 1, and 
\item $\MZ^\star$ can be written as $\MZ^\star = (\vxx^\star) (\vxx^\star)\tran$ where $\vxx^\star \doteq [(\vq^\star)\tran \; (\vq_1^\star)\tran \; \ldots \; (\vq_N^\star)\tran]$ is a global minimizer of the original non-convex problem~\eqref{eq:TLSBinaryClone}.
\end{enumerate}
\end{theorem}

A formal proof, based on Lagrangian duality theory, is given in the \supp. While Theorem~\ref{thm:strongDualityNoiseless} ensures that the naive SDP relaxation computes optimal solutions in noiseless and outlier-free problems, our original motivation was to solve problems with many outliers. 
One can still empirically assess tightness by solving the SDP and verifying if a rank-1 solution is obtained.
Unfortunately, the naive relaxation produces solutions with rank larger than 1 in the presence of outliers, \cf~Fig.~\ref{fig:naiveRelaxation}(a).
Even when the rank is larger than 1, one can \emph{round} the solution by computing a rank-1 approximation of 
$\MZ^\star$; however, we empirically observe that, whenever the rank is larger than 1, the 
rounded estimates exhibit large errors, as shown in Fig.~\ref{fig:naiveRelaxation}(b).

To address these issues, we propose to add \emph{redundant constraints} to \emph{tighten} (improve) the SDP relaxation, \edit{inspired by~\cite{Briales17cvpr-registration,Tron15rssws3D-dualityPGO3D}}.
The \edit{following} proposed relaxation
is tight (empirically satisfies the three claims of Theorem~\ref{thm:strongDualityNoiseless}) even in the presence of noise and \maxOutliers of outliers.
\begin{proposition}[SDP Relaxation with Redundant Constraints] \label{prop:relaxationRedundant}
The following SDP is a convex relaxation of~\eqref{eq:qcqpZ}:
\bea \label{eq:relaxationRedundant}
\min_{\MZ \succeq 0} & \trace{\MQ \MZ} \label{eq:SDPrelax} \\
\subject 
& \trace{[\MZ]_{qq}} = 1 \nonumber \\
& [\MZ]_{q_i q_i} = [\MZ]_{qq}, \forall i=1,\dots,N \nonumber \\
& [\MZ]_{q q_i} = [\MZ]\tran_{q q_i}, \forall i=1,\dots,N \nonumber \\
& [\MZ]_{q_i q_j} = [\MZ]\tran_{q_i q_j}, \forall 1\leq i < j \leq N \nonumber
\eea
and it is always tighter, \ie the optimal objective  of~\eqref{eq:SDPrelax} is always closer to the optimal objective of~\eqref{eq:TLSBinaryClone}, when compared to the naive relaxation~\eqref{eq:naiveRelaxation}.
\end{proposition}

We name the improved relaxation~\eqref{eq:SDPrelax} \name (\nameLong).
While our current theoretical results only guarantee  tightness  in the noiseless and outlier-free case (Theorem~\ref{thm:strongDualityNoiseless} also holds for \name, since it is always tighter than the naive relaxation), 
in the next section we empirically demonstrate the tightness of~\eqref{eq:SDPrelax} 
in the face of noise and extreme outlier rates, \cf~Fig.~\ref{fig:naiveRelaxation}.

\vspace{-1mm}
\section{Experiments}
\label{sec:experiments}

We evaluate \name in both synthetic and real datasets for 
point cloud registration and image stitching showing that 
(i) the proposed relaxation is tight even with extreme (\maxOutliers) outlier rates, and  
(ii) \name is more accurate and robust than state-of-the-art techniques for rotation search. 

We implemented \name in Matlab, using \cvx~\cite{CVXwebsite} to model the convex programs~\eqref{eq:naiveRelaxation} and~\eqref{eq:SDPrelax}, and used \mosek~\cite{mosek} as the SDP solver. The parameters in \name are set according to Remark~\ref{rmk:chi2} with $p = 1-10^{-4}$.
\vspace{-1mm}
\subsection{Evaluation on Synthetic Datasets}
\label{sec:exp_synthetic}
\vspace{-1mm}
\myParagraph{Comparison against the Naive SDP Relaxation} We first test the performance of the naive SDP relaxation~\eqref{eq:naiveRelaxation} and \name~\eqref{eq:SDPrelax} under zero noise and increasing outlier rates. 
In each test, we sample $N=40$ unit-norm vectors $\calA=\{\va_i\}_{i=1}^N$ uniformly at random. Then we apply a random rotation $\MR$ to $\calA$ according to~\eqref{eq:generativeModel} with no additive noise to get $\calB=\{\vb_i\}_{i=1}^N$. To generate outliers, we replace a fraction of $\vb_i$'s with random unit-norm vectors. 
Results are averaged over 40 Monte Carlo runs. Fig.~\ref{fig:naiveRelaxation}(a) shows the rank of the solution produced by the naive SDP relaxation~\eqref{eq:naiveRelaxation} and \name~\eqref{eq:SDPrelax}; recall that the relaxation is tight when the rank is 1. 
Both relaxations are tight when there are no outliers, which validates Theorem~\ref{thm:strongDualityNoiseless}. However, the performance of the naive relaxation quickly degrades as the measurements include more outliers: at $10-40\%$, the naive relaxation starts becoming loose and completely breaks when the outlier ratio is above 40\%. On the other hand, \name produces a 
certifiably optimal rank-1 solution even with 90\% outliers. Fig.~\ref{fig:naiveRelaxation}(b) confirms that a tighter relaxation translates into more accurate rotation estimates.


\newcommand{\mpw}{4.2cm}
\begin{figure}[t]
	\begin{center}
	\begin{minipage}{\columnwidth}
	\hspace{-0.2cm}
	\begin{tabular}{cc}%
		\hspace{-3mm}
			\begin{minipage}{\mpw}%
			\centering%
			\includegraphics[width=1.0\columnwidth]{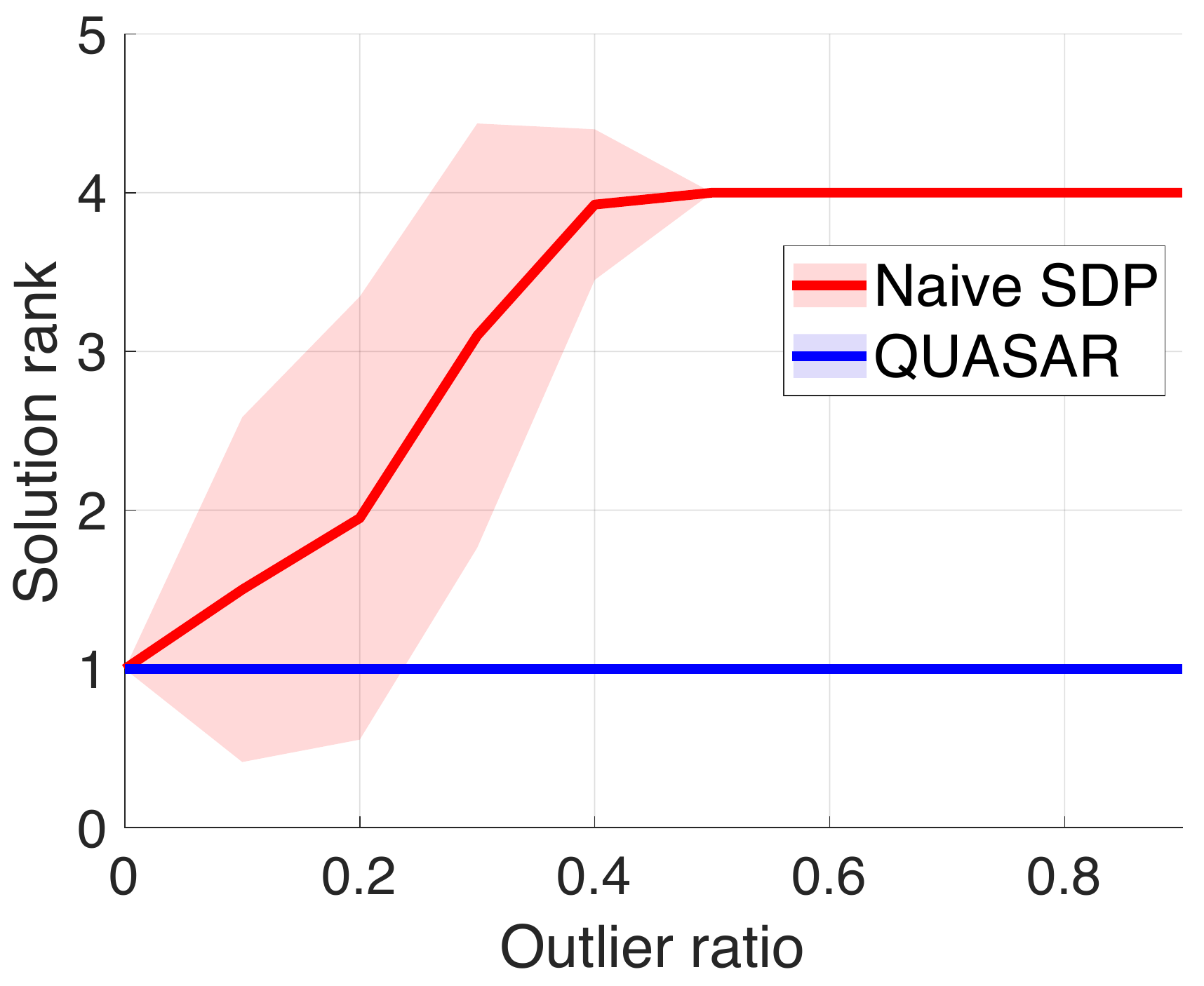}\\
			(a) Solution rank 
			\end{minipage}
		& \hspace{-3mm}
			\begin{minipage}{\mpw}%
			\centering%
			\includegraphics[width=1.0\columnwidth]{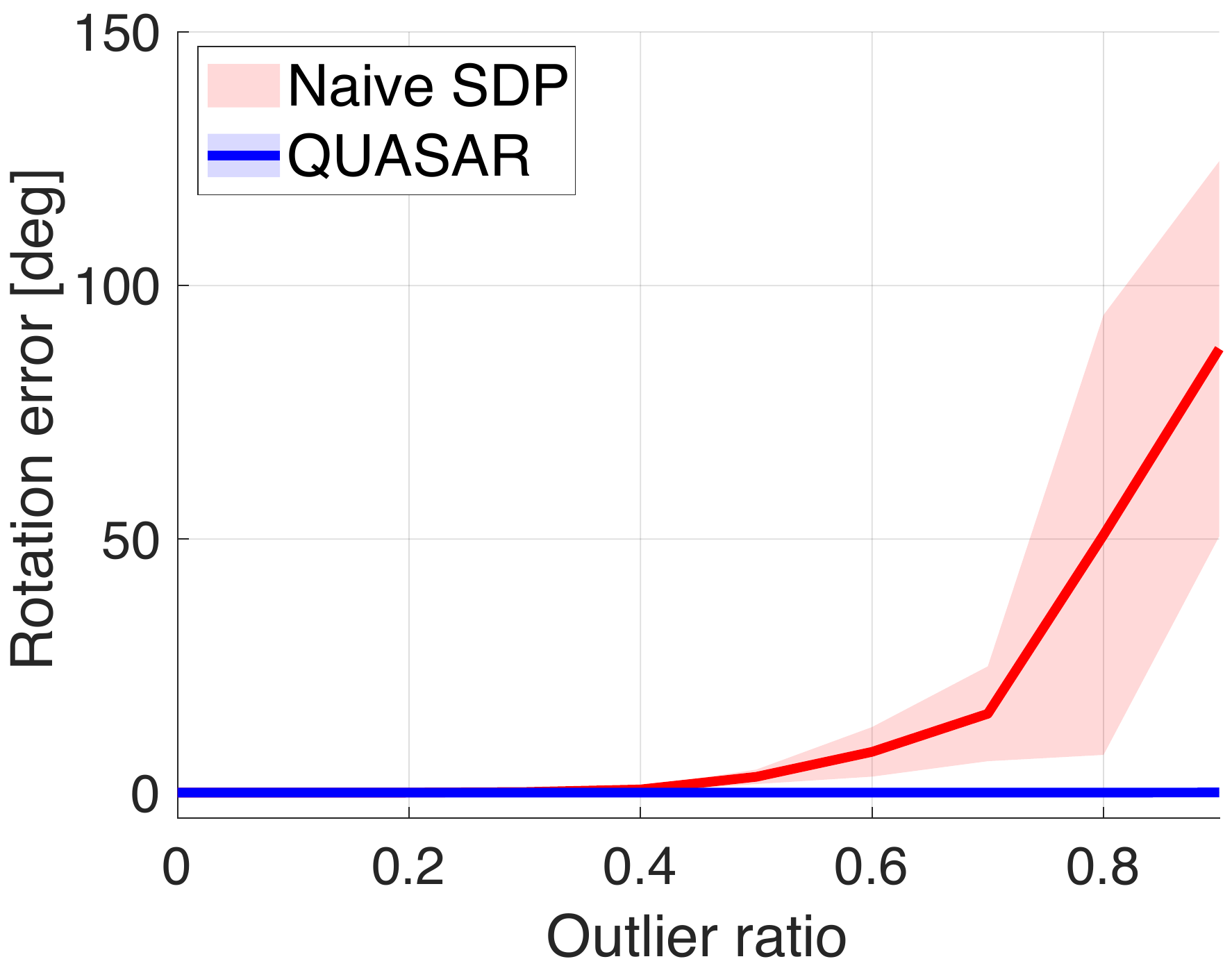}\\
			(b) Rotation errors
			\end{minipage}
		\end{tabular}
	\end{minipage}
	\begin{minipage}{\textwidth}
	\end{minipage}
	\caption{Comparison between the naive relaxation~\eqref{eq:naiveRelaxation} and the proposed relaxation~\eqref{eq:SDPrelax} for increasing outlier percentages. (a) Rank of solution (relaxation is tight when rank is 1); (b) rotation errors
	(solid line: mean; shaded area: 1-sigma standard deviation).
	\label{fig:naiveRelaxation}}
	\vspace{-12mm} 
	\end{center}
\end{figure}


\newcommand{\mpwthree}{6cm}
\newcommand{\myhspace}{\hspace{-2mm}}

\begin{figure*}[h]
	\begin{center}
	\begin{minipage}{\textwidth}
	\hspace{-0.7cm}
	\vspace{-2mm}
	\begin{tabular}{ccc}%
			\begin{minipage}{\mpwthree}%
			\centering%
			\includegraphics[width=\columnwidth]{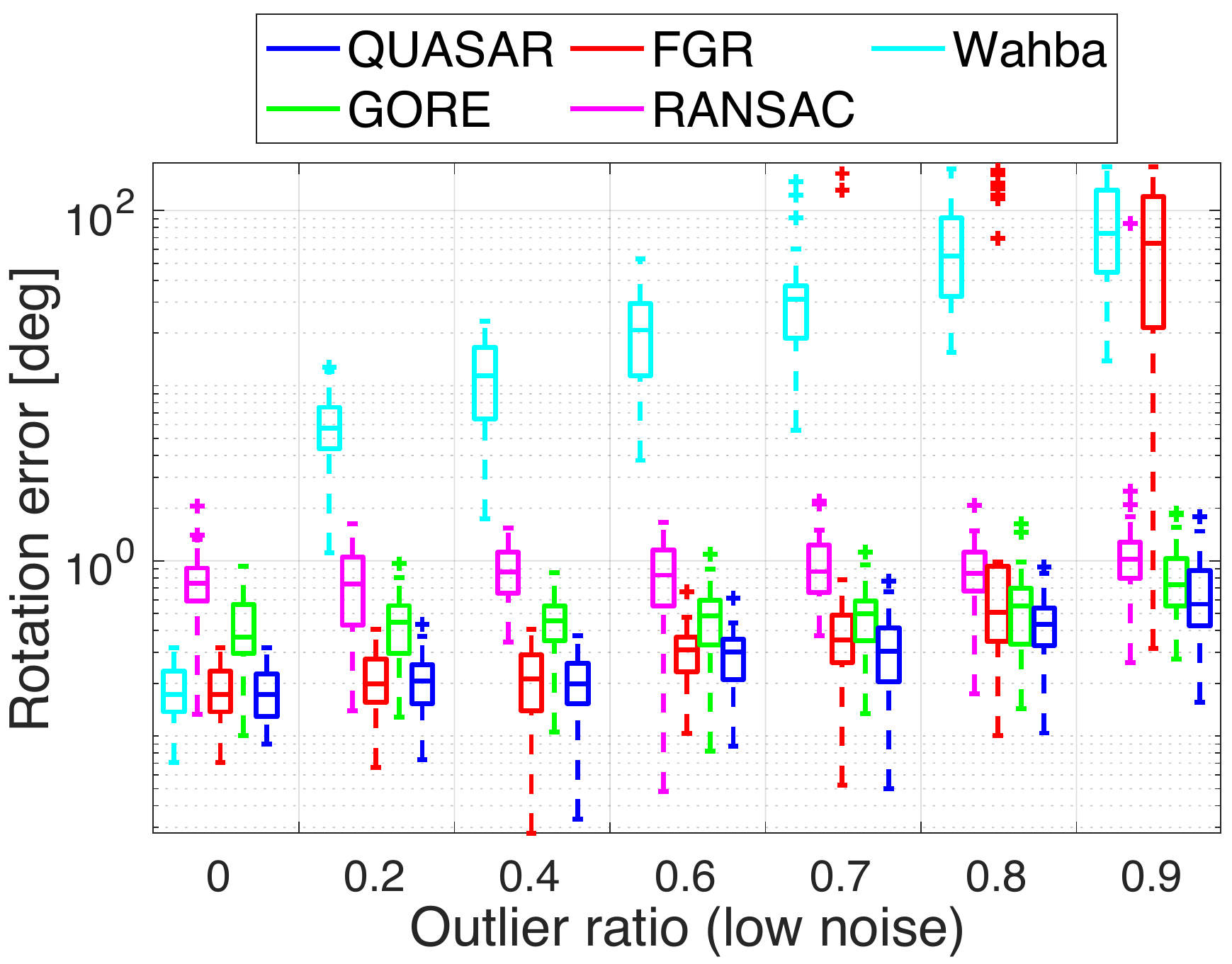} \\
			\end{minipage}
			\vspace{-3mm} 
		& \myhspace
			\begin{minipage}{\mpwthree}%
			\centering%
			\includegraphics[width=0.9\columnwidth]{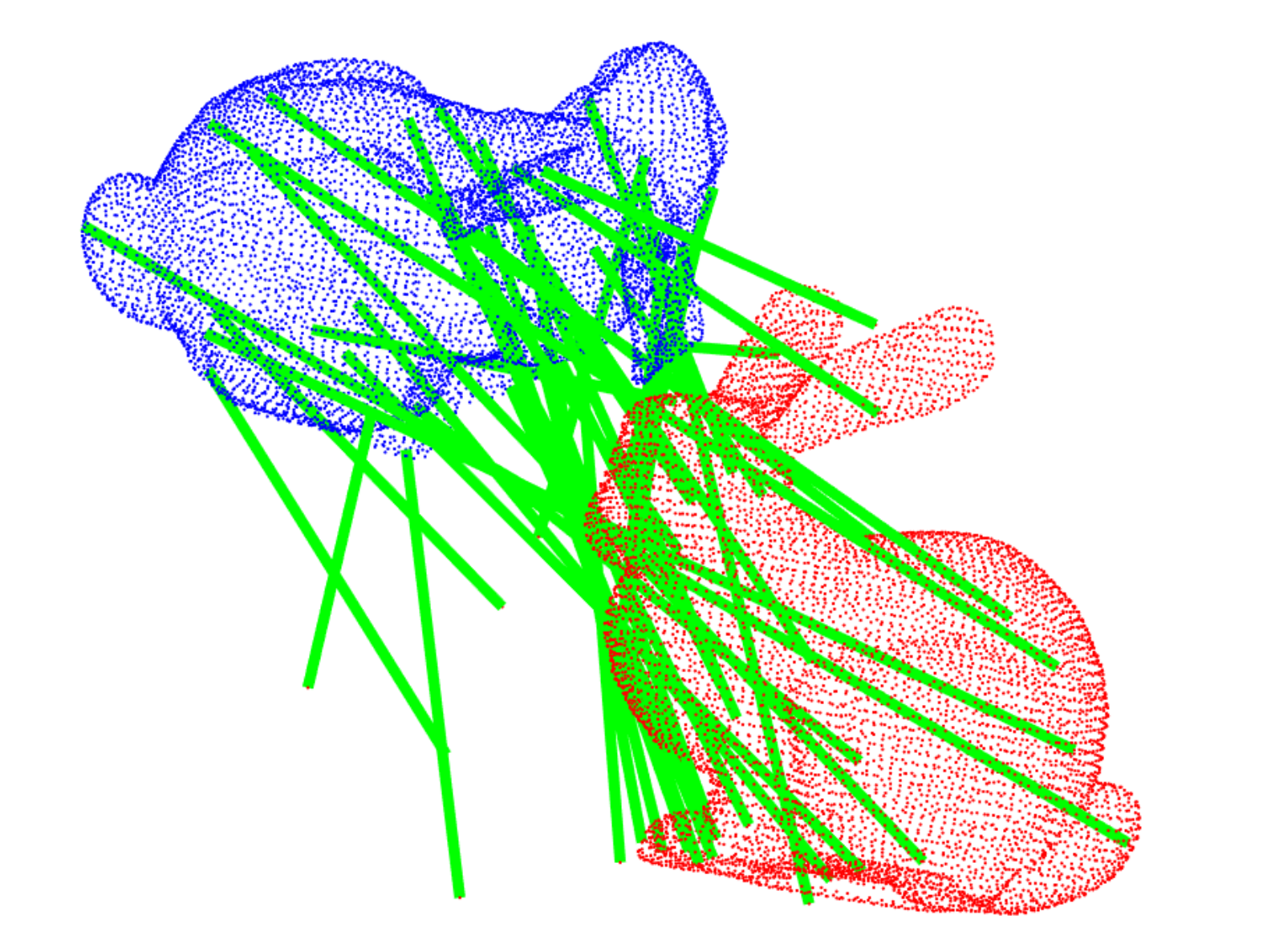} \\
			\end{minipage}
		& \myhspace
			\begin{minipage}{\mpwthree}%
			\centering
			Input images and \SURF correspondences \\
			(Green: inliers, Red: outliers) \\
			\vspace{2mm}
			\includegraphics[width=\columnwidth]{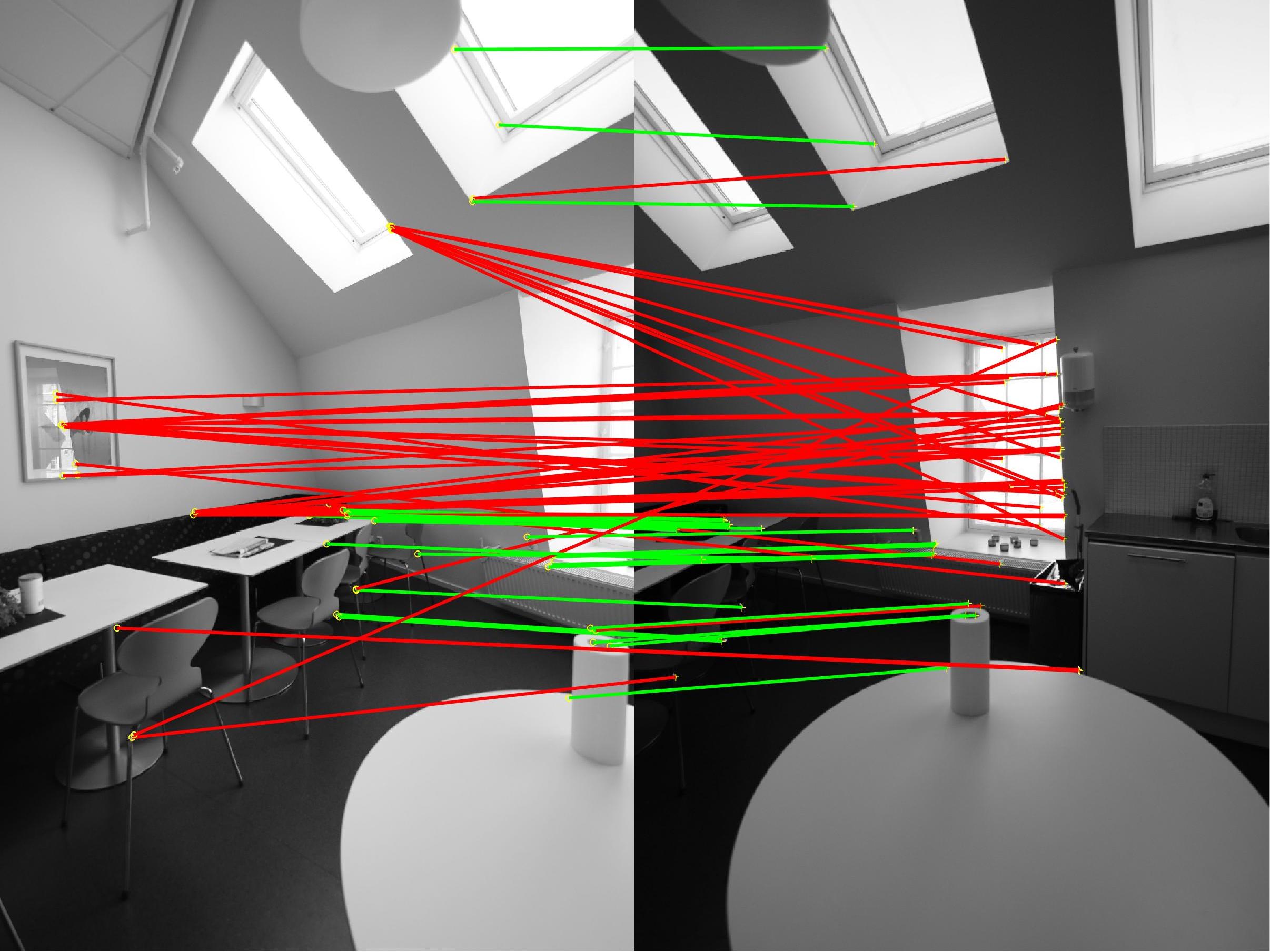} \\
			\end{minipage}  \\
		\begin{minipage}{\mpwthree}%
			\centering
			\includegraphics[width=\columnwidth]{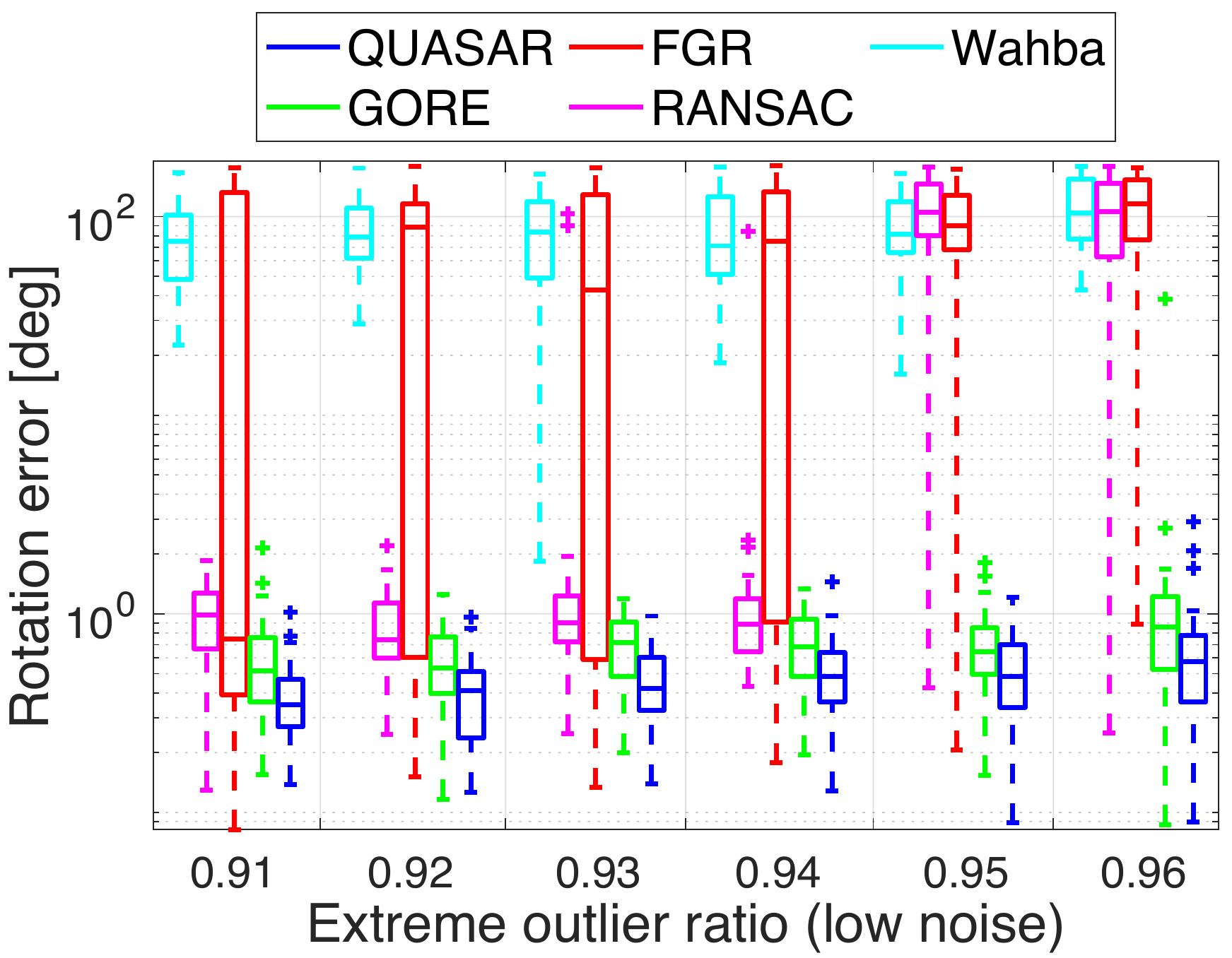} \\
			\end{minipage}
			\vspace{2mm}
		& \myhspace
			\begin{minipage}{\mpwthree}%
			\centering%
			\includegraphics[width=\columnwidth]{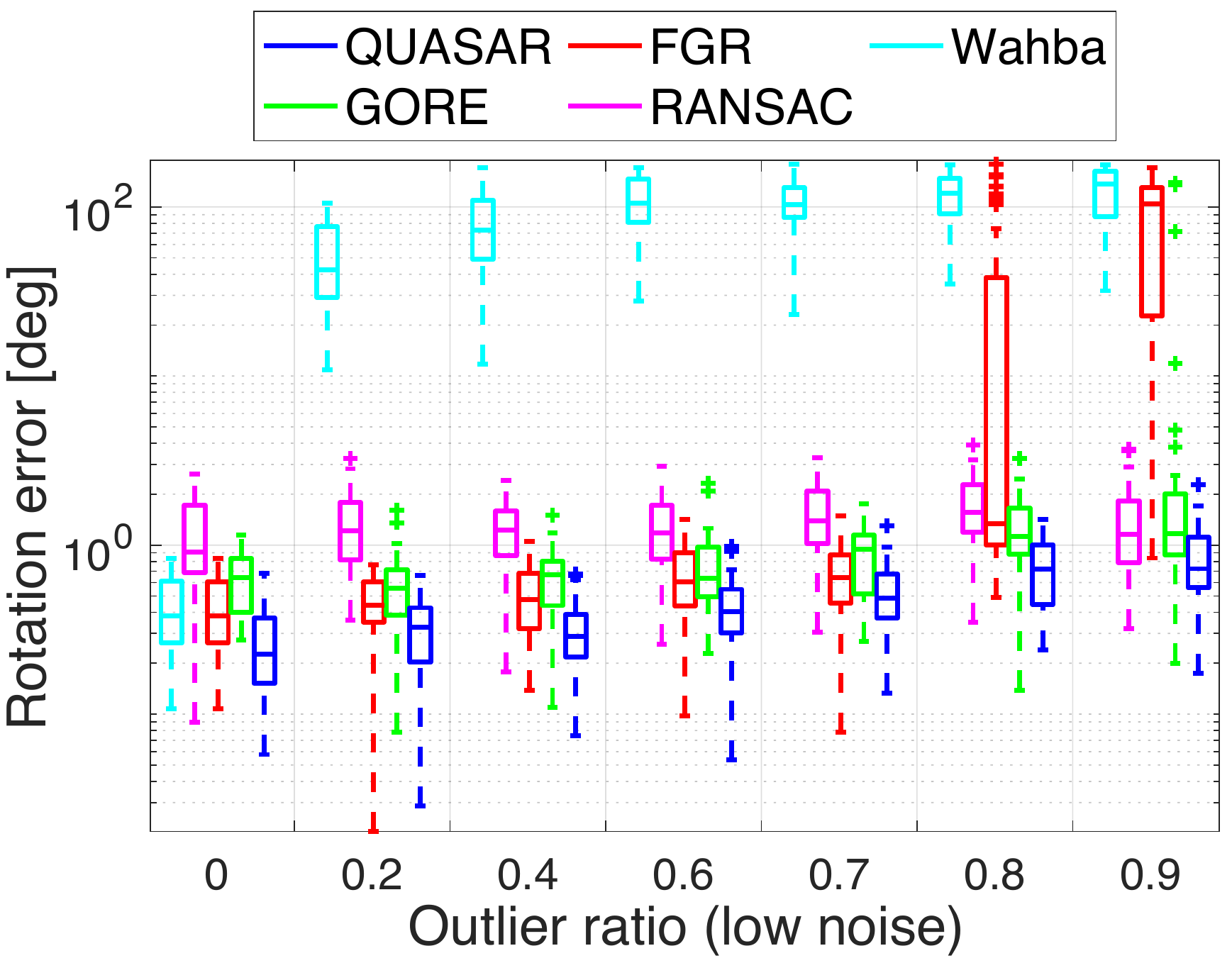} \\
			\end{minipage}
		& \multirow[h]{2}{*}{
		    \myhspace
			\begin{minipage}{\mpwthree}%
			\vspace{-16mm}
			\centering%
			\name stitched image \\
			\vspace{0mm}
			\includegraphics[width=\columnwidth]{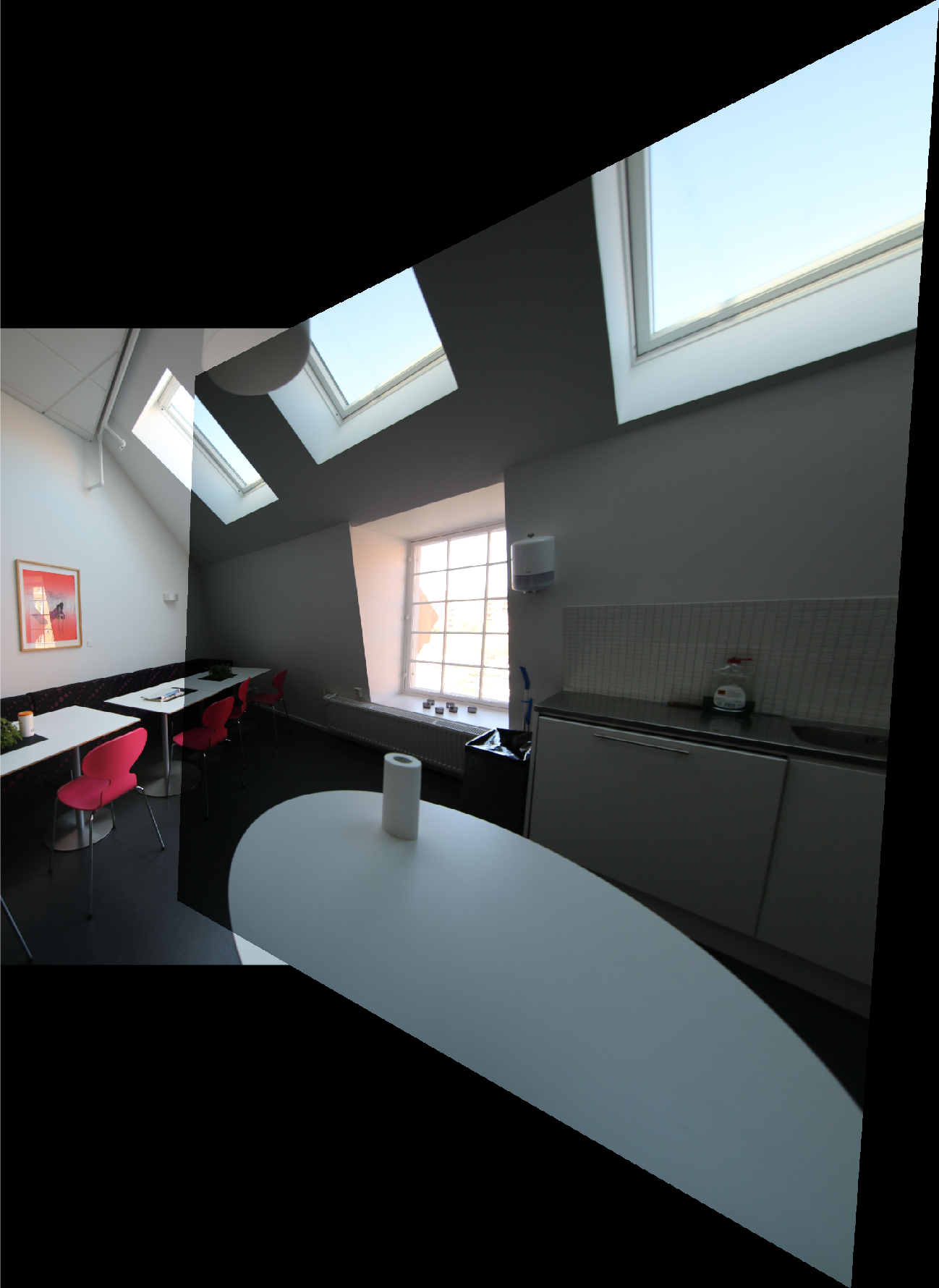} \\
			\vspace{3.0mm}
			(c) Image stitching
			\end{minipage} }\\
		\begin{minipage}{\mpwthree}%
			\centering%
			\includegraphics[width=\columnwidth]{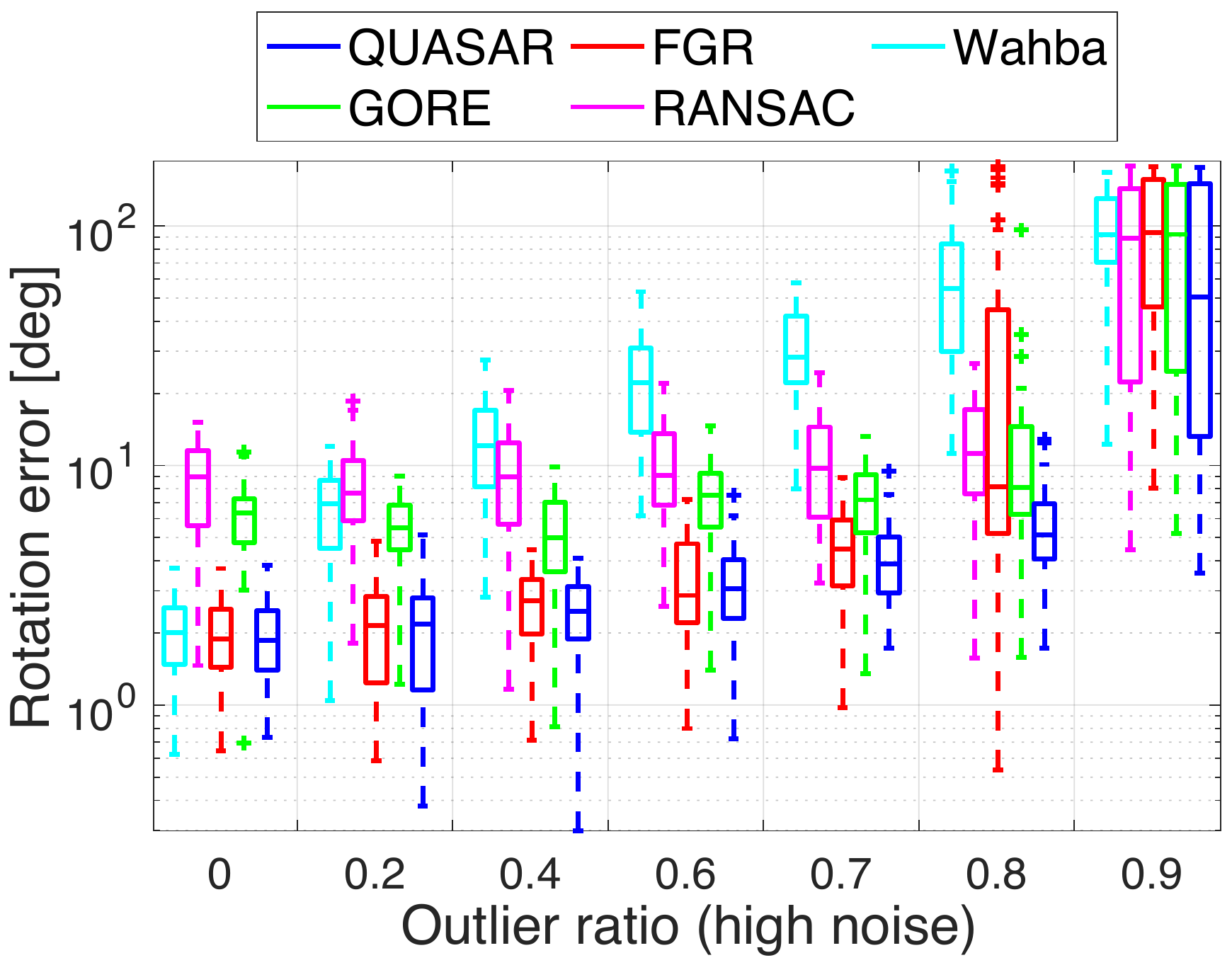} \\
			(a) Synthetic datasets
			\end{minipage}
		& \myhspace
			\begin{minipage}{\mpwthree}%
			\centering%
			\includegraphics[width=\columnwidth]{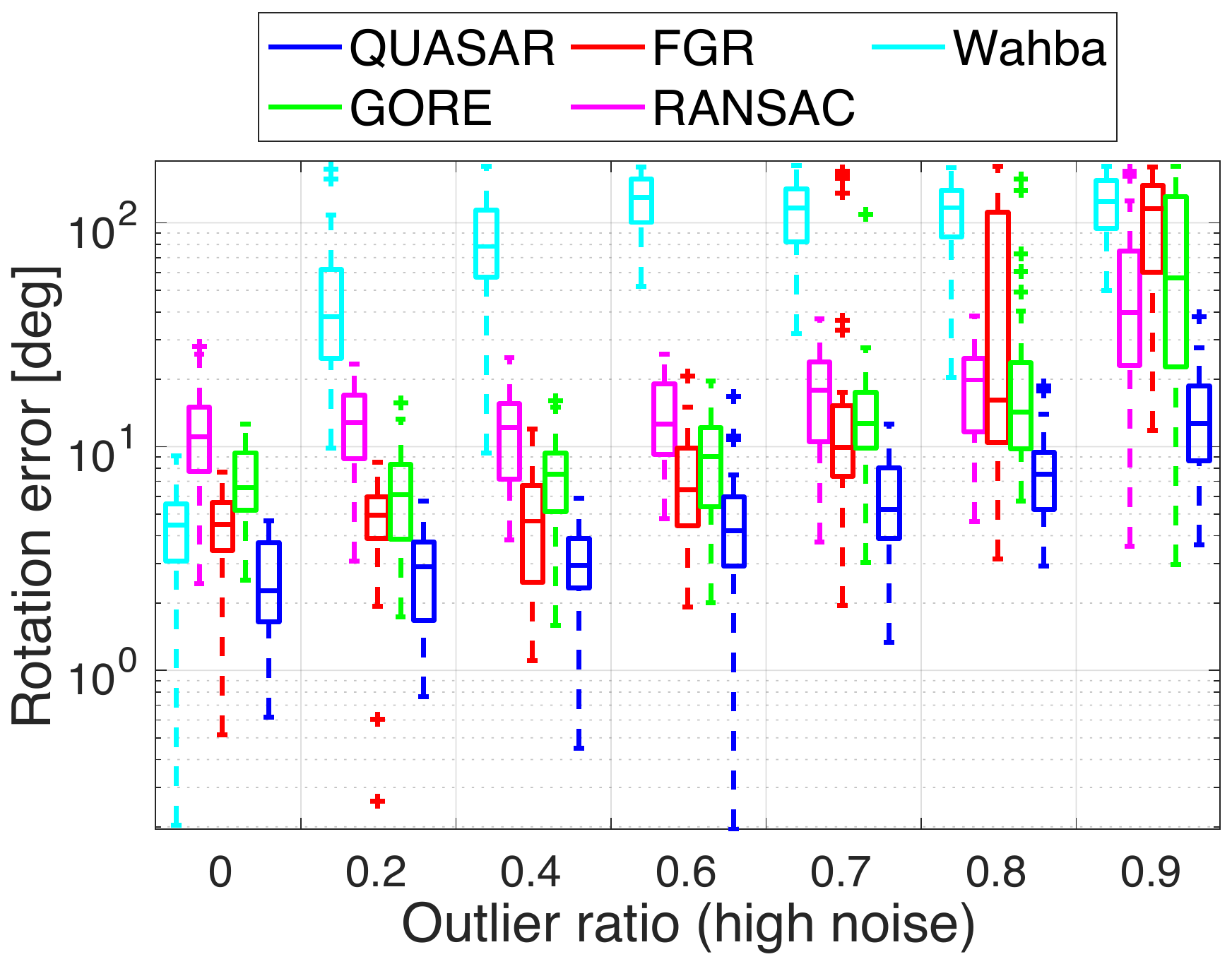} \\
			(b) Point cloud registration
			\end{minipage}
		& 

		\end{tabular}
	\end{minipage}
	\vspace{1mm} 
	\caption{(a) Rotation errors for increasing levels of outliers in synthetic datasets with low and high inlier noise.
	(b) Rotation errors on the \bunny dataset for point cloud registration.
	(c) Sample result of \name on the PASSTA image stitching dataset. 
	}
	 \label{fig:experiments_benchmark}
	\vspace{-6mm} 
	\end{center}
\end{figure*}

\myParagraph{Comparison against the State of the Art} 
Using the setup described in the previous section, we test the performance of \name in the presence of noise and outliers and 
 benchmark it against (i) the closed-form solution~\cite{horn1988josa-rotmatsol} to the standard Wahba problem~\eqref{eq:wahbaProblem} (label: \wahba); (ii) \ransac with $1000$ maximum iteration (label: \ransac); (iii) \emph{Fast Global Registration}~\cite{Zhou16eccv-fastGlobalRegistration} (label: \FGR); (iv) \emph{Guaranteed Outlier Removal}~\cite{Bustos2015iccv-gore3D} (label: \GORE). 
 We used default parameters for all approaches. \edit{We also benchmarked \name against \bnb methods~\cite{bazin2012accv-globalRotSearch,Bustos2015iccv-gore3D}, but we found \GORE had similar or better performance than \bnb. Therefore, the comparison with \bnb is presented in the \supp.}

Fig.~\ref{fig:experiments_benchmark}(a, first row) shows the box-plot of the rotation errors produced by the 
compared techniques for increasing ratios ($0-90\%$) of outliers when the inlier noise is low ($\sigma_i=0.01$, $i=1,\ldots,N$). 
As expected, \wahba only produces reasonable results in the outlier-free case. 
\FGR is robust against $70\%$ outliers but breaks at $90\%$ (many tests result in large errors even at $80\%$, see red ``\myplus'' in the figure).
\ransac, \GORE, and \name are robust to $90\%$ outliers, with \name being slight more accurate than the others.
In some of the tests with $90\%$ outliers, 
\ransac converged to poor solutions (red ``\myplus'' in the figure).


\begin{table*}
\centering
\begin{tabular}{ccccccccc}
  &  & &  \multicolumn{2}{c}{Relative relaxation gap} & \multicolumn{2}{c}{Solution rank} & \multicolumn{2}{c}{Solution stable rank} \\
  Datasets & Noise level & Outlier ratio & Mean & SD & Mean & SD & Mean & SD \\
  \hline
  \multirow{4}{*}{Synthetic} & Low & $0-0.9$ & $4.32\mathrm{e}{-9}$ & $3.89\mathrm{e}{-8}$ & $1$ & $0$ & $1+1.19\mathrm{e}{-16}$ & $1.20\mathrm{e}{-15}$ \\
   & Low & $0.91-0.96$ & $1.47\mathrm{e}{-8}$ & $2.10\mathrm{e}{-7}$  & 1 & 0 & $1+6.76\mathrm{e}{-9}$ & $1.05\mathrm{-7}$ \\
   & High & $0-0.8$ & $2.25\mathrm{e}{-8}$ & $3.49\mathrm{e}{-7}$ & 1 & 0 & $1+9.08\mathrm{e}{-8}$ & $1.41\mathrm{e}{-6}$ \\
   & High & $0.9$ & $0.03300$ & $0.05561$ & $33.33$ & $36.19$ & $1.1411$ & $0.2609$ \\
  \hline
  \multirow{2}{*}{\bunny} & Low & $0-0.9$ & $1.53\mathrm{e}{-8}$ & $1.64\mathrm{e}{-7}$ & 1 & 0 & $1+4.04\mathrm{-16}$ & $4.83\mathrm{e}{-15}$ \\
   & High & $0-0.9$ & $9.96\mathrm{e}{-12}$ & $4.06\mathrm{e}{-11}$ & 1 & 0 & $1+7.53\mathrm{e}{-18}$ & $3.85\mathrm{e}{-16}$ \\
  \hline
\end{tabular}
\vspace{2mm}
\caption{Mean and standard deviation (SD) of the relative relaxation gap, the solution rank, and the solution stable rank of \name on the synthetic and the \bunny datasets. The relaxation is \emph{tight} in all tests except in the synthetic tests with high noise and 90\% outliers.} 
\label{tab:dualityGapandRank}
\vspace{-4mm}
\end{table*}

Fig.~\ref{fig:experiments_benchmark}(a, second row) shows the rotation errors for
low noise ($\sigma_i=0.01$, $i=1,\ldots,N$) and extreme outlier ratios ($91-96\%$). 
Here we use $N=100$ to ensure a sufficient number of inliers.
Not surprisingly, \wahba, \FGR, and \ransac break at such extreme levels of outliers. \GORE fails only once at $96\%$ outliers (red ``\myplus'' in the figure), 
while \name returns highly-accurate solutions in all tests. 

Fig.~\ref{fig:experiments_benchmark}(a, third row) shows the rotation errors
for a higher noise level ($\sigma_i=0.1$, $i=1,\ldots,N$) and increasing outlier ratios ($0-90\%$).
Even with large noise, \name is still robust against $80\%$ outliers, 
an outlier level where all the other algorithms fail to produce an accurate solution. 
\vspace{-1mm}
\subsection{Point Cloud Registration}
\vspace{-1mm}
\edit{In point cloud registration,~\cite{Yang2019rss-TEASER} showed how to build invariant measurements to decouple the estimation of scale, rotation and translation. Therefore, in this section we test \name to solve the rotation-only subproblem.} We use the \bunny dataset from the Stanford 3D Scanning Repository \cite{Curless96siggraph} and resize the corresponding point cloud to be within the $[0, 1]^3$ cube. The \bunny is first down-sampled to $N = 40$ points, and then we apply a random rotation with additive noise and random outliers, according to eq.~\eqref{eq:generativeModel}. Results are averaged over 40 Monte Carlo runs.
Fig.~\ref{fig:experiments_benchmark}(b, first row) shows an example of the point cloud registration problem with vector-to-vector putative correspondences (including outliers) shown in green. 
Fig.~\ref{fig:experiments_benchmark}(b, second row) and Fig.~\ref{fig:experiments_benchmark}(b, third row)
evaluate the compared techniques for increasing outliers in the 
low noise ($\sigma_i=0.01$) and the high noise regime ($\sigma_i=0.1$), respectively. The results confirm our findings from Fig.~\ref{fig:experiments_benchmark}(a, first row) and Fig.~\ref{fig:experiments_benchmark}(a, third row).
\name dominates the other techniques. Interestingly, in this dataset \name performs even better (more accurate results) 
at $90\%$ outliers and high noise.

\myParagraph{Tightness of \name} Table~\ref{tab:dualityGapandRank} provides a detailed evaluation of the quality of
\edit{\name} in both the synthetic and the \bunny datasets. Tightness is evaluated in terms of (i) the \emph{relative relaxation gap}, defined as $\frac{f_\text{QCQP} - f^\star_\text{SDP}}{f_\text{QCQP}}$, where 
$f^\star_\text{SDP}$ is the optimal cost of the relaxation~\eqref{eq:SDPrelax} 
and $f_\text{QCQP}$ is the cost attained in~\eqref{eq:TLSBinaryClone} by the corresponding (possibly rounded) solution  
(the relaxation is exact when the gap is zero), (ii) the rank of the optimal solution of~\eqref{eq:SDPrelax} (the relaxation is exact when the rank is one).
Since the evaluation of the rank requires setting a numeric tolerance ($10^{-6}$ in our tests), we also report the \emph{stable rank},
 the squared ratio between the Frobenius norm and the spectral norm, which \edit{is less sensitive to the choice of the numeric tolerance.}

\vspace{-2mm}
\subsection{Image Stitching}
We use the PASSTA dataset to test \name in challenging panorama stitching applications~\cite{meneghetti2015scia-PASSATImageStitchingData}. 
To merge two images together, we first use \SURF~\cite{Bay06eccv} feature descriptors to match and establish putative feature correspondences. Fig.~\ref{fig:experiments_benchmark}(c, first row) shows the established correspondences between two input images from the \emph{Lunch Room} dataset. Due to significantly different lighting conditions, small overlapping area, and different objects having similar image features, 46 of the established 70 \SURF correspondences ($66\%$) are outliers (shown in red). 
From the \SURF feature points, we apply the inverse of the known camera intrinsic matrix $\MK$ to obtain unit-norm bearing vectors ($\{\va_i,\vb_i\}_{i=1}^{70}$) observed in each camera frame. 
Then we use \name (with $\sigma^2\barcsq=0.001$) to find the relative rotation $\MR$ between the two camera frames. 
Using the estimated rotation, we compute the homography matrix as $\MH=\MK\MR\MK\inv$ to stitch the pair of images together. 
Fig.~\ref{fig:experiments_benchmark}(c) shows an example of the image stitching results.
 \name performs accurate stitching in very challenging instances with many outliers and small image overlap. 
 On the other hand, applying the \emph{M-estimator SAmple Consensus} (MSAC)~\cite{Torr00cviu,Hartley04book} algorithm, as implemented by the Matlab ``\scenario{estimateGeometricTransform}'' function, results in an incorrect stitching (see the \supp for extra results and statistics).


\vspace{-2mm}
\section{Conclusions}
\label{sec:conclusions}
\vspace{-1mm}
We propose the first polynomial-time certifiably optimal solution to the Wahba problem with outliers.
The crux of the approach is the use of a \edit{\TLS}
cost function that makes the estimation 
insensitive to a large number of outliers. 
The second key ingredient is to write the \TLS problem as a \edit{\QCQP}
using a quaternion representation for the unknown rotation. 
Despite the simplicity of the \QCQP formulation, the problem remains non-convex and hard to solve globally.
Therefore, we develop a convex SDP relaxation. 
While a naive relaxation of the \QCQP 
is loose in the presence of outliers, we propose an improved relaxation with redundant constraints, named \edit{\name}.
 We provide a theoretical proof that \name is tight (computes an exact solution to the \QCQP) 
 in the noiseless and outlier-free case. 
 More importantly, we empirically show that the relaxation remains tight in the face of \edit{high noise and extreme outliers (\maxOutliers)}. 
Experiments on synthetic and real datasets 
for point cloud registration and image stitching 
show that \name outperforms \ransac, as well as state-of-the-art robust local optimization techniques, \edit{global outlier-removal procedures, and \bnb optimization methods}. \edit{While running in polynomial time, the general-purpose SDP solver used in our current implementation scales poorly in the problem size (about 1200 seconds with \mosek~\cite{mosek} and 500 seconds with \sdpnalplus~\cite{Yang2015mpc-sdpnalplus} for $100$ correspondences). Current research effort is devoted to developing fast specialized SDP solvers along the line of~\cite{Rosen18ijrr-sesync}.}


\vspace{-2mm}
\section*{Acknowledgements}
\label{sec:acknowledgements}
\vspace{-2mm}
\edit{This work was partially funded by ARL DCIST CRA W911NF-17-2-0181, ONR RAIDER
N00014-18-1-2828, and Lincoln Laboratory’s Resilient Perception in Degraded Environments program. The authors want to thank \'Alvaro Parra Bustos and Tat-Jun Chin for kindly sharing \bnb implementations used in~\cite{bazin2012accv-globalRotSearch,Bustos2015iccv-gore3D}.}


\renewcommand{\theequation}{A\arabic{equation}}
\renewcommand{\thetheorem}{A\arabic{theorem}}
\renewcommand{\thefigure}{A\arabic{figure}}
\renewcommand{\thetable}{A\arabic{table}}
\appendix
\large
\begin{center}
{\bf Supplementary Material }
\end{center}
\normalsize

\section{Proof of Proposition~\ref{prop:binaryCloning}}
\label{sec:proof:prop:binaryCloning}
\begin{proof}
Here we prove the equivalence between the mixed-integer program~\eqref{eq:TLSadditiveForm} and the optimization in~\eqref{eq:TLSadditiveForm2} involving $N+1$ quaternions. To do so, 
we note that since $\theta_i \in \{+1,-1\}$ and $\frac{1+\theta_i}{2} \in \{0,1\}$, we can safely move $\frac{1+\theta_i}{2}$ inside the squared norm (because $0=0^2,1=1^2$) in each summand of the cost function~\eqref{eq:TLSadditiveForm}:
\bea \label{eq:moveInsideNorm}
& \sumAllPointsi \frac{1+\theta_i}{2}\frac{\Vert \hatvb_i - \vq \kron \hatva_i \kron \vq\inv \Vert^2}{\sigma_i^2} + \frac{1-\theta_i}{2} \barcsq 
 \hspace{-5mm}\\
& \hspace{-9mm} = \sumAllPointsi \frac{\Vert \hatvb_i - \vq \kron \hatva_i \kron \vq\inv + \theta_i\hatvb_i - \vq \kron \hatva_i \kron (\theta_i \vq\inv) \Vert^2}{4\sigma_i^2} + \frac{1-\theta_i}{2}\barcsq  \nonumber\hspace{-5mm}  
\eea
Now we introduce $N$ new unit quaternions $\vq_i = \theta_i \vq$, $i=1,\dots,N$ by multiplying $\vq$ by the $N$ binary variables $\theta_i \in \{+1,-1\}$, a re-parametrization we called \emph{binary cloning}. One can easily verify that $\vq\tran \vq_i = \theta_i (\vq\tran \vq) = \theta_i$. Hence, by substituting $\theta_i=\vq\tran \vq_i$ into~\eqref{eq:moveInsideNorm}, we can rewrite the mixed-integer program~\eqref{eq:TLSadditiveForm} as:
\bea
& \min\limits_{\substack{\vq \in S^3 \\ \vq_i \in \{\pm \vq\}}} 
\sumAllPointsi \frac{\Vert \hatvb_i - \vq \kron \hatva_i \kron \vq\inv + \vq\tran \vq_i \hatvb_i - \vq \kron \hatva_i \kron \vq_i\inv \Vert^2 }{4\sigma_i^2} \nonumber \\
& + \frac{1-\vq\tran\vq_i}{2}\barcsq,
\eea
which is the same as the optimization in~\eqref{eq:TLSadditiveForm2}.
\end{proof}

\section{Proof of Proposition~\ref{prop:qcqp}}
\label{sec:proof:prop:qcqp}
\begin{proof}
Here we show that the optimization involving $N+1$ quaternions in~\eqref{eq:TLSadditiveForm2} can be reformulated as the Quadratically-Constrained Quadratic Program (\QCQP) in~\eqref{eq:TLSBinaryClone}. Towards this goal, 
we prove that the objective function and the constraints in the \QCQP are a 
re-parametrization of the ones in~\eqref{eq:TLSadditiveForm2}.

{\bf Equivalence of the objective functions.}
We start by developing the squared 2-norm term in~\eqref{eq:TLSadditiveForm2}:
\bea
& \Vert \hatvb_i - \vq \kron \hatva_i \kron \vq\inv + \vq\tran \vq_i \hatvb_i - \vq \kron \hatva_i \kron \vq_i\inv \Vert^2 \nonumber \\
& \expl{$\Vert \vq\tran \vq_i \hatvb_i \Vert^2 = \Vert \hatvb_i \Vert^2 = \Vert \vb_i \Vert^2$, $\hatvb_i\tran(\vq\tran\vq_i)\hatvb_i = \vq\tran\vq_i\Vert \vb_i \Vert^2$}
& \expl{$\Vert \vq \kron \hatva_i \kron \vq\inv \Vert^2 = \Vert \MR \va_i\Vert^2 = \Vert \va_i \Vert^2$}
& \expl{$\Vert \vq \kron \hatva_i \kron \vq_i\inv \Vert^2 = \Vert \theta_i \MR \va_i\Vert^2 = \Vert \va_i \Vert^2$}
& \expl{${\scriptstyle (\vq \kron \hatva_i \kron \vq\inv)\tran (\vq \kron \hatva_i \kron \vq_i\inv) = (\MR \va_i)\tran(\theta_i \MR \va_i) = \vq\tran\vq_i\Vert \va_i \Vert^2}$}
& = 2\Vert \vb_i \Vert^2 + 2\Vert \va_i \Vert^2 + 2\vq\tran \vq_i \Vert \vb_i \Vert^2 + 2\vq\tran \vq_i \Vert \va_i \Vert^2 \nonumber \\
& - 2\hatvb_i\tran (\vq\kron \hatva_i \kron \vq\inv) -2\hatvb_i\tran (\vq\kron\hatva_i\kron\vq_i\inv) \nonumber \\
& \hspace{-8mm}- 2\vq\tran\vq_i\hatvb_i\tran(\vq\kron \hatva_i \kron \vq\inv) -2\vq\tran\vq_i\hatvb_i\tran(\vq\kron\hatva_i\kron\vq_i\inv) \\
& \expl{${\scriptstyle \vq\tran\vq_i\hatvb_i\tran(\vq\kron\hatva_i\kron\vq_i\inv)=(\theta_i)^2\hatvb_i\tran(\vq\kron\hatva_i\kron\vq\inv)= \hatvb_i\tran(\vq\kron\hatva_i\kron\vq\inv)}$ }
& \expl{ $\hatvb_i\tran(\vq\kron\hatva_i\kron\vq_i\inv) = \vq\tran \vq_i \hatvb_i\tran(\vq\kron\hatva_i\kron\vq\inv) $ }
& = 2\Vert \vb_i \Vert^2 + 2\Vert \va_i \Vert^2 + 2\vq\tran \vq_i \Vert \vb_i \Vert^2 + 2\vq\tran \vq_i \Vert \va_i \Vert^2 \nonumber \\
& - 4\hatvb_i\tran (\vq\kron \hatva_i \kron \vq\inv) - 4\vq\tran\vq_i\hatvb_i\tran (\vq\kron \hatva_i \kron \vq\inv) \label{eq:afterDevelopSquares}
\eea
where we have used multiple times the binary cloning equalities $\vq_i = \theta_i \vq, \theta_i = \vq\tran \vq_i$, the equivalence between applying rotation to a homogeneous vector $\hatva_i$ using quaternion product and using rotation matrix in eq.~\eqref{eq:q_pointRot} from the main document, as well as the fact that vector 2-norm is invariant to rotation and homogenization (with zero padding). 

Before moving to the next step, we make the following observation by combing eq.~\eqref{eq:Omega_1} and eq.~\eqref{eq:q_inverse}:
\bea
\MOmega_1(\vq\inv) = \MOmega_1\tran(\vq),\quad \MOmega_2(\vq\inv) = \MOmega_2\tran(\vq)
\eea
which states the linear operators $\MOmega_1(\cdot)$ and $\MOmega_2(\cdot)$ of $\vq$ and its inverse $\vq\inv$ are related by a simple transpose operation.
In the next step, we use the equivalence between quaternion product and linear operators in $\MOmega_1(\vq)$ and $\MOmega_2(\vq)$ as defined in eq.~\eqref{eq:quatProduct}-\eqref{eq:Omega_1} to simplify $\hatvb_i\tran(\vq\kron\hatva_i \kron \vq\inv)$ in eq.~\eqref{eq:afterDevelopSquares}:
\bea
& \hatvb_i\tran(\vq\kron\hatva_i \kron \vq\inv) \nonumber \\
& \expl{${\scriptstyle \vq\kron\hatva_i = \MOmega_1(\vq)\hatva_i \;,\; \MOmega_1(\vq)\hatva_i \kron \vq\inv = \MOmega_2(\vq\inv)\MOmega_1(\vq)\hatva_i = \MOmega_2\tran(\vq)\MOmega_1(\vq)\hatva_i}$} 
& = \hatvb_i\tran (\MOmega_2\tran(\vq)\MOmega_1(\vq) \hatva_i) \\
& \hspace{-8mm} \expl{$\MOmega_2(\vq)\hatvb_i = \hatvb_i \kron \vq = \MOmega_1(\hatvb_i) \vq$ \;,\; $\MOmega_1(\vq)\hatva_i = \vq \kron \hatva_i = \MOmega_2(\hatva_i)\vq$}
& = \vq\tran \MOmega_1\tran(\hatvb_i) \MOmega_2(\hatva_i) \vq. \label{eq:developQuatProduct}
\eea
Now we can insert eq.~\eqref{eq:developQuatProduct} back to eq.~\eqref{eq:afterDevelopSquares} and write:
\bea
& \Vert \hatvb_i - \vq \kron \hatva_i \kron \vq\inv + \vq\tran \vq_i \hatvb_i - \vq \kron \hatva_i \kron \vq_i\inv \Vert^2 \nonumber \\
& = 2\Vert \vb_i \Vert^2 + 2\Vert \va_i \Vert^2 + 2\vq\tran \vq_i \Vert \vb_i \Vert^2 + 2\vq\tran \vq_i \Vert \va_i \Vert^2 \nonumber \\
& \hspace{-5mm} - 4\hatvb_i\tran (\vq\kron \hatva_i \kron \vq\inv) - 4\vq\tran\vq_i\hatvb_i\tran (\vq\kron \hatva_i \kron \vq\inv) \\
& = 2\Vert \vb_i \Vert^2 + 2\Vert \va_i \Vert^2 + 2\vq\tran \vq_i \Vert \vb_i \Vert^2 + 2\vq\tran \vq_i \Vert \va_i \Vert^2 \nonumber \\
& \hspace{-5mm} - 4\vq\tran\MOmega_1\tran(\hatvb_i)\MOmega_2(\hatva_i)\vq - 4\vq\tran\vq_i \vq\tran\MOmega_1\tran(\hatvb_i)\MOmega_2(\hatva_i)\vq \\
& \expl{${\scriptstyle \vq\tran\vq_i \vq\tran\MOmega_1\tran(\hatvb_i)\MOmega_2(\hatva_i)\vq=\theta_i \vq\tran\MOmega_1\tran(\hatvb_i)\MOmega_2(\hatva_i)\vq = \vq\tran\MOmega_1\tran(\hatvb_i)\MOmega_2(\hatva_i)\vq_i}$}
& = 2\Vert \vb_i \Vert^2 + 2\Vert \va_i \Vert^2 + 2\vq\tran \vq_i \Vert \vb_i \Vert^2 + 2\vq\tran \vq_i \Vert \va_i \Vert^2 \nonumber \\
& \hspace{-5mm} - 4\vq\tran\MOmega_1\tran(\hatvb_i)\MOmega_2(\hatva_i)\vq - 4 \vq\tran\MOmega_1\tran(\hatvb_i)\MOmega_2(\hatva_i)\vq_i \\
& \expl{$-\MOmega_1\tran(\hatvb_i) = \MOmega_1(\hatvb_i)$} 
& = 2\Vert \vb_i \Vert^2 + 2\Vert \va_i \Vert^2 + 2\vq\tran \vq_i \Vert \vb_i \Vert^2 + 2\vq\tran \vq_i \Vert \va_i \Vert^2 \nonumber \\
& \hspace{-5mm} + 4\vq\tran\MOmega_1(\hatvb_i)\MOmega_2(\hatva_i)\vq + 4 \vq\tran\MOmega_1(\hatvb_i)\MOmega_2(\hatva_i)\vq_i, \label{eq:finishDevelopSquareNorm}
\eea
which is quadratic in $\vq$ and $\vq_i$. Substituting eq.~\eqref{eq:finishDevelopSquareNorm} back to~\eqref{eq:TLSadditiveForm2}, we can write the cost function as:
\bea \label{eq:costInQuadraticFormBeforeLifting}
& \sum\limits_{i=1}^N \frac{\Vert \hatvb_i - \vq \kron \hatva_i \kron \vq\inv + \vq\tran \vq_i \hatvb_i - \vq \kron \hatva_i \kron \vq_i\inv \Vert^2 }{4\sigma_i^2} + \frac{1-\vq\tran\vq_i}{2}\barcsq \nonumber \hspace{-5mm}\\
& \hspace{-10mm}= \sum\limits_{i=1}^N \vq_i\tran \left( \underbrace{ \frac{(\Vert \vb_i \Vert^2+ \Vert \va_i \Vert^2)\eye_4 + 2\MOmega_1(\hatvb_i)\MOmega_2(\hatva_i)}{2\sigma_i^2} + \frac{\barcsq}{2}\eye_4 }_{:=\MQ_{ii}}\right) \vq_i \nonumber\hspace{-10mm} \\
& \hspace{-15mm} + 2\vq\tran \left( \underbrace{ \frac{(\Vert \vb_i \Vert^2+ \Vert \va_i \Vert^2)\eye_4 + 2\MOmega_1(\hatvb_i)\MOmega_2(\hatva_i)}{4\sigma_i^2}  - \frac{\barcsq}{4}\eye_4 }_{:=\MQ_{0i}} \right) \vq_i, \nonumber\hspace{-15mm}\\
\eea
where we have used two facts: (i) $\vq\tran \MA \vq = \theta_i^2 \vq\tran \MA \vq = \vq_i\tran \MA \vq_i$ for any matrix $\MA\in\Real{4 \times 4}$, (ii) $c=c\vq\tran \vq = \vq\tran (c\eye_4) \vq$ for any real constant $c$, which allowed writing the quadratic forms of $\vq$ and constant terms in the cost as quadratic forms of $\vq_i$. Since we have not changed the decision variables $\vq$ and $\{\vq_i\}_{i=1}^N$, the optimization in~\eqref{eq:TLSadditiveForm2} is therefore equivalent to the following optimization:
\bea \label{eq:optQCQPbeforeLifting}
\min_{\substack{\vq \in S^3 \\ \vq_i \in \{\pm \vq\}}} \sumAllPointsi \vq_i\tran \MQ_{ii} \vq_i + 2\vq\tran \MQ_{0i} \vq_i
\eea
where $\MQ_{ii}$ and $\MQ_{0i}$ are the known $4\times 4$ data matrices as defined in eq.~\eqref{eq:costInQuadraticFormBeforeLifting}.

Now it remains to prove that the above optimization~\eqref{eq:optQCQPbeforeLifting} is equivalent to the \QCQP in~\eqref{eq:TLSBinaryClone}. Recall that $\vxx$ is the column vector stacking all the $N+1$ quaternions, \ie, $\vxx=[\vq\tran\ \vq_1\tran\ \dots\ \vq_N\tran]\tran \in \Real{4(N+1)}$. 
Let us introduce 
symmetric matrices $\MQ_i \in \Real{4(N+1) \times 4(N+1) },i=1,\dots,N$ and let the $4 \times 4$ sub-block of $\MQ_i$ corresponding to sub-vector $\vu$ and $\vv$, be denoted as $[\MQ_i]_{uv}$; each $\MQ_i$ is defined as:
\bea \label{eq:Qi}
[\MQ_i]_{uv} = \begin{cases}
\MQ_{ii} & \text{if } \vu=\vq_i \text{ and }\vv=\vq_i \\
\MQ_{0i} & \substack{ \text{if } \vu=\vq \text{ and }\vv=\vq_i \\ \text{or } \vu=\vq_i \text{ and }\vv=\vq} \\
\zero_{4\times 4} & \text{otherwise}
\end{cases}
\eea
\ie, $\MQ_i$ has the diagonal $4\times 4$ sub-block corresponding to $(\vq_i,\vq_i)$ be $\MQ_{ii}$, has the two off-diagonal $4\times 4$ sub-blocks corresponding to $(\vq,\vq_i)$ and $(\vq_i,\vq)$ be $\MQ_{0i}$, and has all the other $4\times 4$ sub-blocks be zero. Then we can write the cost function in eq.~\eqref{eq:optQCQPbeforeLifting} compactly using $\vxx$ and $\MQ_i$:
\bea
& \displaystyle \sumAllPointsi \vq_i\tran \MQ_{ii} \vq_i + 2\vq\tran \MQ_{0i} \vq_i = \sumAllPointsi \vxx\tran \MQ_i \vxx
\eea
Therefore, we proved that the objective functions in~\eqref{eq:TLSadditiveForm2} and the \QCQP~\eqref{eq:TLSBinaryClone} are the same.

{\bf Equivalence of the constraints.}
  We are only left to prove that~\eqref{eq:TLSadditiveForm2} and~\eqref{eq:TLSBinaryClone} have the same feasible set, 
\ie, the following two sets of constraints are equivalent:
\bea \label{eq:equivalentConstraints}
\begin{cases}
\vq \in \calS^3 \\
\vq_i \in \{\pm \vq\},\\
i=1,\dots,N
\end{cases} \Leftrightarrow
\begin{cases}
\vxx_q\tran \vxx_q = 1 \\
\vxx_{q_i}\vxx_{q_i}\tran = \vxx_q \vxx_q\tran,\\
i=1,\dots,N
\end{cases}
\eea
We first prove the ($\Rightarrow$) direction. Since $\vq \in \calS^3$, it is obvious that $\vxx_q\tran \vxx_q = \vq\tran \vq = 1$. In addition, since $\vq_i \in \{+\vq,-\vq\}$, it follows that $\vxx_{q_i} \vxx_{q_i}\tran = \vq_i \vq_i\tran = \vq\vq\tran = \vxx_q \vxx_q\tran$. Then we proof the reverse direction ($\Leftarrow$). Since $\vxx_q\tran \vxx_q = \vq\tran \vq$, so $\vxx_q\tran \vxx_q = 1$ implies $\vq\tran \vq =1$ and therefore $\vq \in \calS^3$. On the other hand, $\vxx_{q_i} \vxx_{q_i}\tran = \vxx_q \vxx_q\tran$ means $\vq_i \vq_i\tran = \vq\vq\tran$. If we write $\vq_i = [q_{i1},q_{i2},q_{13},q_{i4}]\tran$ and $\vq = [q_1,q_2,q_3,q_4]$, then the following matrix equality holds:
\smaller
\bea
\hspace{-6mm}\bmat{cccc}
q_{i1}^2 & q_{i1}q_{i2} & q_{i1}q_{i3} & q_{i1}q_{i4} \\
\star & q_{i2}^2 & q_{i2}q_{i3} & q_{i2}q_{i4} \\
\star & \star & q_{i3}^2 & q_{i3}q_{i4} \\
\star & \star & \star & q_{i4}^2
\emat = 
\bmat{cccc}
q_1^2 & q_1q_2 & q_1q_3 & q_1q_4 \\
\star & q_2^2 & q_2q_3 & q_2q_4 \\
\star & \star & q_3^2 & q_3q_4 \\
\star & \star & \star & q_4^2
\emat
\nonumber\hspace{-10mm}\\
\eea
\normalsize
First, from the diagonal equalities, we can get $q_{ij}=\theta_j q_j, \theta_j\in\{+1,-1\}, j=1,2,3,4$. Then we look at the off-diagonal equality: $q_{ij}q_{ik}=q_jq_k, j\neq k$, since $q_{ij}=\theta_j q_j$ and $q_{ik} = \theta_k q_k$, we have $q_{ij}q_{ik} = \theta_j \theta_k q_j q_k$, from which we can have $\theta_j \theta_k =1, \forall j \neq k$. This implies that all the binary values $\{\theta_j\}_{j=1}^4$ have the same sign, and therefore they are equal to each other. As a result, $\vq_i = \theta_i \vq = \{+\vq, -\vq\}$, showing the two sets of constraints in eq.~\eqref{eq:equivalentConstraints} are indeed equivalent. Therefore, the \QCQP in eq.~\eqref{eq:TLSBinaryClone} is equivalent to the optimization in~\eqref{eq:optQCQPbeforeLifting}, and the original optimization in~\eqref{eq:TLSadditiveForm2} that involves $N+1$ quaternions, concluding the proof.
\end{proof}

\section{Proof of Proposition~\ref{prop:qcqpZ}}
\label{sec:proof:prop:matrixBinaryClone}
\begin{proof}
Here we show that the non-convex \QCQP written in terms of the vector $\vxx$ in Proposition~\ref{prop:qcqp} (and eq.~\eqref{eq:TLSBinaryClone}) is equivalent to the non-convex problem written using the matrix $\MZ$ in Proposition~\ref{prop:qcqpZ} (and eq.~\eqref{eq:qcqpZ}). We do so by showing that the objective function and the constraints in~\eqref{eq:qcqpZ} are a re-parametrization of the ones in~\eqref{eq:TLSBinaryClone}.

\myParagraph{Equivalence of the objective function} 
Since $\MZ = \vxx \vxx\tran$, and denoting $\MQ \doteq \sumAllPointsi \MQ_i$, we can rewrite the cost function in~\eqref{eq:qcqpZ} as:
\bea
& \sumAllPointsi \vxx\tran\MQ_i\vxx = \vxx\tran \left( \sumAllPointsi \MQ_i \right)\vxx = \vxx\tran \MQ \vxx  \nonumber \\
& = \trace{\MQ \vxx \vxx\tran} = \trace{\MQ \MZ}
\eea
showing the equivalence of the objectives in~\eqref{eq:TLSBinaryClone} and~\eqref{eq:qcqpZ}.

\myParagraph{Equivalence of the constraints} It is trivial to see that $\vxx_q\tran \vxx_q = \trace{\vxx_q \vxx_q\tran} = 1$ is equivalent to $\trace{[\MZ]_{qq}}=1$ by using the cyclic property of the trace operator and inspecting the structure of $\MZ$. In addition, $\vxx_{q_i}\vxx_{q_i}\tran = \vxx_q \vxx_q\tran$ also directly maps to $[\MZ]_{q_i q_i} = [\MZ]_{qq}$ for all $i=1,\dots,N$. Lastly, requiring $\MZ\succeq 0$ and $\rank{\MZ} = 1$ is equivalent to restricting $\MZ$ to the form $\MZ = \vxx\vxx\tran$ for some vector $\vxx \in \Real{4(N+1)}$. Therefore, the constraint sets of eq.~\eqref{eq:TLSBinaryClone} and~\eqref{eq:qcqpZ} are also equivalent, concluding the proof. 
\end{proof}

\section{Proof of Proposition~\ref{prop:naiveRelax}}
\label{sec:proof:prop:naiveRelax}
\begin{proof}
We show eq.~\eqref{eq:naiveRelaxation} is a convex relaxation of~\eqref{eq:qcqpZ} by showing that (i) eq.~\eqref{eq:naiveRelaxation} is a relaxation (\ie, the constraint set of~\eqref{eq:naiveRelaxation} includes the one of~\eqref{eq:qcqpZ}), and (ii) eq.~\eqref{eq:naiveRelaxation} is convex. (i) is true because from~\eqref{eq:qcqpZ} to~\eqref{eq:naiveRelaxation} we have dropped the $\rank{\MZ}=1$ constraint. Therefore, the feasible set of~\eqref{eq:qcqpZ} is a subset of the feasible set of~\eqref{eq:naiveRelaxation}, and the optimal cost of~\eqref{eq:naiveRelaxation} is always smaller or equal than the optimal cost of~\eqref{eq:qcqpZ}. To prove (ii), we note that the objective function and the constraints of~\eqref{eq:naiveRelaxation} are all linear in $\MZ$, and $\MZ\succeq 0$ is a convex constraint, hence~\eqref{eq:naiveRelaxation} is a convex program.
\end{proof}

\section{Proof of Theorem~\ref{thm:strongDualityNoiseless}}
\label{sec:proof:theorem:strongDuality}
\begin{proof}
To prove Theorem~\ref{thm:strongDualityNoiseless}, we first use Lagrangian duality to derive the \emph{dual} problem of the \QCQP in~\eqref{eq:TLSBinaryClone}, and draw connections to the naive SDP relaxation in~\eqref{eq:naiveRelaxation} (Section~\ref{sec:LagrangianandWeakDuality}). 
Then we leverage the well-known \emph{Karush-Kuhn-Tucker} (KKT) conditions~\cite{Boyd04book} to prove a general sufficient condition for tightness, as shown in Theorem~\ref{thm:generalStrongDuality} (Section~\ref{sec:KKTStrongDuality}). 
%
Finally, in Section~\ref{sec:strongDualityNoiselessOutlierfree}, we demonstrate that in the case of no noise and no outliers, we can provide a constructive proof to show the sufficient condition in Theorem~\ref{thm:generalStrongDuality} always holds.

\subsection{Lagrangian Function and Weak Duality}
\label{sec:LagrangianandWeakDuality}
Recall the expressions of $\MQ_i$ in eq.~\eqref{eq:Qi}, and define $\MQ = \sumAllPointsi \MQ_i$.
The matrix $\MQ$ has the following block structure:
\bea
\hspace{-5mm}
& \MQ = \sumAllPointsi \MQ_i 
= 
\bmat{c|ccc}
\MZero & \MQ_{01} & \dots & \MQ_{0N} \\
\hline
\MQ_{01} & \MQ_{11} & \dots & \MZero \\
\vdots & \vdots & \ddots & \vdots \\
\MQ_{0N} & \MZero & \dots & \MQ_{NN}
\emat. \label{eq:matrixCompactQ}
\eea
With cost matrix $\MQ$, and using the cyclic property of the trace operator, the \QCQP in eq.~\eqref{eq:TLSBinaryClone} can be written compactly as in the following proposition.
\begin{proposition}[Primal \QCQP] \label{prop:primalQCQP}
The \QCQP in eq.~\eqref{eq:TLSBinaryClone} is equivalent to the following \QCQP:
\bea \label{eq:qcqpforLagrangian}
(P)\quad \min_{\vxx \in \Real{4(N+1)}} & \trace{\MQ \vxx \vxx\tran} \\
\subject & \trace{\vxx_q \vxx_q\tran} = 1 \nonumber \\
& \vxx_{q_i} \vxx_{q_i}\tran = \vxx_{q} \vxx_{q}\tran, \forall i=1,\dots,N \nonumber
\eea
We call this \QCQP the \emph{primal problem} (P).
\end{proposition}

The proposition can be proven by inspection.
We now introduce the \emph{Lagrangian function}~\cite{Boyd04book} of the primal ($P$).
\begin{proposition}[Lagrangian of Primal \QCQP]\label{prop:LagrangianPrimalQCQP}
The Lagrangian function of ($P$) can be written as:
\bea
& \calL(\vxx,\mu,\MLambda)=\trace{\MQ\vxx\vxx\tran}-\mu(\trace{\MJ\vxx\vxx\tran}-1) \nonumber \\
& - \sumAllPointsi \trace{\MLambda_i \vxx\vxx\tran} \label{eq:LagrangianNaive0}\\
& = \trace{(\MQ - \mu\MJ - \MLambda) \vxx\vxx\tran} + \mu. \label{eq:LagrangianNaive}
\eea
where $\MJ$ is a sparse matrix with all zeros except the first $4\times 4$ diagonal block being identity matrix:
\bea \label{eq:matrixJ}
\MJ = \bmat{c|ccc}
\eye_4 & \MZero & \cdots & \MZero \\
\hline
\MZero & \MZero & \cdots & \MZero \\
\vdots & \vdots & \ddots & \vdots \\
\MZero & \MZero & \cdots & \MZero
\emat,
\eea
and each $\MLambda_i,i=1,\dots,N$ is a sparse Lagrangian multiplier matrix with all zeros except two diagonal sub-blocks $\pm \MLambda_{ii} \in \sym^{4\times 4}$ (symmetric $4\times4$ matrices):
\bea \label{eq:matrixLambdai}
\MLambda_i = 
\bmat{c|ccccc} 
\MLambda_{ii} & \MZero & \cdots & \MZero & \cdots & \MZero \\
\hline
\MZero & \MZero & \cdots & \MZero & \cdots & \MZero \\
\vdots & \vdots & \ddots & \vdots & \ddots & \vdots\\
\MZero & \MZero & \cdots & -\MLambda_{ii} & \cdots & \MZero \\
\vdots & \vdots & \ddots & \vdots & \ddots & \vdots\\
\MZero & \MZero & \cdots & \MZero & \cdots & \MZero
\emat,
\eea
and $\MLambda$ is the sum of all $\MLambda_i$'s, $i=1,\dots,N$:
\bea \label{eq:matrixLambda}
& \displaystyle \MLambda=\sumAllPointsi \MLambda_i = \nonumber \\
& \hspace{-8mm} \bmat{c|ccccc} 
\sum\limits_{i=1}^N \MLambda_{ii} & \MZero & \dots & \MZero & \dots & \MZero \\
\hline
\MZero & -\MLambda_{11} & \dots & \MZero & \dots & \MZero \\
\vdots & \vdots & \ddots & \vdots & \ddots & \vdots\\
\MZero & \MZero & \dots & -\MLambda_{ii} & \dots & \MZero \\
\vdots & \vdots & \ddots & \vdots & \ddots & \vdots\\
\MZero & \MZero & \dots & \MZero & \dots & -\MLambda_{NN}
\emat.
\eea
\end{proposition}
\begin{proof}
The sparse matrix $\MJ$ defined in~\eqref{eq:matrixJ} satisfies $\trace{\MJ \vxx \vxx\tran} = \trace{\vxx_q \vxx_q\tran}$. Therefore $\mu (\trace{\MJ \vxx\vxx\tran}-1)$ is the same as $\mu ( \trace{\vxx_q \vxx_q\tran} - 1)$, and $\mu$ is the Lagrange multiplier associated to the constraint $\trace{\vxx_q \vxx_q\tran} = 1$ in ($P$). 
Similarly, from the definition of the matrix $\MLambda_i$ in~\eqref{eq:matrixLambda}, 
it follows:
\bea
\trace{\MLambda_i \vxx\vxx\tran} = \trace{\MLambda_{ii} (\vxx_{q_i}\vxx_{q_i}\tran - \vxx_q \vxx_q\tran)},
\eea
where $\MLambda_{ii}$ is the Lagrange multiplier (matrix) associated to each of the constraints $\vxx_{q_i}\vxx_{q_i}\tran = \vxx_q \vxx_q\tran$ in ($P$).
This proves that~\eqref{eq:LagrangianNaive0} (and eq.~\eqref{eq:LagrangianNaive}, which rewrites~\eqref{eq:LagrangianNaive0} in compact form) is the Lagrangian function of ($P$).
\end{proof}

From the expression of the Lagrangian, we can readily obtain the Lagrangian \emph{dual} problem.

\begin{proposition}[Lagrangian Dual of Primal \QCQP]\label{prop:LagrangianDualPrimal}
The following SDP is the Lagrangian dual for the primal \QCQP (P) in eq.~\eqref{eq:qcqpforLagrangian}:
\bea \label{eq:LagrangianDualNaive}
(D)\quad \max_{\substack{\mu \in \Real{} \\ \MLambda \in \Real{4(N+1)\times 4(N+1)}} } & \mu \\
\subject & \MQ - \mu\MJ - \MLambda \succeq 0 \nonumber 
\eea
where $\MJ$ and $\MLambda$ satisfy the structure in eq.~\eqref{eq:matrixJ} and~\eqref{eq:matrixLambda}.
\end{proposition}

\begin{proof}
By definition, the dual problem is~\cite{Boyd04book}:
\bea
\hspace{-3mm}\max_{\mu,\MLambda}\min_{\vxx}\calL(\vxx,\mu,\MLambda), 
\eea
where $\calL(\vxx,\mu,\MLambda)$ is the Lagrangian function. We observe:
\bea
\hspace{-3mm}\max_{\mu,\MLambda}\min_{\vxx}\calL(\vxx,\mu,\MLambda) = \begin{cases}
\mu & \text{if } \MQ - \mu\MJ - \MLambda \succeq 0 \\
-\infty & \text{otherwise}
\end{cases}.
\eea
Since we are trying to maximize the Lagrangian (with respect to the dual variables), we discard the case 
leading to a cost of $-\infty$, obtaining the dual problem in~\eqref{eq:LagrangianDualNaive}.
\end{proof}

To connect the Lagrangian dual ($D$) to the naive SDP relaxation~\eqref{eq:naiveRelaxation} in Proposition~\ref{prop:naiveRelax}, we notice that the naive SDP relaxation is the dual SDP of the Lagrangian dual ($D$).
\begin{proposition}[Naive Relaxation is the Dual of the Dual]\label{prop:naiveRelaxIsDualofLagrangianDual}
The following SDP is the dual of the Lagrangian dual (D) in~\eqref{eq:LagrangianDualNaive}:
\bea
(DD)\quad \min_{\MZ \succeq 0}  & \trace{\MQ \MZ} \\
\subject  & \trace{[\MZ]_{qq}} = 1 \nonumber \\
& [\MZ]_{q_i q_i} = [\MZ]_{qq}, \forall i=1,\dots,N \nonumber
\eea
and (DD) is the same as the naive SDP relaxation in~\eqref{eq:naiveRelaxation}.
\end{proposition}
\begin{proof} 
We derive the Lagrangian dual problem of $(DD)$ and show that it is indeed ($D$) (see similar example in~\cite[p. 265]{Boyd04book}). Similar to the proof of Proposition~\eqref{prop:LagrangianPrimalQCQP}, we can associate Lagrangian multiplier $\mu\MJ$ (eq.~\eqref{eq:matrixJ}) to the constraint $\trace{[\MZ]_{qq}}=1$, and associate $\MLambda_i,i=1,\dots,N$ (eq.~\eqref{eq:matrixLambdai}) to constraints $[\MZ]_{q_i q_i} = [\MZ]_{qq},i=1,\dots,N$. In addition, we can associate matrix $\MTheta \in \sym^{4(N+1)\times 4(N+1)}$ to the constraint $\MZ \succeq 0$. Then the Lagrangian of the SDP ($DD$) is:
\bea
& \calL(\MZ,\mu,\MLambda,\MTheta)  \nonumber \\
& \hspace{-5mm}= \trace{\MQ \MZ} - \mu(\trace{\MJ \MZ} - 1) - \sum\limits_{i=1}^N (\trace{\MLambda_i \MZ}) - \trace{\MTheta \MZ} \nonumber \\
& = \trace{(\MQ - \mu \MJ - \MLambda - \MTheta) \MZ} + \mu,
\eea
and by definition, the dual problem is:
\bea
\max_{\mu,\MLambda,\MTheta}\min_{\MZ} \calL(\MZ,\mu,\MLambda,\MTheta).
\eea
Because:
\bea
\max_{\mu,\MLambda,\MTheta}\min_{\MZ} \calL = \begin{cases}
\mu & \text{if } \MQ - \mu\MJ - \MLambda - \MTheta \succeq 0 \\
-\infty & \text{otherwise}
\end{cases},
\eea
we can get the Lagrangian dual problem of ($DD$) is:
\bea \label{eq:LagrangianDualofDD}
\max_{\mu,\MLambda,\MTheta} & \mu \\
\subject & \MQ - \mu\MJ - \MLambda \succeq \MTheta \nonumber \\
& \MTheta \succeq 0 \nonumber 
\eea
Since $\MTheta$ is independent from the other decision variables and the cost function, we inspect that setting $\MTheta = \zero$ actually maximizes $\mu$ and therefore can be removed. Removing $\MTheta$ from~\eqref{eq:LagrangianDualofDD} indeed leads to ($D$) in eq.~\eqref{eq:LagrangianDualNaive}.
\end{proof}

We can also verify weak duality by the following calculation. Denote $f_{DD}=\trace{\MQ\MZ}$ and $f_D=\mu$. Recalling the structure of $\MLambda$ from eq.~\eqref{eq:matrixLambda}, we have $\trace{\MLambda \MZ}  = 0 $ because $[\MZ]_{q_i q_i}=[\MZ]_{qq},\forall i=1,\dots,N$. Moreover, we have $\mu=\mu\trace{\MJ \MZ}$ due to the pattern of $\MJ$ from eq.~\eqref{eq:matrixJ} and $\trace{[\MZ]_{qq}}=1$. Therefore, the following inequality holds:
\bea \label{eq:weakDualitySDP}
& f_{DD}-f_D=\trace{\MQ\MZ} - \mu \nonumber \\
& = \trace{\MQ\MZ} - \mu\trace{\MJ \MZ} - \trace{\MLambda \MZ} \nonumber \\
& = \trace{(\MQ-\mu\MJ-\MLambda)\MZ} \geq 0
\eea
where the last inequality holds true because both $\MQ-\mu\MJ-\MLambda$ and $\MZ$ are positive semidefinite matrices. Eq.~\eqref{eq:weakDualitySDP} shows $f_{DD} \geq f_D$ always holds inside the feasible set and therefore by construction of ($P$), ($D$) and ($DD$), we have the following \emph{weak duality} relation:
\bea \label{eq:weakDualityByConstruction}
f^\star_D \leq f^\star_{DD} \leq f^\star_P.
\eea
where the first inequality follows from eq.~\eqref{eq:weakDualitySDP} and the second inequality originates from the point that ($DD$) is a convex relaxation of ($P$), which has a larger feasible set and therefore the optimal cost of ($DD$) ($f^\star_{DD}$) is always smaller than the optimal cost of ($P$) ($f^\star_P$).

\subsection{KKT conditions and strong duality}
\label{sec:KKTStrongDuality}
Despite the fact that weak duality gives lower bounds for objective of the primal \QCQP ($P$), in this context we are interested in cases when \emph{strong duality} holds, \ie:
\bea \label{eq:strongdualityDef}
f^\star_D = f^\star_{DD} = f^\star_P,
\eea
since in these cases solving any of the two convex SDPs ($D$) or ($DD$) will also solve the original non-convex \QCQP ($P$) to \emph{global optimality}.

Before stating the main theorem for strong duality, we study the \emph{Karush-Kuhn-Tucker} (KKT) conditions~\cite{Boyd04book} for the primal \QCQP ($P$) in~\eqref{eq:qcqpforLagrangian}, which will help pave the way to study strong duality.
\begin{proposition}[KKT Conditions for Primal \QCQP]\label{prop:KKTForPrimalQCQP}
If $\vxx^\star$ is an optimal solution to the primal \QCQP ($P$) in~\eqref{eq:qcqpforLagrangian} (also~\eqref{eq:TLSBinaryClone}), and let $(\mu^\star,\MLambda^\star)$ be the corresponding optimal dual variables (maybe not unique), then it must satisfy the following KKT conditions:
\bea
& \expl{Stationary condition} 
& \label{eq:KKTStationary}(\MQ - \mu^\star \MJ - \MLambda^\star) \vxx^\star = \zero, \\
& \expl{Primal feasibility condition} 
& \label{eq:KKTFeasibility} \vxx^\star \text{ satisfies the constraints in~\eqref{eq:qcqpforLagrangian} }.
\eea
\end{proposition}
Using Propositions~\ref{prop:primalQCQP}-\ref{prop:KKTForPrimalQCQP}, we state the following theorem that provides a sufficient condition for strong duality. 
\begin{theorem}[Sufficient Condition for Strong Duality]\label{thm:generalStrongDuality}
Given a stationary point $\vxx^\star$, if there exist dual variables $(\mu^\star,\MLambda^\star)$ (maybe not unique) such that $(\vxx^\star,\mu^\star,\MLambda^\star)$ satisfy both the KKT conditions in Proposition~\ref{prop:KKTForPrimalQCQP} and the dual feasibility condition $\MQ - \mu^\star\MJ - \MLambda^\star \succeq 0$ in Proposition~\ref{prop:LagrangianDualPrimal}, then:
\begin{enumerate}[label=(\roman*)]
\item \label{item:strongDuality1} There is no duality gap between (P), (D) and (DD), \ie $f^\star_{P}=f^\star_{D}=f^\star_{DD}$,
\item \label{item:strongDuality2} $\vxx^\star$ is a global minimizer for (P).
\end{enumerate}
Moreover, if we have $\rank{\MQ - \mu^\star\MJ - \MLambda^\star}=4(N+1)-1$, \ie, $\MQ - \mu^\star\MJ - \MLambda^\star$ has $4(N+1)-1$ strictly positive eigenvalues and only one zero eigenvalue, then we have the following:
\begin{enumerate}[resume,label=(\roman*)]
\item \label{item:strongDuality3} $\pm \vxx^\star$ are the two unique global minimizers for (P),
\item \label{item:strongDuality4} The optimal solution to (DD), denoted as $\MZ^\star$, has rank 1 and can be written as $\MZ^\star=(\vxx^\star)(\vxx^\star)\tran$.
\end{enumerate}
\end{theorem}
\begin{proof}
Recall from eq.~\eqref{eq:weakDualityByConstruction} that we already have weak duality by construction of ($P$), ($D$) and ($DD$). Now since $(\vxx^\star,\mu^\star,\MLambda^\star)$ satisfies the KKT conditions~\eqref{eq:KKTFeasibility} and~\eqref{eq:KKTStationary}, we have:
\bea
& (\MQ - \mu^\star \MJ - \MLambda^\star)\vxx^\star = \zero \Rightarrow \nonumber \\
& (\vxx^\star)\tran(\MQ - \mu^\star\MJ - \MLambda^\star)(\vxx^\star) = 0 \Rightarrow \\
& \hspace{-8mm}(\vxx^\star)\tran \MQ (\vxx^\star) = \mu^\star (\vxx^\star)\tran \MJ (\vxx^\star) + (\vxx^\star)\tran \MLambda^\star (\vxx^\star) \Rightarrow \\
& \expl{$\vxstar$ satisfies the constraints in ($P$) by KKT~\eqref{eq:KKTFeasibility}}
& \expl{Recall structural partition of $\MJ$ and $\MLambda$ in~\eqref{eq:matrixJ} and~\eqref{eq:matrixLambda}}
& \trace{\MQ (\vxstar)(\vxstar)\tran} = \mu^\star,
\eea
which shows the cost of ($P$) is equal to the cost of ($D$) at $(\vxstar, \mu^\star, \MLambda^\star)$. Moreover, since $\MQ - \mu^\star\MJ - \MLambda^\star \succeq 0$ means $(\mu^\star,\MLambda^\star)$ is actually dual feasible for ($D$), hence we have strong duality between ($P$) and ($D$):
$f^\star_P = f^\star_D$. Because $f^\star_{DD}$ is sandwiched between $f^\star_P$ and $f^\star_D$ according to~\eqref{eq:weakDualityByConstruction}, we have indeed strong duality for all of them:
\bea
f^\star_D = f^\star_{DD} = f^\star_P,
\eea
proving~\ref{item:strongDuality1}. To prove~\ref{item:strongDuality2}, we observe that for any $\vxx \in \Real{4(N+1)}$, $\MQ - \mu^\star\MJ - \MLambda^\star \succeq 0$ means:
\bea
\vxx\tran (\MQ - \mu^\star\MJ - \MLambda^\star) \vxx \geq 0.
\eea
Specifically, let $\vxx$ be any vector that lies inside the feasible set of ($P$), \ie, $\trace{\vxx_q \vxx_q\tran}=1$ and $\vxx_{q_i}\vxx_{q_i}\tran = \vxx_q \vxx_q\tran,\forall i=1,\dots,N$, then we have:
\bea
& \vxx\tran (\MQ - \mu^\star\MJ - \MLambda^\star) \vxx \geq 0 \Rightarrow \nonumber \\
& \vxx\tran \MQ \vxx \geq \mu^\star \vxx\tran \MJ \vxx + \vxx\tran \MLambda^\star \vxx \Rightarrow \\
& \trace{\MQ \vxx \vxx\tran} \geq \mu^\star = \trace{\MQ (\vxstar)(\vxstar)\tran},
\eea
showing that the cost achieved by $\vxstar$ is no larger than the cost achieved by any other vectors inside the feasible set, which means $\vxstar$ is indeed a global minimizer to ($P$).

Next we use the additional condition of $\rank{\MQ - \mu^\star\MJ - \MLambda^\star}=4(N+1)-1$ to prove $\pm \vxstar$ are the two unique global minimizers to ($P$). Denote $\MM^\star=\MQ - \mu^\star\MJ - \MLambda^\star$, since $\MM^\star$ has only one zero eigenvalue with associated eigenvector $\vxstar$ (\cf KKT condition~\eqref{eq:KKTStationary}), its nullspace is defined by $\ker(\MM^\star) = \{\vxx \in \Real{4(N+1)}: \vxx = a \vxstar ,a\in \Real{} \}$. Now denote the feasible set of ($P$) as $\Omega(P)$. It is clear to see that any vector in $\Omega(P)$ is a vertical stacking of $N+1$ unit quaternions and thus must have 2-norm equal to $\sqrt{N+1}$. Since $\vxstar \in \Omega(P)$ is already true, in order for any vector $\vxx=a\vxstar$ in $\ker(\MM^\star)$ to be in $\Omega(P)$ as well, it must hold $|a|\Vert \vxx \Vert = \sqrt{N+1}$ and therefore $a=\pm 1$, \ie, $\ker(\MM^\star)\cap \Omega(P) = \{\pm\vxstar \}$. With this observation, we can argue that for any $\vxx$ inside $\Omega(P)$ that is not equal to $\{\pm \vxstar \}$, $\vxx$ cannot be in $\ker(\MM^\star)$ and therefore:
\bea
& \vxx\tran (\MM^\star) \vxx > 0 \Rightarrow \\
& \vxx\tran \MQ \vxx > \mu^\star \vxx\tran \MJ \vxx + \vxx\tran \MLambda^\star \vxx \Rightarrow \\
& \trace{\MQ \vxx \vxx\tran} > \mu^\star = \trace{\MQ(\vxstar)(\vxstar)\tran},
\eea
which means for any vector $\vxx \in \Omega(P)/ \{\pm \vxstar \} $, it results in strictly higher cost than $\pm \vxstar$. Hence $\pm \vxstar$ are the two unique global minimizers to ($P$) and~\ref{item:strongDuality3} is true.

To prove~\ref{item:strongDuality4}, notice that since strong duality holds and $f^\star_{DD}=f^\star_D$, we can write the following according to eq.~\eqref{eq:weakDualitySDP}:
\bea \label{eq:strongDualitySDP}
\trace{(\MQ - \mu^\star \MJ - \MLambda^\star) \MZ^\star} = 0.
\eea
Since $\MM^\star = \MQ - \mu^\star \MJ - \MLambda^\star \succeq 0$ and has rank $4(N+1)-1$, we can write $\MM^\star = \barM\tran \barM$ with $\barM \in \Real{(4(N+1)-1) \times 4(N+1)}$ and $\rank{\barM} = 4(N+1)-1$. Similarly, we can write $\MZ^\star = \barZ \barZ\tran$ with $\barZ \in \Real{4(N+1) \times r}$ and $\rank{\barZ}=r=\rank{\MZ^\star}$. Then from~\eqref{eq:strongDualitySDP} we  have:
\bea
& \trace{\MM^\star \MZ^\star} = \trace{\barM\tran \barM \barZ \barZ\tran} \nonumber \\
& = \trace{\barZ\tran\barM\tran \barM\barZ} = \trace{(\barM\barZ)\tran (\barM \barZ)} \nonumber \\
& = \Vert \barM \barZ \Vert_F^2 = 0,
\eea
which gives us $\barM \barZ = \zero$. Using the rank inequality $\rank{\barM \barZ} \geq \rank{\barM} + \rank{\barZ} - 4(N+1)$, we have:
\bea
& 0 \geq 4(N+1)-1 + r - 4(N+1) \Rightarrow \nonumber \\
& r \leq 1.
\eea
Since $\barZ \neq \zero$, we conclude that $\rank{\MZ^\star} = \rank{\barZ}=r=1$. As a result, since $\rank{\MZ^\star} = 1$, and the rank constraint was the only constraint we dropped when relaxing the \QCQP ($P$) to SDP ($DD$), we conclude that the relaxation is indeed tight. In addition, the rank 1 decomposition $\barZ$ of $\MZ^\star$ is also the global minimizer to ($P$). However, from~\ref{item:strongDuality3}, we know there are only two global minimizers to ($P$): $\vxstar$ and $-\vxstar$, so $\barZ \in \{\pm \vxstar\}$. Since the sign is irrelevant, we can always write $\MZ^\star = (\vxstar) (\vxstar)\tran$, concluding the proof for~\ref{item:strongDuality4}. 
\end{proof}

\subsection{Strong duality in noiseless and outlier-free case}
\label{sec:strongDualityNoiselessOutlierfree}

Now we are ready to prove Theorem~\ref{thm:strongDualityNoiseless} using Theorem~\ref{thm:generalStrongDuality}. To do so, we will show that in the noiseless and outlier-free case, it is always possible to construct $\mu^\star$ and $\MLambda^\star$ from $\vxstar$ and $\MQ$ such that $(\vxstar,\mu^\star,\MLambda^\star)$ satisfies the KKT conditions, and the dual matrix $\MM^\star = \MQ - \mu^\star\MJ - \MLambda^\star$ is positive semidefinite and has only one zero eigenvalue. 

\myParagraph{Preliminaries} When there are no noise and outliers in the measurements, \ie, $\vb_i = \MR \va_i, \forall i=1,\dots,N$, we have $\Vert \vb_i \Vert^2 = \Vert \va_i \Vert^2,\forall i=1,\dots,N$. Moreover, without loss of generality, we assume $\sigma_i^2=1$ and $\barcsq > 0$. 
With these assumptions, we simplify the blocks $\MQ_{0i}$ and $\MQ_{ii}$ in the matrix $\MQ$, \cf~\eqref{eq:matrixCompactQ} and~\eqref{eq:costInQuadraticFormBeforeLifting}:
\bea
& \label{eq:simplifiedQ0i} \MQ_{0i} = \frac{\Vert \va_i \Vert^2}{2}\eye_4 + \frac{\MOmega_1(\hatvb_i) \MOmega_2(\hatva_i)}{2} - \frac{\barcsq}{4}\eye_4,\\
& \label{eq:simplifiedQii} \MQ_{ii} = \Vert \va_i \Vert^2 \eye_4 + \MOmega_1(\hatvb_i)\MOmega_2(\hatva_i) + \frac{\barcsq}{2}\eye_4.
\eea

Due to the primal feasibility condition~\eqref{eq:KKTFeasibility}, we know $\vxx^\star$ can be written as $N+1$ quaternions (\cf proof of~\eqref{eq:equivalentConstraints}): $\vxx^\star= [(\vq^\star)\tran\ \theta_1^\star (\vq^\star)\tran\ \dots\ \theta_N^\star(\vq^\star)\tran]\tran$, where each $\theta_i^\star$ is a binary variable in $\{-1,+1\}$. Since we have assumed no noise and no outliers, we know $\theta_i^\star=+1$ for all $i$'s and therefore $\vxx^\star= [(\vq^\star)\tran\ (\vq^\star)\tran\ \dots\ (\vq^\star)\tran]\tran$. We can write the KKT stationary condition in matrix form as:
\bea \label{eq:KKTStationaryMatrixForm}
\hspace{-7mm}\overbrace{ \bmat{c|ccc}
\substack{-\mu^\star \eye_4 - \\ \sum\limits_{i=1}^N\MLambda^\star_{ii} }& \MQ_{01} & \dots & \MQ_{0N} \\
\hline
\MQ_{01} & {\scriptstyle \MQ_{11}+\MLambda^\star_{11} } & \dots & \MZero \\
\vdots & \vdots & \ddots & \vdots \\
\MQ_{0N} & \MZero & \dots & {\scriptstyle \MQ_{NN}+\MLambda^\star_{NN}}
\emat}^{\MM^\star = \MQ - \mu^\star\MJ - \MLambda^\star}
\overbrace{ \bmat{c}
\vq^\star \\
\hline
\vq^\star \\
\vdots \\
\vq^\star
\emat}^{\vxstar} = \zero \nonumber \hspace{-8mm}\\\hspace{-8mm}
\eea
and we index the block rows of $\MM^\star$ from top to bottom as $0,1,\dots,N$. The first observation we make is that eq.~\eqref{eq:KKTStationaryMatrixForm} is a (highly) under-determined linear system with respect to the dual variables $(\mu^\star,\MLambda^\star)$, because the linear system has $10N+1$ unknowns  (each symmetric $4\times4$ matrix $\MLambda^\star_{ii}$ has 10 unknowns, plus one unknown from $\mu^\star$), but only has $4(N+1)$ equations. To expose the structure of the linear system, we will apply a similarity transformation to the matrix $\MM^\star$. Before we introduce the similarity transformation, we need additional properties about quaternions, described in the Lemma below. The properties can be proven by inspection.

\begin{lemma}[More Quaternion Properties]\label{lemma:quaternionProperties}
The following properties about unit quaternions, involving the linear operators $\MOmega_1(\cdot)$ and $\MOmega_2(\cdot)$ introduced in eq.~\eqref{eq:Omega_1} hold true: 
\begin{enumerate}[label=(\roman*)]
\item \label{item:commutativeOmega} Commutative: for any two vectors $\vxx,\vy \in \Real{4}$, The following equalities hold:
\bea \label{eq:commutativeOmega}
\MOmega_1(\vxx)\MOmega_2(\vy) = \MOmega_2(\vy)\MOmega_1(\vxx); \\
\MOmega_1(\vxx)\MOmega_2\tran(\vy) = \MOmega_2\tran(\vy)\MOmega_1(\vxx); \\
\MOmega_1\tran(\vxx)\MOmega_2(\vy) = \MOmega_2(\vy)\MOmega_1\tran(\vxx); \\
\MOmega_1\tran(\vxx)\MOmega_2\tran(\vy) = \MOmega_2\tran(\vy)\MOmega_1\tran(\vxx). 
\eea
\item \label{item:orthogonality} Orthogonality: for any unit quaternion $\vq \in \calS^3$, $\MOmega_1(\vq)$ and $\MOmega_2(\vq)$ are orthogonal matrices:
\bea
\MOmega_1(\vq)\MOmega_1\tran(\vq) = \MOmega_1\tran(\vq)\MOmega_1(\vq) = \eye_4; \\
\MOmega_2(\vq)\MOmega_1\tran(\vq) = \MOmega_2\tran(\vq)\MOmega_1(\vq) = \eye_4.
\eea
\item \label{item:rotateqtoe} For any unit quaternion $\vq \in \calS^3$, the following equalities hold:
\bea
\MOmega_1\tran(\vq)\vq = \MOmega_2\tran(\vq)\vq = [0,0,0,1]\tran.
\eea
\item \label{item:OmegaandR} For any unit quaternion $\vq \in \calS^3$, denote $\MR$ as the unique rotation matrix associated with $\vq$, then the following equalities hold:
\bea
\hspace{-10mm} \MOmega_1(\vq)\MOmega_2\tran(\vq)=\MOmega_2\tran(\vq)\MOmega_1(\vq)=\bmat{cc} \MR & \zero \\ \zero & 1 \emat \doteq \tilde{\MR}; \\
\hspace{-10mm} \MOmega_2(\vq)\MOmega_1\tran(\vq) = \MOmega_1\tran(\vq)\MOmega_2(\vq) = \bmat{cc} \MR\tran & \zero \\ \zero & 1 \emat \doteq \tilde{\MR}\tran.
\eea
\end{enumerate}
\end{lemma}

\myParagraph{Rewrite dual certificates using similarity transform}
Now we are ready to define the similarity transformation. We define the matrix $\MD \in \Real{4(N+1)\times 4(N+1)}$ as the following block diagonal matrix:
\bea
& \MD = \bmat{cccc}
\MOmega_1(\vq^\star) & \zero & \cdots & \zero \\
\zero & \MOmega_1(\vq^\star) & \cdots & \zero \\
\vdots & \vdots & \ddots & \vdots \\
\zero & \zero & \cdots & \MOmega_1(\vq^\star)
\emat.
\eea
It is obvious to see that $\MD$ is an orthogonal matrix from~\ref{item:orthogonality} in Lemma~\ref{lemma:quaternionProperties}, \ie, $\MD\tran \MD = \MD \MD\tran = \eye_{4(N+1)}$. Then we have the following Lemma.
\begin{lemma}[Similarity Transformation]\label{lemma:similarityTransform}
Define $\MN^\star \doteq \MD\tran \MM^\star \MD$, then:
\begin{enumerate}[label=(\roman*)]
\item \label{item:sameEigenvalues} $\MN^\star$ and $\MM^\star$ have the same eigenvalues, and
\bea
\hspace{-5mm}\MM^\star \succeq 0 \Leftrightarrow \MN^\star \succeq 0,\ \rank{\MM^\star}=\rank{\MN^\star}.
\eea
\item \label{item:rotatedq} Define $\ve = [0,0,0,1]\tran$ and $\vr = [\ve\tran\ \ve\tran\ \dots\ \ve\tran]\tran$ as the vertical stacking of $N+1$ copies of $\ve$, then:
\bea
\MM^\star \vxstar=\zero \Leftrightarrow \MN^\star \vr = \zero.
\eea
\end{enumerate}
\end{lemma}
\begin{proof} Because $\MD\tran \MD = \eye_{4(N+1)}$, we have $\MD\tran = \MD\inv$ and $\MN^\star = \MD\tran \MM^\star \MD = \MD\inv \MM\tran \MD$ is similar to $\MM^\star$. Therefore, by matrix similarity, $\MM^\star$ and $\MN^\star$ have the same eigenvalues, and $\MM^\star$ is positive semidefinite if and only if $\MN^\star$ is positive semidefinite~\cite[p. 12]{Boyd06notes}. To show~\ref{item:rotatedq}, we start by pre-multiplying both sides of eq.~\eqref{eq:KKTStationaryMatrixForm} by $\MD\tran$:
\bea
& \MM^\star \vxstar =\zero \Leftrightarrow \MD\tran\MM^\star \vxstar = \MD\tran \zero \Leftrightarrow \\
& \expl{$\MD \MD\tran=\eye_{4(N+1)}$}
& \MD\tran\MM^\star (\MD \MD\tran) \vxstar = \zero \Leftrightarrow \\
& (\MD\tran\MM^\star\MD)(\MD\tran\vxstar) = \zero \Leftrightarrow \\
& \expl{$\MOmega_1\tran(\vq^\star) \vq^\star = \ve$ from~\ref{item:rotateqtoe} in Lemma~\ref{lemma:quaternionProperties}}
& \MN^\star \vr = \zero, \label{eq:KKTStationaryN}
\eea
concluding the proof.
\end{proof}
Lemma~\ref{lemma:similarityTransform} suggests that constructing $\MM^\star \succeq 0$ and $\rank{\MM^\star}=4(N+1)-1$ that satisfies the KKT conditions~\eqref{eq:KKTStationary} is equivalent to constructing $\MN^\star \succeq 0$ and $\rank{\MN^\star}=4(N+1)-1$ that satisfies~\eqref{eq:KKTStationaryN}. We then study the structure of $\MN^\star$ and rewrite the KKT stationary condition.

\myParagraph{Rewrite KKT conditions}
In noiseless and outlier-free case, the KKT condition $\MM^\star \vxstar = \zero$ is equivalent to $\MN^\star\vr =\zero$. Formally, after the similarity transformation $\MN^\star = \MD\tran \MM^\star \MD$, the KKT conditions can be explicitly rewritten as in the following proposition.

\begin{proposition}[KKT conditions after similarity transformation]
The KKT condition $\MN^\star \vr = \zero$ (which is equivalent to eq.~\eqref{eq:KKTStationaryMatrixForm}) can be written in matrix form:
\bea \label{eq:KKTStationaryMatrixFormInN}
\hspace{-8mm}\bmat{c|ccc}
 {\scriptstyle -\sum\limits_{i=1}^N\barLambda^\star_{ii} } & \barMQ_{01} & \dots & \barMQ_{0N} \\
\hline
\barMQ_{01} & {\scriptstyle \barMQ_{11}+\barLambda^\star_{11} } & \dots & \MZero \\
\vdots & \vdots & \ddots & \vdots \\
\barMQ_{0N} & \MZero & \dots & {\scriptstyle \barMQ_{NN}+\barLambda^\star_{NN}}
\emat
\bmat{c}
\ve \\
\hline 
\ve \\
\vdots \\
\ve
\emat = \zero.
\eea
with $\mu^\star = 0$ being removed compared to eq.~\eqref{eq:KKTStationaryMatrixForm} and $\barMQ_{0i}$ and $\barMQ_{ii},i=1,\dots,N$ are the following sparse matrices:
\bea
& \barMQ_{0i} \doteq \MOmega_1\tran(\vq^\star)\MQ_{0i}\MOmega_1(\vq^\star) \nonumber \\
& \hspace{-5mm} = \bmat{cc}
\left( \frac{\Vert \va_i \Vert^2}{2} - \frac{\barcsq}{4} \right) \eye_3 - \frac{ \vSkew{\va_i}^2}{2} - \frac{\va_i \va_i\tran}{2} & \zero \\
\zero & - \frac{\barcsq}{4}
\emat; \label{eq:barMQ0i}\\
& \barMQ_{ii} \doteq \MOmega_1\tran(\vq^\star)\MQ_{ii}\MOmega_1(\vq^\star) \nonumber \\
& \hspace{-5mm} = \bmat{cc}
\left( \Vert \va_i \Vert^2 + \frac{\barcsq}{2} \right)\eye_3 - \vSkew{\va_i}^2 - \va_i \va_i\tran & \zero \\
\zero & \frac{\barcsq}{2}
\emat, \label{eq:barMQii}
\eea
and $\barLambda_{ii}^\star \doteq \MOmega_1\tran(\vq^\star) \MLambda^\star_{ii} \MOmega_1(\vq^\star) $ has the following form:
\bea \label{eq:generalFormBarLambda}
\barLambda^\star_{ii} = \bmat{cc}
\ME_{ii} & \valpha_i \\
\valpha_i\tran & \lambda_i
\emat,
\eea
where $\ME_{ii} \in \sym^{3\times3}$,$\valpha_i \in \Real{3}$ and $\lambda_i \in \Real{}$.
\end{proposition}
\begin{proof}
We first prove that $\barMQ_{0i}$ and $\barMQ_{ii}$ have the forms in~\eqref{eq:barMQ0i} and~\eqref{eq:barMQii} when there are no noise and outliers in the measurements. Towards this goal, we examine the similar matrix to $\MOmega_1(\hatvb_i)\MOmega_2(\hatva_i)$ (as it is a common part to $\MQ_{0i}$ and $\MQ_{ii}$):
\bea
& \MOmega_1\tran(\vq^\star) \MOmega_1(\hatvb_i)\MOmega_2(\hatva_i) \MOmega_1(\vq^\star) \nonumber \\
& \expl{Commutative property in Lemma~\ref{lemma:quaternionProperties}~\ref{item:commutativeOmega}}
& = \MOmega_1\tran(\vq^\star) \MOmega_1(\hatvb_i) \MOmega_1(\vq^\star) \MOmega_2(\hatva_i) \\
& \expl{Orthogonality property in Lemma~\ref{lemma:quaternionProperties}~\ref{item:orthogonality}}
& \hspace{-5mm}= \MOmega_1\tran(\vq^\star) \MOmega_2(\vq^\star) \MOmega_2\tran(\vq^\star)  \MOmega_1(\hatvb_i) \MOmega_1(\vq^\star) \MOmega_2(\hatva_i) \\
& \expl{Lemma~\ref{lemma:quaternionProperties}~\ref{item:commutativeOmega} and~\ref{item:OmegaandR}}
& = (\tilde{\MR}^\star)\tran \MOmega_1(\hatvb_i)\MOmega_2\tran(\vq^\star)\MOmega_1(\vq^\star)\MOmega_2(\hatva_i) \\
& \expl{Lemma~\ref{lemma:quaternionProperties}~\ref{item:OmegaandR} }
& = (\tilde{\MR}^\star)\tran \MOmega_1(\hatvb_i) (\tilde{\MR}^\star) \MOmega_2(\hatva_i) \\
& = \bmat{cc}
(\MR^\star)\tran \vSkew{\vb_i}\MR^\star & (\MR^\star)\tran \vb_i \\
-\vb_i\tran \MR^\star & 0
\emat \MOmega_2(\hatva_i) \\
& = \bmat{cc}
\vSkew{\va_i} & \va_i \\
-\va_i\tran & 0 
\emat
\bmat{cc}
-\vSkew{\va_i} & \va_i \\
-\va_i\tran & 0
\emat \\
&=\bmat{cc}
-\vSkew{\va_i}^2-\va_i\va_i\tran & \zero \\
\zero & -\Vert \va_i \Vert^2
\emat.
\eea
Using this property, and recall the definition of $\MQ_{0i}$ and $\MQ_{ii}$ in eq.~\eqref{eq:simplifiedQ0i} and~\eqref{eq:simplifiedQii}, the similar matrices to $\MQ_{0i}$ and $\MQ_{ii}$ can be shown to have the expressions in~\eqref{eq:barMQ0i} and~\eqref{eq:barMQii} by inspection. 

Showing $\barLambda_{ii}$ having the expression in~\eqref{eq:generalFormBarLambda} is straightforward. Since $\MLambda_{ii}^\star$ is symmetric, $\barLambda_{ii}^\star= \MOmega_1\tran(\vq^\star) \MLambda^\star_{ii} \MOmega_1(\vq^\star)$ must also be symmetric and therefore eq.~\eqref{eq:generalFormBarLambda} must be true for some $\ME_{ii}$, $\valpha_i$ and $\lambda_i$.

Lastly, in the noiseless and outlier-free case, $\mu^\star$ is zero due to the following:
\bea
& \mu^\star = \trace{\MQ(\vxstar)(\vxstar)\tran} \nonumber \\
& \expl{Recall $\MQ$ from eq.~\eqref{eq:matrixCompactQ} }
& = (\vq^\star)\tran \left( \sumAllPointsi (\MQ_{ii} + 2\MQ_{0i}) \right) (\vq^\star) \\
& \hspace{-5mm}= (\vq^\star)\tran \left( 2 \sumAllPointsi \Vert \va_i \Vert^2 +  \MOmega_1(\hatvb_i)\MOmega_2(\hatva_i)  \right) (\vq^\star) \\
& \expl{Recall eq.~\eqref{eq:developQuatProduct}} 
& = 2 \sumAllPointsi \Vert \va_i \Vert^2 - \hatvb_i\tran (\vq^\star \kron \hatva_i \kron (\vq^\star)\inv) \\
& = 2 \sumAllPointsi \Vert \va_i \Vert^2 - \vb_i\tran (\MR^\star\va_i) \\
& = 2 \sumAllPointsi \Vert \va_i \Vert^2 - \| \vb_i \|^2 = 0.
\eea
concluding the proof.
\end{proof}

\myParagraph{From KKT condition to sparsity pattern of dual variable} From the above proposition about the rewritten KKT condition~\eqref{eq:KKTStationaryMatrixFormInN}, we can claim the following sparsity pattern on the dual variables $\barLambda_{ii}$.
\begin{lemma}[Sparsity Pattern of Dual Variables] \label{lemma:sparsityPatternDualVar}
The KKT condition eq.~\eqref{eq:KKTStationaryMatrixFormInN} holds if and only if the dual variables $\{\barLambda_{ii}\}_{i=1}^N$ have the following sparsity pattern:
\bea \label{eq:sparsityPatternLambda}
\barLambda^\star_{ii} = \bmat{cc}
\ME_{ii} & \zero_3 \\
\zero_3 & -\frac{\barcsq}{4}
\emat,
\eea
\ie, $\valpha_i=0$ and $\lambda_i = -\frac{\barcsq}{4}$ in eq.~\eqref{eq:generalFormBarLambda} for every $i=1,\dots,N$.
\end{lemma}
\begin{proof}
We first proof the trivial direction ($\Leftarrow$). If $\barLambda^\star_{ii}$ has the sparsity pattern in eq.~\eqref{eq:sparsityPatternLambda}, then the product of the $i$-th block row of $\MN^\star$ ($i=1,\dots,N$) and $\vr$ writes (\cf eq.~\eqref{eq:KKTStationaryMatrixFormInN}):
\bea
& \left( \barMQ_{0i} + \barMQ_{ii} + \barLambda_{ii}^\star \right) \ve \nonumber \\
& \expl{Recall $\barMQ_{0i}$ and $\barMQ_{ii}$ from eq.~\eqref{eq:barMQ0i} and~\eqref{eq:barMQii}}
& = \bmat{cc}
\star & \zero_3 \\
\zero_3 & 0
\emat \bmat{c}
\zero_3 \\ 1 \emat = \zero_4, \label{eq:KKTEveryRowVanish}
\eea
which is equal to $\zero_4$ for sure. For the product of the $0$-th block row (the very top row) of $\MN^\star$ and $\vr$, we get:
\bea
& \left( \sumAllPointsi\barMQ_{0i} - \barLambda_{ii}^\star \right)\ve \nonumber \\
& = \bmat{cc}
\star & \zero_3 \\
\zero_3 & 0
\emat \bmat{c}
\zero_3 \\ 1 \emat = \zero_4,
\eea
which vanishes as well. Therefore, $\barLambda^\star_{ii}$ having the sparsity pattern in eq.~\eqref{eq:sparsityPatternLambda} provides a sufficient condition for KKT condition in eq.~\eqref{eq:KKTStationaryMatrixFormInN}. To show the other direction ($\Rightarrow$), first notice that eq.~\eqref{eq:KKTStationaryMatrixFormInN} implies eq.~\eqref{eq:KKTEveryRowVanish} holds true for all $i=1,\dots,N$ and in fact, eq.~\eqref{eq:KKTEveryRowVanish} provides the following equation for constraining the general form of $\barLambda^\star_{ii}$ in eq.~\eqref{eq:generalFormBarLambda}:
\bea
& \left( \barMQ_{0i} + \barMQ_{ii} + \barLambda_{ii}^\star \right) \ve \nonumber \\
& = \bmat{cc}
\star & \valpha_i \\
\valpha_i\tran & \lambda_i + \frac{\barcsq}{4}
\emat \bmat{c}
\zero_3 \\ 1 \emat = \zero_4,
\eea
which directly gives rise to:
\bea
\begin{cases}
\valpha_i = \zero_3 \\
\lambda_i + \frac{\barcsq}{4} = 0
\end{cases}.
\eea
and the sparsity pattern in eq.~\eqref{eq:sparsityPatternLambda}, showing that $\barLambda^\star_{ii}$ having the sparsity pattern in eq.~\eqref{eq:sparsityPatternLambda} is also a necessary condition.
\end{proof}

\myParagraph{Find the dual certificate}
Lemma~\ref{lemma:sparsityPatternDualVar} further suggests that the linear system resulted from KKT conditions~\eqref{eq:KKTStationaryMatrixFormInN} (also~\eqref{eq:KKTStationaryMatrixForm}) is highly under-determined in the sense that we have full freedom in choosing the $\ME_{ii}$ block of the dual variable $\barLambda_{ii}^\star$. Therefore, we introduce the following proposition.

\begin{proposition}[Construction of Dual Variable]\label{prop:constructionDualVariable}
In the noiseless and outlier-free case, choosing $\ME_{ii}$ as:
\bea \label{eq:choiceOfE}
\ME_{ii} = \vSkew{\va_k}^2 - \frac{1}{4} \barcsq \eye_3, \forall i =1,\dots,N
\eea
and choosing $\barLambda^\star_{ii}$ as having the sparsity pattern in~\eqref{eq:sparsityPatternLambda} will not only satisfy the KKT conditions in~\eqref{eq:KKTStationaryMatrixFormInN} but also make $\MN^\star \succeq 0$ and $\rank{\MN^\star}=4(N+1)-1$. Therefore, by Theorem~\ref{thm:generalStrongDuality}, the naive relaxation in Proposition~\ref{prop:naiveRelax} is always tight and Theorem~\ref{thm:strongDualityNoiseless} is true.
\end{proposition}
\begin{proof}
By Lemma~\ref{lemma:sparsityPatternDualVar}, we only need to prove the choice of $\ME_{ii}$ in~\eqref{eq:choiceOfE} makes $\MN^\star \succeq 0$ and $\rank{\MN^\star}=4(N+1)-1$. Towards this goal, we will show that for any vector $\vu \in \Real{4(N+1)}$, $\vu\tran \MN^\star \vu \geq 0$ and the nullspace of $\MN^\star$ is $\ker(\MN^\star)=\{\vu: \vu = a\vr, a\in \Real{} \text{ and } a\neq 0\}$ ($\MN^\star$ has a single zero eigenvalue with associated eigenvector $\vr$). We partition $\vu = [\vu_0\tran\ \vu_1\tran\ \dots\ \vu_N\tran]\tran$, where $\vu_i \in \Real{4},\forall i=0,\dots,N$. Then using the form of $\MN^\star$ in~\eqref{eq:KKTStationaryMatrixFormInN}, we can write $\vu\tran \MN^\star \vu$ as:
\bea
& \vu\tran \MN^\star \vu = - \sumAllPointsi \vu_0\tran \barLambda_{ii} \vu_0 + \nonumber \\
& 2\sumAllPointsi \vu_0\tran \barMQ_{0i} \vu_i + \sumAllPointsi \vu_i\tran (\barMQ_{ii}+\barLambda_{ii})\vu_i \\
& \hspace{-13mm}= \sum\limits_{i=1}^N \underbrace{ \vu_0\tran (-\barLambda_{ii})\vu_0 + \vu_0\tran (2 \barMQ_{0i})\vu_i + \vu_i\tran (\barMQ_{ii}+\barLambda_{ii})\vu_i}_{m_i}.
\eea
Further denoting $\vu_i = [\barvu_i\tran\ u_i]\tran$ with $\barvu_i \in \Real{3}$, $m_i$ can be written as:
\bea
& m_i = \bmat{c}
\barvu_0 \\ u_0
\emat\tran \bmat{cc}
-\ME_{ii} & \zero \\
\zero & \frac{\barcsq}{4}
\emat 
\bmat{c}
\barvu_0 \\ u_0
\emat + \nonumber \\
& \bmat{c}
\barvu_0 \\ u_0
\emat\tran 
\bmat{c|c}
\substack{ \left( \Vert \va_i \Vert^2 - \frac{\barcsq}{2} \right)\eye_3 \\ -\vSkew{\va_i}^2 - \va_i \va_i\tran } & \MZero \\
\hline
\MZero & - \frac{\barcsq}{2} 
\emat
\bmat{c}
\barvu_i \\ u_i
\emat + \nonumber \\
& \bmat{c}
\barvu_i \\ u_i
\emat\tran
\bmat{c|c}
\substack{ \left( \Vert \va_i \Vert^2 + \frac{\barcsq}{2} \right)\eye_3 \\ - \vSkew{\va_i}^2 - \va_i \va_i\tran + \ME_{ii}} & \MZero \\
\hline
\MZero & \frac{\barcsq}{4}
\emat
\bmat{c}
\barvu_i \\ u_i
\emat \\
& = \frac{\barcsq}{4}\underbrace{ \left( u_0^2 - 2u_0 u_i + u_i^2\right)}_{=(u_0-u_i)^2 \geq 0}+ \barvu_0\tran (-\ME_{ii}) \barvu_0 +  \nonumber \\
&\barvu_0\tran \left( \Vert \va_i \Vert^2 \eye_3 - \frac{\barcsq}{2}\eye_3 - \vSkew{\va_i}^2 - \va_i\va_i\tran \right) \barvu_i + \nonumber \\
& \hspace{-5mm}\barvu_i\tran \left(\Vert \va_i \Vert^2\eye_3 + \frac{\barcsq}{2} \eye_3 - \vSkew{\va_i}^2 - \va_i \va_i\tran + \ME_{ii} \right) \barvu_i, \label{eq:isolateScalar}
\eea
where equality holds only when $u_i=u_0$ in the underbraced inequality in~\eqref{eq:isolateScalar}. Now we insert the choice of $\ME_{ii}$ in eq.~\eqref{eq:choiceOfE} to $m_i$ to get the following inequality:
\bea
& m_i \geq \barvu_0\tran (-\vSkew{\va_i}^2)\barvu_0  + \frac{1}{4} \barcsq \barvu_0\tran \barvu_0 + \nonumber \\
& \hspace{-8mm}\Vert \va_i \Vert^2 \barvu_0\tran \barvu_i - \frac{\barcsq}{2}\barvu_0\tran \barvu_i + \barvu_0\tran (-\vSkew{\va_i}^2) \barvu_i - \barvu_0\tran \va_i \va_i\tran \barvu_i + \nonumber \\
&  \Vert \va_i \Vert^2 \barvu_i\tran \barvu_i + \frac{1}{4} \barcsq \barvu_i\tran \barvu_i - \barvu_i\tran \va_i \va_i\tran \barvu_i. \label{eq:afterInsertingEii}
\eea
Using the following facts:
\bea
& \Vert \va_i \Vert^2 \barvu_0\tran \barvu_i = \barvu_0\tran (-\vSkew{\va_i}^2 + \va_i\va_i\tran) \barvu_i; \\
& \Vert \va_i \Vert^2 \barvu_i\tran \barvu_i = \barvu_i\tran (-\vSkew{\va_i}^2 + \va_i\va_i\tran) \barvu_i, 
\eea
eq.~\eqref{eq:afterInsertingEii} can be simplified as:
\bea
& m_i \geq \barvu_0\tran (-\vSkew{\va_i}^2)\barvu_0  + \frac{1}{4} \barcsq \barvu_0\tran \barvu_0 + \nonumber \\
& 2\barvu_0\tran (-\vSkew{\va_i}^2)\barvu_i - \frac{\barcsq}{2} \barvu_0\tran \barvu_i + \nonumber \\
& \barvu_i\tran (-\vSkew{\va_i}^2) \barvu_i + \frac{1}{4} \barcsq \barvu_i\tran \barvu_i \\
& = \left( \vSkew{\va_i} \barvu_0 + \vSkew{\va_i} \barvu_i \right)\tran \left( \vSkew{\va_i} \barvu_0 + \vSkew{\va_i} \barvu_i \right) + \nonumber \\
&  \frac{\barcsq}{4} \left( \barvu_0\tran \barvu_0 - 2\barvu_0\tran \barvu_i  + \barvu_i\tran \barvu_i \right) \\
& = \Vert \va_i \times (\barvu_0 + \barvu_i) \Vert^2 +  \frac{ \barcsq}{4} \Vert  \barvu_0 -  \barvu_i \Vert^2  \nonumber \\
& \geq 0.
\eea
Since each $m_i$ is nonnegative, $\vu\tran \MN^\star \vu \geq 0$ holds true for any vector $\vu$ and therefore $\MN^\star \succeq 0$ is true. To see $\MN^\star$ only has one zero eigenvalue, we notice that $\vu\tran \MN^\star \vu = 0$ holds only when:
\bea
\begin{cases} \label{eq:zeroConditions}
u_i = u_0, \forall i=1,\dots,N \\
\barvu_0=\barvu_i, \forall i=1,\dots,N \\
\va_i \times (\barvu_0 + \barvu_i) = \zero_3,\forall i=1,\dots,N
\end{cases}
\eea
because we have more than two $\va_i$'s that are not parallel to each other, eq.~\eqref{eq:zeroConditions} leads to $\barvu_0=\barvu_i = \zero_3,\forall i=1,\dots,N$. Therefore the only set of nonzero vectors that satisfy the above conditions are $\{\vu \in \Real{4(N+1)}: \vu = a\vr,a\in\Real{} \text{ and }a\neq 0\}$. Therefore, $\MN^\star$ has only one zero eigenvalue and $\rank{\MN^\star}=4(N+1)-1$.
\end{proof} 
Proposition~\ref{prop:constructionDualVariable} indeed proves the original Theorem~\ref{thm:strongDualityNoiseless} by giving valid constructions of dual variables under which strong duality always holds in the noiseless and outlier-free case.
\end{proof}

\section{Proof of Proposition~\ref{prop:relaxationRedundant}}
\label{sec:proof:prop:relaxationRedundant}
\begin{proof}
Here we prove that eq.~\eqref{eq:relaxationRedundant} is a convex relaxation of eq.~\eqref{eq:qcqpZ}
and that the relaxation is always tighter, \ie the optimal objective  of~\eqref{eq:SDPrelax} is always closer to the optimal objective of~\eqref{eq:TLSBinaryClone}, when compared to the naive relaxation~\eqref{eq:naiveRelaxation}.

To prove the first claim, we first show that the additional constraints in the last two lines of~\eqref{eq:SDPrelax} are 
\emph{redundant} for~\eqref{eq:qcqpZ}, \ie, they are trivially satisfied by any feasible solution of~\eqref{eq:qcqpZ}.
Towards this goal  
%
we note that eq.~\eqref{eq:qcqpZ} is equivalent to~\eqref{eq:TLSBinaryClone}, where $\vxx$ is a column vector stacking $N+1$ quaternions: $\vxx = [\vq\tran\ \vq_1\tran\ \dots \vq_N\tran]\tran$, and where each $\vq_i = \theta_i \vq,\theta_i \in \{\pm 1\},\forall i=1,\dots,N$. Therefore, we have:
\bea
{[\MZ]_{qq_i} = \vq \vq_i\tran = \theta_i\vq \vq\tran = \vq_i \vq\tran = (\vq \vq_i\tran)\tran = [\MZ]\tran_{qq_i} } 
\nonumber \\
{ [\MZ]_{q_iq_j} = \vq_i \vq_j\tran = \theta_i\theta_j\vq \vq\tran = \vq_j \vq_i\tran = (\vq_i \vq_j\tran)\tran = [\MZ]\tran_{q_iq_j} }
\nonumber 
\eea
This proves that the constraints $[\MZ]_{qq_i} = [\MZ]\tran_{qq_i}$ and $[\MZ]_{q_iq_j} = [\MZ]\tran_{q_iq_j}$ are redundant for~\eqref{eq:qcqpZ}.
Therefore, problem~\eqref{eq:qcqpZ} is equivalent to:
\bea 
\label{eq:qcqpZredundant}
\min_{\MZ \succeq 0} & \trace{\MQ \MZ} \\
\subject 
& \trace{[\MZ]_{qq}} = 1 \nonumber \\
& [\MZ]_{q_i q_i} = [\MZ]_{qq}, \forall i=1,\dots,N \nonumber \\
& [\MZ]_{q q_i} = [\MZ]\tran_{q q_i}, \forall i=1,\dots,N \nonumber \\
& [\MZ]_{q_i q_j} = [\MZ]\tran_{q_i q_j}, \forall 1\leq i < j \leq N \nonumber \\
& \rank{\MZ} = 1 \nonumber
\eea 
where we added the redundant constraints as they do not alter the feasible set.
At this point, proving that~\eqref{eq:relaxationRedundant} is a convex relaxation of~\eqref{eq:qcqpZredundant} 
(and hence of~\eqref{eq:qcqpZ}) can be done with the same arguments of the proof of Proposition~\ref{prop:naiveRelax}:  
in~\eqref{eq:relaxationRedundant} we dropped the rank constraint (leading to a larger feasible set) and the remaining constraints 
are convex.

The proof of the second claim is straightforward.
 Since we added more constraints in~\eqref{eq:relaxationRedundant} compared to the naive relaxation~\eqref{eq:naiveRelaxation}, the optimal cost of~\eqref{eq:relaxationRedundant} always achieves a higher objective than~\eqref{eq:naiveRelaxation}, 
and since they are both relaxations, their objectives provide a lower bound 
  to the original non-convex problem~\eqref{eq:qcqpZ}.
\end{proof}

\section{Benchmark against \bnb}
\label{sec:benchmarkBnB}
\edit{We follow the same experimental setup as in Section~\ref{sec:exp_synthetic} of the main document, and benchmark \name against (i) \emph{Guaranteed Outlier Removal}~\cite{Bustos2015iccv-gore3D} (label: \GORE); (ii) \bnb with \emph{L-2 distance threshold}~\cite{bazin2012accv-globalRotSearch} (label: \BnBLtwo); and (iii) \bnb with \emph{angular distance threshold}~\cite{Bustos2015iccv-gore3D} (label: \BnBAng). Fig.~\ref{fig:benchmarkBnBGORE} boxplots the distribution of rotation errors for 30 Monte Carlo runs in different combinations of outlier rates and noise corruptions. In the case of low inlier noise ($\sigma=0.01$), \name is robust against $96\%$ outliers and achieves significantly better estimation accuracy compared to \GORE, \BnBLtwo and \BnBAng, all of which experience failures at $96\%$ outlier rates (Fig.~\ref{fig:benchmarkBnBGORE}(a,b)). In the case of high inlier noise ($\sigma=0.1$), \name is still robust against $80\%$ outlier rates and has lower estimation error compared to the other methods. }



\begin{figure}[b]
	\begin{center}
	\begin{minipage}{\linewidth}
	\begin{tabular}{c}%
			\begin{minipage}{\linewidth}%
			\centering%
			\includegraphics[width=0.9\columnwidth]{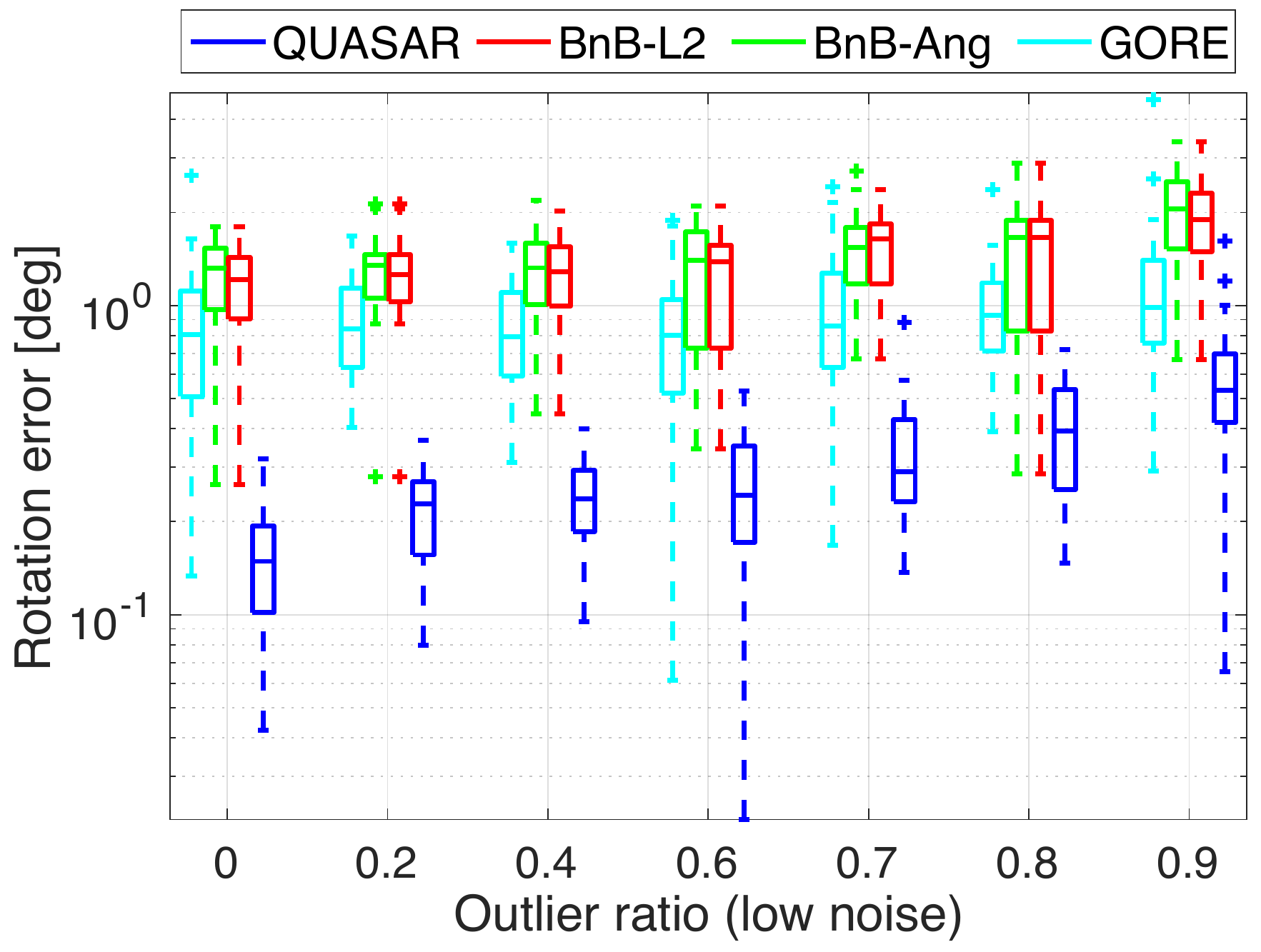} \\
			\end{minipage} \\
			\quad (a) $0-90\%$ outlier rates, low inlier noise \vspace{1mm} \\
			\begin{minipage}{\linewidth}%
			\centering%
			\includegraphics[width=0.9\columnwidth]{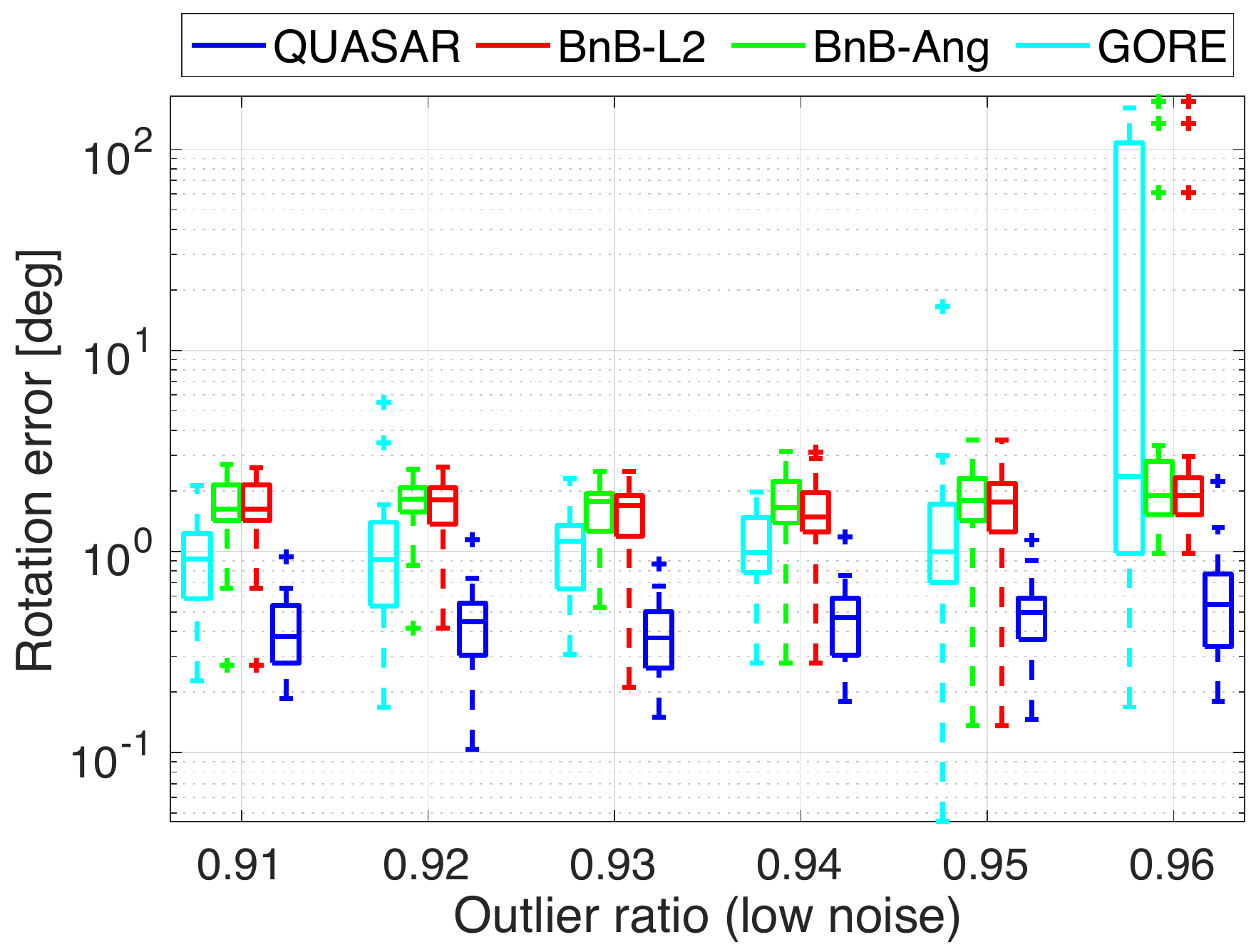} \\
			\end{minipage} \\
			\quad (b) $91-96\%$ outlier rates, low inlier noise \vspace{1mm} \\
			\begin{minipage}{\linewidth}%
			\centering
			\includegraphics[width=0.9\columnwidth]{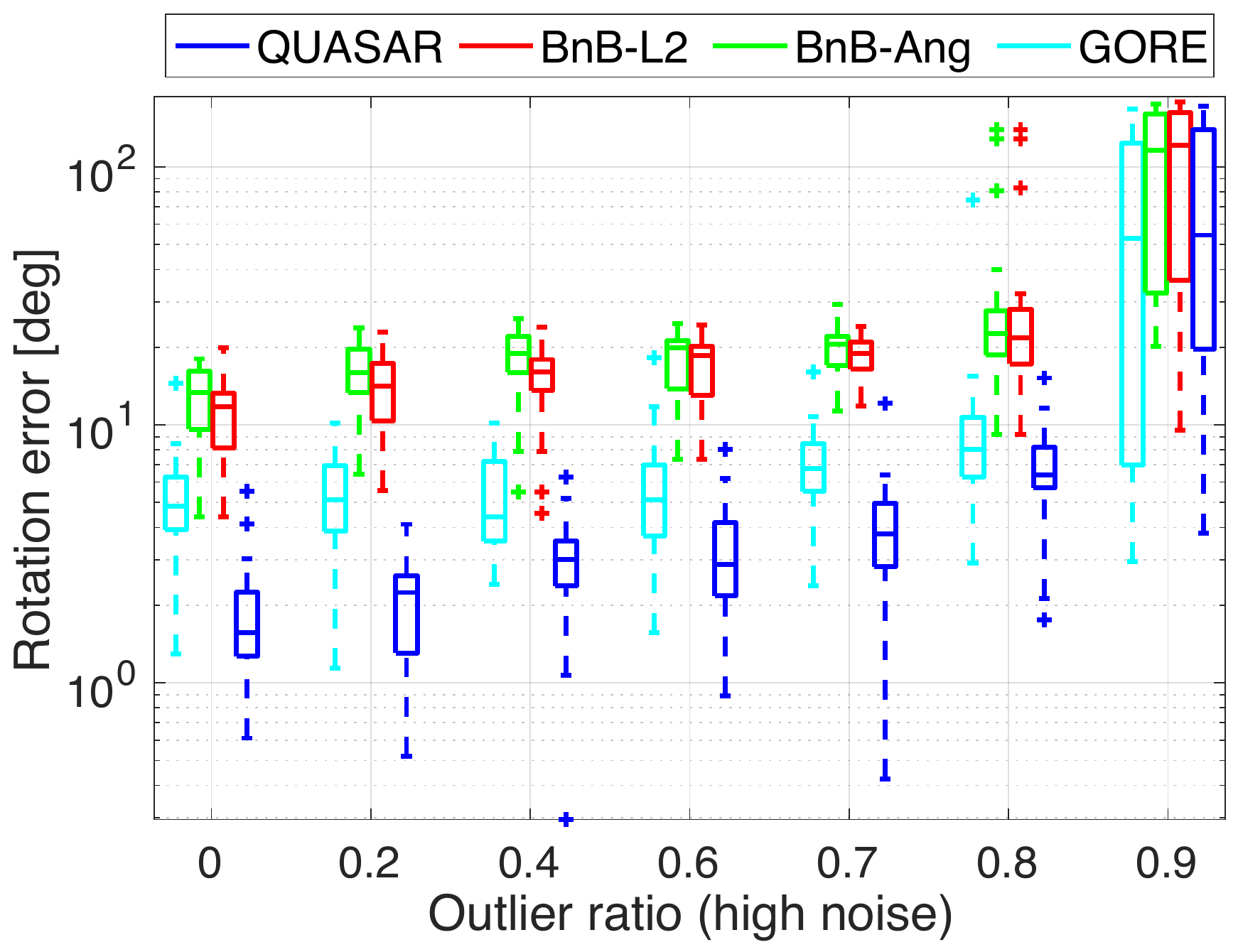} \\
			\end{minipage} \\
			\quad (c) $0-90\%$ outlier rates, high inlier noise  \\
		\end{tabular}
	\end{minipage}
	\vspace{1mm} 
	\caption{Rotation estimation error by \name, \GORE, \BnBLtwo and \BnBAng on (a) $0-90\%$ outlier rates with low inlier noise, (b) $91-96\%$ outlier rates with low inlier noise and (c) $0-90\%$ outlier rates with high inlier noise.}
	 \label{fig:benchmarkBnBGORE}
	\vspace{-6mm} 
	\end{center}
\end{figure}


\section{Image Stitching Results}
\label{sec:imageStitching}
Here we provide extra image stitching results. As mentioned in the main document, we use the \emph{Lunch Room} images from the PASSTA dataset~\cite{meneghetti2015scia-PASSATImageStitchingData}, which contains 72 images in total. We performed pairwise image stitching for 12 times, stitching image pairs $(6i+1,6i+7)$ for $i=0,\dots,11$ (when $i=11$, $6i+7=73$ is cycled to image 1). The reason for not doing image stitching between consecutive image pairs is to reduce the relative overlapping area
 so that \SURF~\cite{Bay06eccv} feature matching is more prone to output outliers, 
 creating a more challenging benchmark for \name.
 
 \name successfully performed all 12 image stitching tasks and Table~\ref{tab:imageStitchingStatistics} reports the statistics. 
As we mentioned in the main document, the stitching of image pair (7,13) was the most challenging, due to the high outlier ratio of $66\%$, and the \ransac-based stitching method~\cite{Torr00cviu,Hartley04book} as implemented by the Matlab ``\scenario{estimateGeometricTransform}'' function failed in that case. We show the failed example from \ransac in Fig.~\ref{fig:failedRansac}.


\begin{table}
\centering
\begin{tabular}{ccc}
& Mean & SD \\
\hline
\SURF Outlier Ratio & 14\% & 18.3\% \\
Relative Duality Gap & $1.40\mathrm{e}{-09}$ & $2.18\mathrm{e}{-09}$ \\
Rank & 1 & 0 \\
Stable Rank & $1+8.33\mathrm{e}{-17}$ & $2.96\mathrm{e}{-16}$
\end{tabular}
\caption{Image stitching statistics (mean and standard deviation (SD)) of \name on the \emph{Lunch Room} dataset~\cite{meneghetti2015scia-PASSATImageStitchingData}.}
\label{tab:imageStitchingStatistics}
\end{table}
\begin{figure}[t]
	\begin{center}
	\begin{tabular}{c}
	\begin{minipage}{\columnwidth}
			\centering%
			\includegraphics[width=\columnwidth]{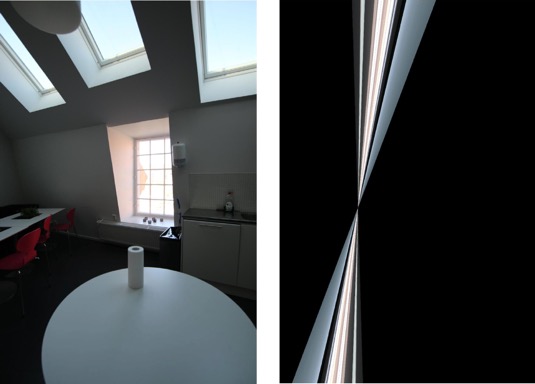}
	\end{minipage} \\
	\begin{minipage}{\columnwidth}
			\centering%
			\includegraphics[width=\columnwidth]{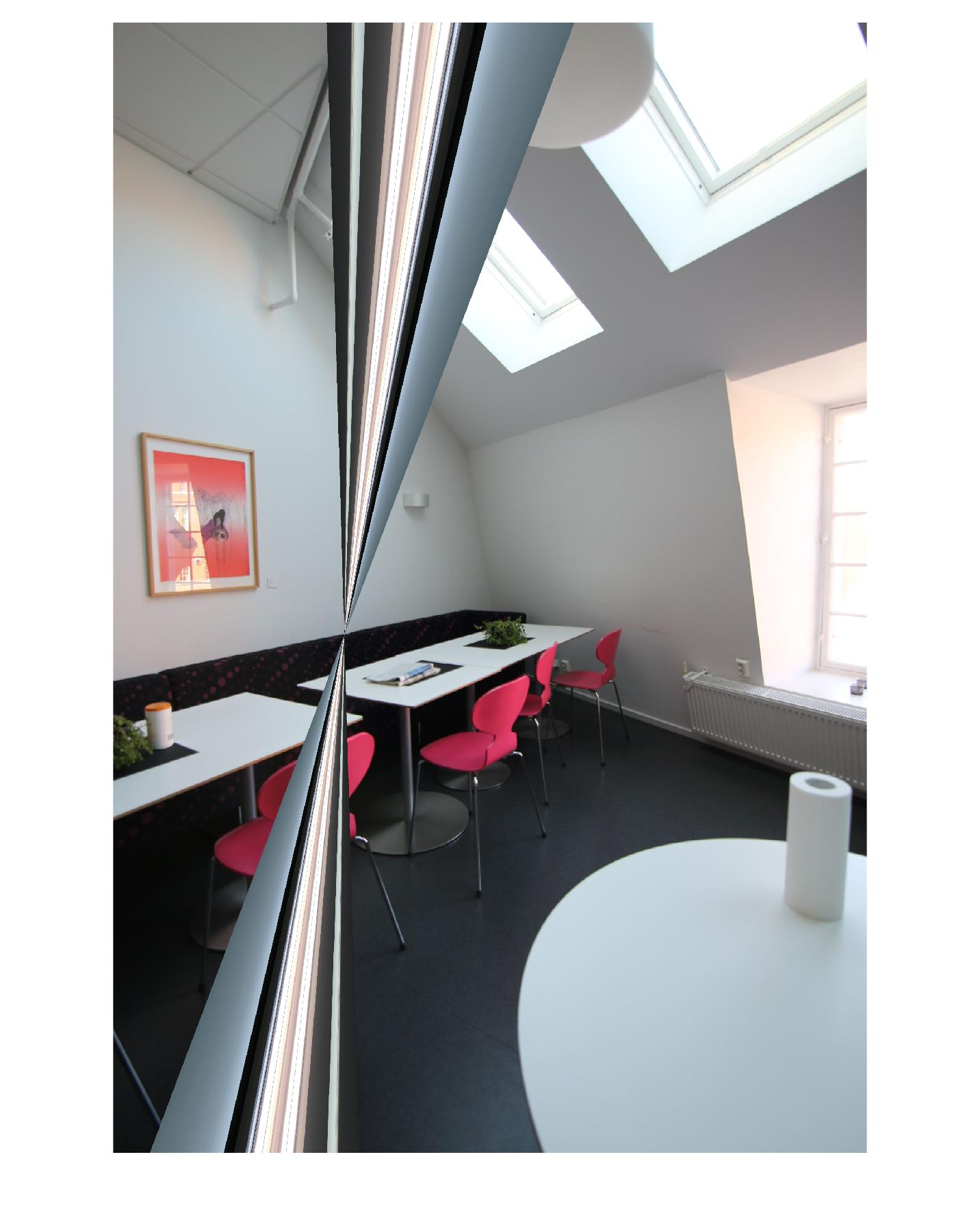}
	\end{minipage}
	 \end{tabular}
	\end{center}
	\caption{\ransac-based image stitching algorithm as implemented in Matlab failed when merging image pair (7,13) of the \emph{Lunch Room} dataset~\cite{meneghetti2015scia-PASSATImageStitchingData}. Top row, left: the original image 13; top row, right: image 13 after applying the (wrong) homography matrix estimated by Matlab's \ransac-based \scenario{estimateGeometricTransform} algorithm; bottom row: failed stitching of image 13 and image 7 due to the incorrect estimation of the  homography matrix.
	 \label{fig:failedRansac} }
	\vspace{0mm}
\end{figure}

{\small
\bibliographystyle{ieee_fullname}
\bibliography{iccv_refs_fullname}

\begin{thebibliography}{10}\itemsep=-1pt

\bibitem{Ahmed12tsp-wahba}
Shakil Ahmed, Eric~C. Kerrigan, and Imad~M. Jaimoukha.
\newblock A semidefinite relaxation-based algorithm for robust attitude
  estimation.
\newblock {\em IEEE Transactions on Signal Processing}, 60(8):3942--3952, 2012.

\bibitem{mosek}
MOSEK ApS.
\newblock {\em The MOSEK optimization toolbox for MATLAB manual. Version 8.1.},
  2017.

\bibitem{Arun87pami}
K.~Somani Arun, Thomas~S. Huang, and Steven~D. Blostein.
\newblock Least-squares fitting of two 3-{D} point sets.
\newblock {\em {IEEE} Trans. Pattern Anal. Machine Intell.}, 9(5):698--700,
  1987.

\bibitem{Bandeira16crm}
Afonso~S. Bandeira.
\newblock A note on probably certifiably correct algorithms.
\newblock {\em Comptes Rendus Mathematique}, 354(3):329--333, 2016.

\bibitem{Bay06eccv}
Herbert Bay, Tinne Tuytelaars, and Luc Van~Gool.
\newblock Surf: speeded up robust features.
\newblock In {\em European Conf. on Computer Vision (ECCV)}, 2006.

\bibitem{Bazin14eccv-robustRelRot}
Jean-Charles Bazin, Yongduek Seo, Richard Hartley, and Marc Pollefeys.
\newblock Globally optimal inlier set maximization with unknown rotation and
  focal length.
\newblock In {\em European Conf. on Computer Vision (ECCV)}, pages 803--817,
  2014.

\bibitem{bazin2012accv-globalRotSearch}
Jean-Charles Bazin, Yongduek Seo, and Marc Pollefeys.
\newblock Globally optimal consensus set maximization through rotation search.
\newblock In {\em Asian Conference on Computer Vision}, pages 539--551.
  Springer, 2012.

\bibitem{Besl92pami}
Paul~J. Besl and Neil~D. McKay.
\newblock A method for registration of {3-D} shapes.
\newblock {\em {IEEE} Trans. Pattern Anal. Machine Intell.}, 14(2), 1992.

\bibitem{Black96ijcv-unification}
Michael~J. Black and Anand Rangarajan.
\newblock On the unification of line processes, outlier rejection, and robust
  statistics with applications in early vision.
\newblock {\em Intl. J. of Computer Vision}, 19(1):57--91, 1996.

\bibitem{Blais95pami-registration}
G{\'e}rard Blais and Martin~D. Levine.
\newblock Registering multiview range data to create 3d computer objects.
\newblock {\em {IEEE} Trans. Pattern Anal. Machine Intell.}, 17(8):820--824,
  1995.

\bibitem{Boyd04book}
Stephen Boyd and Lieven Vandenberghe.
\newblock {\em Convex optimization}.
\newblock Cambridge University Press, 2004.

\bibitem{Boyd06notes}
Stephen Boyd and Lieven Vandenberghe.
\newblock Subgradients.
\newblock Notes for EE364b, 2006.

\bibitem{Breckenridge99tr-quaternions}
William~G. Breckenridge.
\newblock Quaternions - proposed standard conventions.
\newblock In {\em JPL, Tech. Rep. {INTEROFFICE MEMORANDUM IOM} 343-79-1199},
  1999.

\bibitem{Briales17cvpr-registration}
Jesus Briales and Javier Gonzalez-Jimenez.
\newblock {Convex Global 3D Registration with Lagrangian Duality}.
\newblock In {\em IEEE Conf. on Computer Vision and Pattern Recognition
  (CVPR)}, 2017.

\bibitem{Bustos18pami-GORE}
{\'A}lvaro~Parra Bustos and Tat-Jun Chin.
\newblock Guaranteed outlier removal for point cloud registration with
  correspondences.
\newblock {\em {IEEE} Trans. Pattern Anal. Machine Intell.}, 40(12):2868--2882,
  2018.

\bibitem{campbell2017iccv-2D3Dposeestimation}
Dylan Campbell, Lars Petersson, Laurent Kneip, and Hongdong Li.
\newblock Globally-optimal inlier set maximisation for simultaneous camera pose
  and feature correspondence.
\newblock In {\em Intl. Conf. on Computer Vision (ICCV)}, pages 1--10, 2017.

\bibitem{cheng2019aiaaScitech-totalLeastSquares}
Yang Cheng and John~L. Crassidis.
\newblock A total least-squares estimate for attitude determination.
\newblock In {\em AIAA Scitech 2019 Forum}, page 1176, 2019.

\bibitem{Chin16cvpr-outlierRejection}
Tat-Jun Chin, Yang Heng~Kee, Anders Eriksson, and Frank Neumann.
\newblock Guaranteed outlier removal with mixed integer linear programs.
\newblock In {\em IEEE Conf. on Computer Vision and Pattern Recognition
  (CVPR)}, pages 5858--5866, 2016.

\bibitem{chin2017slcv-maximumconsensusadvances}
Tat-Jun Chin and David Suter.
\newblock The maximum consensus problem: recent algorithmic advances.
\newblock {\em Synthesis Lectures on Computer Vision}, 7(2):1--194, 2017.

\bibitem{Chin2018arxiv-starTrackEvent}
Tat-Jun Chin and David Suter.
\newblock Star tracking using an event camera.
\newblock {\em ArXiv Preprint: 1812.02895}, 12 2018.

\bibitem{Choi15cvpr-robustReconstruction}
Sungjoon Choi, Qian-Yi Zhou, and Vladlen Koltun.
\newblock Robust reconstruction of indoor scenes.
\newblock In {\em Proceedings of the IEEE Conference on Computer Vision and
  Pattern Recognition}, pages 5556--5565, 2015.

\bibitem{Curless96siggraph}
Brian Curless and Marc Levoy.
\newblock A volumetric method for building complex models from range images.
\newblock In {\em SIGGRAPH}, pages 303--312, 1996.

\bibitem{enqvist2012eccv-robustfitting}
Olof Enqvist, Erik Ask, Fredrik Kahl, and Kalle {\AA}str{\"o}m.
\newblock Robust fitting for multiple view geometry.
\newblock In {\em European Conf. on Computer Vision (ECCV)}, pages 738--751.
  Springer, 2012.

\bibitem{Fischler81}
Martin~A. Fischler and Robert~C. Bolles.
\newblock Random sample consensus: a paradigm for model fitting with
  application to image analysis and automated cartography.
\newblock {\em Commun. ACM}, 24:381--395, 1981.

\bibitem{Forbes15jgcd-wahba}
James~Richard Forbes and Anton H.~J. de Ruiter.
\newblock {Linear-Matrix-Inequality-Based} solution to {W}ahba’s problem.
\newblock {\em Journal of Guidance, Control, and Dynamics}, 38(1):147--151,
  2015.

\bibitem{Gower05oss-procrustes}
John Gower and Garmt~B. Dijksterhuis.
\newblock Procrustes problems.
\newblock {\em Procrustes Problems, Oxford Statistical Science Series}, 30, 01
  2005.

\bibitem{CVXwebsite}
Michael Grant, Stephen Boyd, and Yinyu Ye.
\newblock {CVX}: Matlab software for disciplined convex programming, 2014.

\bibitem{Hartley04book}
Richard Hartley and Andrew Zisserman.
\newblock {\em Multiple View Geometry in Computer Vision}.
\newblock Cambridge University Press, second edition, 2004.

\bibitem{Hartley09ijcv-globalRotationRegistration}
Richard~I. Hartley and Fredrik Kahl.
\newblock Global optimization through rotation space search.
\newblock {\em Intl. J. of Computer Vision}, 82(1):64--79, 2009.

\bibitem{Horn87josa}
Berthold K.~P. Horn.
\newblock Closed-form solution of absolute orientation using unit quaternions.
\newblock {\em J. Opt. Soc. Amer.}, 4(4):629--642, 1987.

\bibitem{horn1988josa-rotmatsol}
Berthold K.~P. Horn, Hugh~M. Hilden, and Shahriar Negahdaripour.
\newblock Closed-form solution of absolute orientation using orthonormal
  matrices.
\newblock {\em J. Opt. Soc. Amer.}, 5(7):1127--1135, 1988.

\bibitem{Crassidis07-survey}
F.~Landis~Markley John L.~Crassidis and Yang Cheng.
\newblock Survey of nonlinear attitude estimation methods.
\newblock {\em Journal of Guidance, Control, and Dynamics}, 30(1):12--28, 2007.

\bibitem{Khoshelham16jprs}
Kourosh Khoshelham.
\newblock Closed-form solutions for estimating a rigid motion from plane
  correspondences extracted from point clouds.
\newblock {\em ISPRS Journal of Photogrammetry and Remote Sensing}, 114:78 --
  91, 2016.

\bibitem{Lajoie2019RAL-DCGM}
Pierre-Yves Lajoie, Siyi Hu, Giovanni Beltrame, and Luca Carlone.
\newblock Modeling perceptual aliasing in {SLAM} via {Discrete--Continuous
  Graphical Models}.
\newblock {\em {IEEE} Robotics and Automation Letters}, 4(2):1232--1239, 2019.

\bibitem{Lowe04ijcv}
David~G. Lowe.
\newblock Distinctive image features from scale-invariant keypoints.
\newblock {\em Intl. J. of Computer Vision}, 60(2):91--110, 2004.

\bibitem{MacTavish15crv-robustEstimation}
Kirk Mac~Tavish and Timothy~D. Barfoot.
\newblock At all costs: A comparison of robust cost functions for camera
  correspondence outliers.
\newblock In {\em Computer and Robot Vision (CRV), 2015 12th Conference on},
  pages 62--69. IEEE, 2015.

\bibitem{Makadia06cvpr-registration}
Ameesh Makadia, Alexander Patterson, and Kostas Daniilidis.
\newblock Fully automatic registration of 3d point clouds.
\newblock In {\em IEEE Conf. on Computer Vision and Pattern Recognition
  (CVPR)}, pages 1297--1304, 2006.

\bibitem{markley1988jas-svdAttitudeDeter}
F.~Landis Markley.
\newblock Attitude determination using vector observations and the singular
  value decomposition.
\newblock {\em The Journal of the Astronautical Sciences}, 36(3):245--258,
  1988.

\bibitem{markley2014book-fundamentalsAttitudeDetermine}
F.~Landis Markley and John~L. Crassidis.
\newblock {\em Fundamentals of spacecraft attitude determination and control},
  volume~33.
\newblock Springer, 2014.

\bibitem{Meer91ijcv-robustVision}
Peter Meer, Doron Mintz, Azriel Rosenfeld, and Dong~Yoon Kim.
\newblock Robust regression methods for computer vision: A review.
\newblock {\em Intl. J. of Computer Vision}, 6(1):59--70, 1991.

\bibitem{meneghetti2015scia-PASSATImageStitchingData}
Giulia Meneghetti, Martin Danelljan, Michael Felsberg, and Klas Nordberg.
\newblock Image alignment for panorama stitching in sparsely structured
  environments.
\newblock In {\em Scandinavian Conference on Image Analysis}, pages 428--439.
  Springer, 2015.

\bibitem{Milanese89chapter-ubb}
Mario Milanese.
\newblock Estimation and prediction in the presence of unknown but bounded
  uncertainty: A survey.
\newblock In {\em Robustness in Identification and Control}, pages 3--24.
  Springer US, Boston, MA, 1989.

\bibitem{olsson2008cvpr-polyRegOutlier}
Carl Olsson, Olof Enqvist, and Fredrik Kahl.
\newblock A polynomial-time bound for matching and registration with outliers.
\newblock In {\em IEEE Conf. on Computer Vision and Pattern Recognition
  (CVPR)}, pages 1--8, 2008.

\bibitem{Bustos2015iccv-gore3D}
{\'A}lvaro Parra~Bustos and Tat-Jun Chin.
\newblock Guaranteed outlier removal for rotation search.
\newblock In {\em Intl. Conf. on Computer Vision (ICCV)}, pages 2165--2173,
  2015.

\bibitem{Rosen18ijrr-sesync}
David~M. Rosen, Luca Carlone, Afonso~S. Bandeira, and John~J. Leonard.
\newblock {SE-Sync}: a certifiably correct algorithm for synchronization over
  the {Special Euclidean} group.
\newblock {\em Intl. J. of Robotics Research}, 2018.

\bibitem{rublee2011iccv-orb}
Ethan Rublee, Vincent Rabaud, Kurt Konolige, and Gary Bradski.
\newblock {ORB}: An efficient alternative to {SIFT} or {SURF}.
\newblock In {\em Intl. Conf. on Computer Vision (ICCV)}, pages 2564--2571,
  2011.

\bibitem{Rusu09icra-fast3Dkeypoints}
Radu~Bogdan Rusu, Nico Blodow, and Michael Beetz.
\newblock Fast point feature histograms ({FPFH}) for 3d registration.
\newblock In {\em IEEE Intl. Conf. on Robotics and Automation (ICRA)}, pages
  3212--3217, 2009.

\bibitem{Saunderson15siam}
James Saunderson, Pablo~A. Parrilo, and Alan~S. Willsky.
\newblock Semidefinite descriptions of the convex hull of rotation matrices.
\newblock {\em SIAM J. OPTIM.}, 25(3):1314--1343, 2015.

\bibitem{Schonemann66psycho}
Peter~H. Sch{\"o}nemann.
\newblock A generalized solution of the orthogonal procrustes problem.
\newblock {\em Psychometrika}, 31:1--10, 1966.

\bibitem{Shuster89jas}
Malcolm~D. Shuster.
\newblock Maximum likelihood estimation of spacecraft attitude.
\newblock {\em J. Astronautical Sci.}, 37(1):79--88, 01 1989.

\bibitem{Shuster93jas-attitude}
Malcolm~D. Shuster.
\newblock A survey of attitude representations.
\newblock {\em Journal of the Astronautical Sciences}, 41(4):439--517, 1993.

\bibitem{Torr00cviu}
Philip H.~S. Torr and Andrew Zisserman.
\newblock {MLESAC}: A new robust estimator with application to estimating image
  geometry.
\newblock {\em Comput. Vis. Image Underst.}, 78(1):138--156, 2000.

\bibitem{Tron15rssws3D-dualityPGO3D}
Roberto Tron, David~M. Rosen, and Luca Carlone.
\newblock On the inclusion of determinant constraints in lagrangian duality for
  {3D SLAM}.
\newblock In {\em Robotics: Science and Systems (RSS), Workshop ``The problem
  of mobile sensors: Setting future goals and indicators of progress for
  {SLAM}''}, 2015.

\bibitem{wahba1965siam-wahbaProblem}
Grace Wahba.
\newblock A least squares estimate of satellite attitude.
\newblock {\em SIAM review}, 7(3):409--409, 1965.

\bibitem{Yang2019rss-TEASER}
Heng Yang and Luca Carlone.
\newblock A polynomial-time solution for robust registration with extreme
  outlier rates.
\newblock In {\em Robotics: Science and Systems (RSS)}, 2019.

\bibitem{Yang16pami-goicp}
Jiaolong Yang, Hongdong Li, Dylan Campbell, and Yunde Jia.
\newblock {Go-ICP}: A globally optimal solution to {3D ICP} point-set
  registration.
\newblock {\em {IEEE} Trans. Pattern Anal. Machine Intell.}, 38(11):2241--2254,
  Nov. 2016.

\bibitem{Yang2015mpc-sdpnalplus}
Liuqin Yang, Defeng Sun, and Kim-Chuan Toh.
\newblock {SDPNAL+}: a majorized semismooth {Newton-CG} augmented lagrangian
  method for semidefinite programming with nonnegative constraints.
\newblock {\em Math. Program. Comput.}, 7(3):331--366, 2015.

\bibitem{Zhou16eccv-fastGlobalRegistration}
Qian-Yi Zhou, Jaesik Park, and Vladlen Koltun.
\newblock Fast global registration.
\newblock In {\em European Conf. on Computer Vision (ECCV)}, pages 766--782.
  Springer, 2016.

\end{thebibliography}
}

\end{document}